\documentclass[12pt,openany]{amsbook}
\usepackage{setspace}
\usepackage{fullpage}
\usepackage{amsmath}
\usepackage{amssymb}
\usepackage{amsbsy}
\usepackage{amscd}
\usepackage{eucal}

\usepackage{palatino}
\usepackage{mathpazo}

\newenvironment{artmatrix}{%
   \left[%
   \hskip -\arraycolsep
   \begin{array}{c|c*{20}{cc}}
   }{%
   \end{array}%
   \hskip -\arraycolsep
   \right]
}

\newcommand{\rank}{\operatorname{rank}}
\newcommand{\ceil}{\operatorname{ceil}}

\newtheorem{thm}{Theorem}[chapter]
\newtheorem{prop}[thm]{Proposition}
\newtheorem{lem}[thm]{Lemma}
\newtheorem{cor}[thm]{Corollary}
\newtheorem{deff}[thm]{Definition}

\newtheorem{examp}[thm]{Example}

\begin{document}
\frontmatter
\renewcommand{\thepage}{\roman{page}}
\thispagestyle{empty}
\begin{center}
\ \vspace{2cm}

{\large {\bf SOME NEW NON-UNIMODAL LEVEL ALGEBRAS
}}\\
\vspace{.5cm}
{\small A dissertation}\\
{\small submitted by}\\
\vspace{.5cm}
{\large Arthur Jay Weiss}\\
\vspace{.5cm}
{\small In partial fulfillment of the requirements}\\
{\small for the degree of}\\
\vspace{.5cm}
{\large Doctor of Philosophy}\\
{\small in}\\
{\large Mathematics}\\
\vspace{1cm}
{\large TUFTS UNIVERSITY}\\
\vspace{1cm}
February, 2007\\
\vspace{.5cm}
\copyright \; 2007, \; ARTHUR JAY WEISS\\
\vspace{.5cm}
ADVISOR: George McNinch
\end{center}
\pagebreak

\chapter*{Abstract}

In 2005, building on his own recent work and that of F. Zanello, A.
Iarrobino discovered some constructions that, he conjectured, would
yield level algebras with non-unimodal Hilbert functions.  This
thesis provides proofs of non-unimodality for Iarrobino's level
algebras, as well as for other level algebras that the author
has constructed along similar lines.\\

The key technical contribution is to extend some results published
by Iarrobino in 1984.  Iarrobino's results provide insight into some
naturally arising  vector subspaces of the vector space $R_d$ of
forms of fixed degree in a polynomial ring in several variables.  In
this thesis, the problem is approached by combinatorial methods and
results similar to Iarrobino's are proved for a different class of
vector subspaces of
$R_d$.\\

The combinatorial methods involve the definition of a new class of
matrices called \emph{L-Matrices}, which have useful properties that
are inherited by their submatrices.   A particular class of square
L-Matrices, associated with some specialized partially ordered sets
having interesting combinatorial properties, is identified. For this
class of L-Matrices, necessary and sufficient conditions are given
that they be nonsingular.\\

Several larger questions are discussed whose answers are
incrementally improved by the knowledge that the new non-unimodal
level algebras exist.

\chapter*{Acknowledgements}
This work was done under the supervision of Professor Anthony
Iarrobino of Northeastern University and Professor George McNinch of
Tufts University. The author wishes to express
his sincere appreciation.\\

The author wishes to thank Professor Fabrizio Zanello for reviewing
an early draft of this work and offering numerous insightful
comments and suggestions.\\

The author wishes to thank Professor Juan Migliore for providing a
preprint of \cite{GHMS1}.\\

The author wishes to thank Emeritus Professor George Leger and
Professor Montserrat Teixidor for their excellent advice throughout
the graduate student period.

\tableofcontents \mainmatter \setcounter{page}{1}
\renewcommand{\thepage}{\arabic{page}}

\ \vspace{5cm}
\begin{center}
\LARGE{ Some New Non-Unimodal Level Algebras}
\end{center}

\chapter{Introduction}

In this section, which is intended to provide an overview, we use
some technical terms without stopping to provide definitions.  The
definitions can all be found in later sections.\\

For over a century, mathematicians have been investigating Hilbert functions of
(standard) graded quotients of polynomial rings, and the subject is still a focus of
active study. In particular, among the graded quotients are the Gorenstein Artinian
graded algebras, which arise in various contexts. R. Stanley defined a
generalization of this class, the class of level algebras, that is useful for
studying Gorenstein Artinian graded algebras, but is also interesting in its own
right.\\

The general question for Hilbert functions of level algebras, that is, what
sequences could be their Hilbert functions, is the subject of the recent paper
\cite{GHMS1}, whose introduction provides an excellent history of work that has been
done in this direction to date. Most of that work proceeds in different directions
from what is done in this
thesis.\\

Here, we focus on a property called \emph{unimodality} that Hilbert functions of
level algebras sometimes have. In studying unimodality of level algebras, it is
usual to classify them by codimension and type; and one can ask whether it is
possible for a level algebra of some particular codimension and type to be
non-unimodal. The following is a summary of the history so far, for which the author
is indebted to A. Iarrobino, F. Zanello,
and \cite{BI1}.\\

In codimensions 1 and 2, level algebras of all types are necessarily
unimodal.  The level algebras of codimension 1 are sufficiently
simple that this is easy.  The investigations in codimension 2 were
performed by F. S. Macaulay in \cite{Mac1} and \cite{Mac3}, written
in the first several decades of the
twentieth century .\\

The next step was in showing that Gorenstein Artinian graded
algebras in codimension 3 are necessarily unimodal.  This was done
by R. Stanley in \cite{S1}, although it was D. Buchsbaum and D.
Eisenbud who first determined the actual Hilbert functions in
\cite{BE1}. In \cite{S2}, Stanley also demonstrated
a level algebra in codimension 13 that was not unimodal. \\

The next progress was accomplished in \cite{BI1} by D. Bernstein and
A. Iarrobino, who showed that a non-unimodal Gorenstein Artinian
algebra could be found in
codimension 5 and in any higher codimension.\\

Meanwhile, groundwork was being laid for further progress.   In particular, we note
the work of J. Emsalem and A. Iarrobino in \cite{EI1}, which contained some basic
concepts underlying the investigation of catalecticants by A. Iarrobino in
\cite{Iar1} and differently by R. Froberg and D. Laksov in \cite{FL1}.
Investigations of non-unimodality in Gorenstein Artinian graded algebras were
conducted in \cite{B1} by M. Boij, and by M. Boij and D. Laksov in
\cite{BL1}.\\

In 2005, F. Zanello published the first non-unimodal level algebra
in codimension 3 in \cite{Z1}.  Its type is 28.  Later that year, A.
Iarrobino used the same general idea to produce a level algebra in
codimension 3 of type 5 that, he conjectured, would prove to be
non-unimodal, as well as showing how to perform a similar
construction for any type higher than 5.  Iarrobino also suggested
methods for codimension 4 that, he conjectured, would produce
non-unimodal level algebras.  It is his construction in codimension
3, as well as some constructions in codimensions 3, 4, and 5 that
proceed along lines suggested by his work, that are analyzed in this
thesis, and shown to be non-unimodal.\\

\begin{table}[!h]
\begin{center}
\begin{tabular}{|c|c|c|c|c|c|c|}
  \hline
  $r$ $\backslash$ $t$ & 1 & 2 & 3 & 4 & 5 & \dots \\ \hline
  1 & yes & yes & yes & yes & yes & yes\\ \hline
  2 & yes & yes & yes & yes & yes & yes \\ \hline
  3 & yes & ? & ? & ? & no & no \\ \hline
  4 & ? & ? & no & no & no & no \\ \hline
  5 & no & no & no & no & no & no \\ \hline
  \vdots & no & no & no & no & no & no \\
  \hline
\end{tabular}
\vspace{\baselineskip} \caption{Is a Level Algebra of Codimension
$r$ and Type $t$ Necessarily Unimodal?}
\end{center}
\end{table}

As discussed in a later chapter, with a few additional observations
we will be able to summarize the current state of knowledge as
follows. Necessarily Unimodal: Codimensions 1 and 2 of all types,
codimension 3 of type 1. Non-unimodals exist: Codimension 3, of
types 5 and greater; codimension 4, of types 3 and greater;
codimension 5 and greater, of all types. Unknown: Codimension 3,
types 2, 3, and 4;
codimension 4, types 1 and 2.\\

Among the classes listed as unknown, some useful progress has been
made. In particular, we note \cite{IS1}.

\chapter{Algebraic Preliminaries}

\section{Level Algebras} We fix $k$, an algebraically closed field of
characteristic 0.  Throughout this work, it will be implicitly
assumed that all our vector spaces are over the field $k$.\\

Let $R$ be the polynomial ring over $k$ in $r$ variables: $R := k[X_1,...,X_r]$. $R$
can be written as a direct sum $R = \bigoplus_{d \ge 0}R_d$, where the subspaces
$R_d$ consist of all homogeneous polynomials (forms) in $R$ of degree $d$.  For
every $d$, $R_d$ is a finite-dimensional vector space, of dimension
$\genfrac(){0cm}{0}{d+r-1}{r-1}$.  One basis of $R_d$ consists of all monomials of
degree $d$.  By way of notation, let $D := (d_1,...,d_r)$ be any $r$-tuple of
non-negative integers such that $d_1 + ... + d_r = d$.  Then $D$ determines a
monomial $X^D := X_1^{d_1} \cdots X_r^{d_r}$ of degree $d$, and monomials of degree
$d$ are indexed by the $r$-tuples $D$.  $D$ is
sometimes called a \emph{multi-index} of \emph{dimension} $r$ and \emph{degree} $d$.\\

When considering the monomials $X^D$ of $R$, we sometimes use
\emph{lexicographic ordering}, defined as follows.  For two
different multi-indexes $C := (c_1,...,c_r)$ and $D :=
(d_1,...,d_r)$, we say \emph{$C$ comes before $D$} if, in the
leftmost co-ordinate for which $c_i \neq d_i$, $c_i > d_i$.  In this
case, we write $C > D$. By extension, we place an ordering on the
monomials of $R$: $X^C
> X^D \Leftrightarrow C > D$.  In this definition, there is no
requirement that $X^C$ and $X^D$ be monomials of the same degree.
When listing all monomials of a fixed degree $d$, to say that they
are listed \emph{lexicographically} means that the listing is
according to lexicographical order. For example, if $R =
k[X_1,X_2,X_3]$, we list the monomials of degree 2 lexicographically
as follows: $X_1^2, X_1X_2, X_1X_3, X_2^2,
X_2X_3, X_3^2$.\\

If $C := (c_1,...,c_r)$ and $D := (d_1,...,d_r)$ are two
multi-indexes, we define their addition and subtraction
co-ordinatewise.  That is, $C + D := (c_1 + d_1,...,c_r+ d_r)$, and
$C- D := (c_1 - d_1,...,c_r- d_r)$.\\

The direct-sum decomposition of $R = \bigoplus_{d \ge 0}R_d$ makes
it a graded $R$-module, graded by total degree, since for
non-negative integers $d$ and $e$, $R_dR_e \subseteq R_{d+e}$.\\

Let $I = \bigoplus_{d \ge 0}I_d \subseteq R$ be a homogeneous ideal
of $R$, where $I_d$ consists of all forms in $I$ of degree $d$. We
form the quotient ring $A := R/I$, which is a $k$-algebra.  The
direct-sum decomposition $A = \bigoplus_{d \ge 0}A_d$, where $A_d :=
R_d/I_d$, makes $A$ both a graded $k$-algebra and a graded
$R$-module.  In each case, the grading is by total degree.\\

In considering $R$ or its graded quotients by homogeneous ideals,
the only grading we will ever use is the grading by total degree,
which is sometimes called \emph{standard}.  From now on, standard
grading will always be
implicitly assumed.\\

For any graded quotient $A = R/I$, we say $A$ is \emph{Artinian} if
it is finite-dimensional as a vector space. In this case, we write
$A = \bigoplus_{0 \le d \le j}A_d$, where $j$ is the largest integer
for which $A_j$ is nonzero.  In writing such a direct sum
decomposition, we will always assume $A_j$ is nonzero unless
otherwise stated.\\

For an Artinian quotient $A = \bigoplus_{0 \le d \le j}A_d$, we
define soc($A$), the \emph{socle} of $A$, to be the annihilator of
the linear part of $A$:  soc($A) := \{a \in A | aA_1 = 0$\}. Soc(A)
is easily seen to be a homogeneous ideal of $A$.  We remark that
$A_j \subseteq$ soc($A$) since $A_jA_1 \subseteq A_{j+1} = 0$, but
equality need not hold.

\begin{examp} \label{socle_demo}  $A := k[X,Y]/(X^2,XY,Y^3) \simeq k \bigoplus
k\overline{X} \bigoplus k\overline{Y} \bigoplus k\overline{Y}^2$,
\end{examp}
\noindent where we adopt the usual notation that for any $F \in R,
\overline{F}$ denotes the homomorphic image of $F$ in $A = R/I$.
Then $A_j = A_2 = k\overline{Y}^2$, and soc($A$)= $k\overline{X}
\bigoplus k\overline{Y}^2$.  \\

An Artinian quotient $A = \bigoplus_{0 \le d \le j}A_d$ is said to
be \emph{level} if $A_j$ = soc($A$), and in this case we call $j$
the \emph{socle degree} of $A$, and we call $t := \dim_kA_j$ the
\emph{type} of $A$.  If the type $t = 1$, we say the level algebra
$A$ is \emph{Gorenstein}.\\

The following lemma provides an equivalent condition for a graded
Artinian quotient $A = R/I$ to be level.

\begin{lem} \label{level_char}
The graded Artinian quotient $R/I = A = \bigoplus_{0 \le d \le
j}A_d$ is level if and only if
\begin{equation} \label{level_imp}
\text{For all }d < j, F \in R_d - I_d \Rightarrow R_{j-d}F
\nsubseteq I_j.
\end{equation}
\end{lem}

\begin{proof}
Assume \eqref{level_imp} holds.  Let $\overline{F} \in A_d$ be
nonzero, where $F \in R_d - I_d$ and $d < j$.  We must show
$\overline{F} \notin$ soc($A$).  Reasoning by contradiction, if
$\overline{F} \in$ soc($A$), then $A_1\overline{F} = 0$, and $R_1F
\subseteq I_{d+1}$, so $R_{j-d-1}R_1F \subseteq I_j$. That is,
$R_{j-d}F \subseteq I_j$, contradicting
\eqref{level_imp}.\\

Conversely, assume that $A$ is level, and let $F \in R_d - I_d$ with
$d < j$.  To prove \eqref{level_imp}, it suffices to prove the
following statement, and then iterate $j-d$ times:

\begin{equation} \label{one_step}
\text{For all } d < j, F \in R_d - I_d \Rightarrow R_1F \nsubseteq
I_{d+1}.
\end{equation}

To prove \eqref{one_step}:  Since $\overline{F} \notin$ soc($A$),
there exists some $L \in R_1$ with $\overline{L}\overline{F} \neq
0$, that is $LF \notin I_{d+1}$. This shows $R_1F \nsubseteq
I_{d+1}$.
\end{proof}
\begin{cor} \label{I_deter}
If $A = R/I$ is a level algebra of socle degree $j$, then $I$ is
determined by $I_j$.  More precisely,
\begin{equation*}
\text{For }d < j, I_d = \{F \in R_d | R_{j-d}F \subseteq I_j\}.
\end{equation*}$\hspace{\fill} \square$
\end{cor}

\section{Polynomials as Differential Operators}

A good reference for the material in this section is \cite{IK1},
Appendix A.\\

Recalling that $k$ is an algebraically closed field of
characteristic 0 and $R := k[X_1,...,X_r]$ is a polynomial ring in
$r$ variables, we define $\mathcal{D}:= k[x_1,...x_r]$, an
isomorphic copy of $R$, where the variables $x_i$ are written in
lower case to distinguish them from the variables of $R$. To
distinguish elements of the two rings, we denote elements of $R$ by
uppercase letters $F,G,...$ and elements of $\mathcal{D}$ by
lowercase letters $f,g,...$.\\

We let $R$ operate on $\mathcal{D}$ according to the rule that, for
$f \in \mathcal{D}$,
\begin{equation*}
X_1 * f = \frac{\partial}{\partial x_1} f, \;\;\;\; X_2 * f =
\frac{\partial}{\partial x_2} f, \;\;\;\; X_1X_2 * f =
\frac{\partial}{\partial x_1} \frac{\partial}{\partial
  x_2}f,
\end{equation*}
and so on, extended by linearity.  We remark that this makes
$\mathcal{D}$ an $R$-module, with scalar multiplication $Ff := F*f$
for $F \in R, f \in \mathcal{D}$.  We say the elements of $R$
\emph{act on $\mathcal{D}$ as differential operators}.  If $F$ is a
homogeneous polynomial of degree $e$, we call $F*f$ an
\emph{$e^{th}$ partial
derivative of $f$}.\\

$\mathcal{D}$ can be written as a direct sum $\bigoplus_{d \ge
0}\mathcal{D}_d$, where the submodules $\mathcal{D}_d$ consist of
all homogeneous polynomials (forms) in $\mathcal{D}$ of degree $d$.
For any $d$, $\mathcal{D}_d$ is a finite-dimensional vector space,
of dimension $\genfrac(){0cm}{0}{d+r-1}{r-1}$. One basis consists of
all monomials of degree $d$.  Analogously with monomials in $R_d$,
we adopt the notation that an $r$-tuple $D := (d_1,...,d_r)$ of
non-negative integers such that $d_1 + ... + d_r = d$ determines the
monomial $x^D := x_1^{d_1} \cdots x_r^{d_r}$. \\

The direct-sum decomposition of $\mathcal{D} = \bigoplus_{d \ge
0}\mathcal{D}_d$ does not make $\mathcal{D}$ a graded $R$-module,
since it obeys a different grading rule:
\begin{equation*}
R_d*\mathcal{D}_e \subseteq \mathcal{D}_{d-e}.
\end{equation*}

However, if we fix a value of $d$ and consider what happens when $R_d$ operates on
$\mathcal{D}_d$, we can view $R_d$ as the dual vector space of $D_d$.  Specifically,
we take as basis of $\mathcal{D}_d$ the set of all monomials $x^D$, where as usual
$D := (d_1,...,d_r)$ and $d_1 + ... + d_r = d$.  Then, setting $D! := d_1! \cdots
d_r!$, we evaluate $X^D*x^D = D!$; and for any other monomial $x^{D'} \in
\mathcal{D}_d$, we evaluate $X^D*x^{D'} = 0$.  In other words,
$X^D/D!$ is the dual vector to $x^D$.\\

Since $R_d$ is dual to $\mathcal{D}_d$, there is a perfect pairing
\begin{equation} \label{pairing}
 R_d \times \mathcal{D}_d \rightarrow k
\end{equation}
and in this context it makes sense to talk about perpendicular
spaces.  Specifically, if $V \subseteq R_d$ is a vector subspace,
$V^{\perp} := \{f \in \mathcal{D}_d | V*f = 0\}$; and if
$\mathcal{W} \subseteq \mathcal{D}_d$ is a vector subspace,
$\mathcal{W}^{\perp} := \{F \in R_d | F*\mathcal{W} = 0\}$.

\section{Matlis Duality}

The material in this section was first considered by F. S. Macaulay
in his work on \emph{inverse systems} in \cite{Mac2}.  For a more
recent
treatment, see \cite{E1} or \cite{G1}.\\

We use the structure described in the previous section to get an
alternative description of what it means for an Artinian graded
algebra $A = R/I$ to be a level algebra.  We begin with a
definition.\\

If $A = R/I$ is a level algebra of socle degree $j$, then we define
$\mathcal{W}_A := I_j^{\perp} \subseteq \mathcal{D}_j$, a vector
subspace. That is,
\begin{equation*}
\mathcal{W}_A := \{ f \in \mathcal{D}_j \; | \; I_j*f = 0\}.
\end{equation*}

We give another characterization of the vector space
$\mathcal{W}_A$.
\begin{lem} \label{WA_char}
$\mathcal{W}_A = \{ f \in \mathcal{D}_j \; | \; I*f = 0\}.$
\end{lem}
\begin{proof}
Assume $I_j*f = 0$, where $f \in \mathcal{D}_j$. We must show that
$I_d*f = 0$ for all $d$.  This is surely true when $d > j$, and it
is true when $d = j$ by hypothesis.  If $d < j$, we argue by
contradiction.  Suppose, for $F \in I_d$, $F*f = g \in
\mathcal{D}_{j-d}$ and $g \neq 0$.  Then, recalling that $R_{j-d}$
is dual to $\mathcal{D}_{j-d}$, there is at least one vector $G \in
R_{j-d}$ such that $G*g =1$.  Then $GF \in I_j$ and $GF*f \neq 0$, a
contradiction.
\end{proof}

Since $\mathcal{D}$ is an $R$-module and $\mathcal{W}_A \subseteq
\mathcal{D}$, it is permissible to consider $Ann_R(\mathcal{W}_A)$,
easily seen to be a homogeneous ideal of $R$.  In fact, this
construction just recovers $I$.

\begin{lem} \label{I_char}
Let $A = R/I$ be a level algebra.  Then $Ann_R(\mathcal{W}_A) = I$.
\end{lem}
\begin{proof}
Since $\mathcal{W}_A := I_j^{\perp}$, $I*\mathcal{W}_A = 0$, so $I
\subseteq Ann_R(\mathcal{W}_A)$. For the other direction, we must
show that, for all $d$,
$[Ann_R(\mathcal{W}_A)]_d \subseteq I_d$.\\

For $d > j$, $[Ann_R(\mathcal{W}_A)]_d = R_d = I_d$.

For $d = j$, $[Ann_R(\mathcal{W}_A)]_j = (I_j^{\perp})^{\perp} =
I_j$, the last equality being true because the pairing in
$\eqref{pairing}$ is perfect.

For $d < j$, we argue by contradiction.  Assume $F \in R_d - I_d$
and $F \in [Ann_R(\mathcal{W}_A)]_d$, so that $F*\mathcal{W}_A = 0.$
Then $R_{j-d}F*\mathcal{W}_A = 0$, and $R_{j-d}F \subseteq
[Ann_R(\mathcal{W}_A)]_j = I_j$. However, by Lemma \ref{level_char},
$R_{j-d}F \nsubseteq I_j$, a contradiction.
\end{proof}

We have defined $\mathcal{W}_A$ to be $I_j^{\perp}$.  We now wish to
characterize $I_d^{\perp}$ for values of $d < j$.

\begin{lem} \label{perp_char}
For $d < j, I_d^{\perp} = R_{j-d}*I_j^{\perp}$
\end{lem}
\begin{proof}
For any $F,G \in R$ and $f \in \mathcal{D}$, we have $FG*f =
F*(G*f)$.  In particular, letting $G$ range through $R_{j-d}$ and
$f$ range through $I_j^{\perp}$, we have
\begin{equation} \label {perp_eq}
\text{For any } F \in R_d, FR_{j-d}*I_j^{\perp} =
F*(R_{j-d}*I_j^{\perp}).
\end{equation}
We can equate the set of those $F \in R_d$ for which the left-hand
side of \eqref{perp_eq} equals 0 with the set for which the
right-hand side equals 0.\\

For the left-hand side, \begin{equation*} \{F \in R_d |
FR_{j-d}*I_j^{\perp} = 0\} = \{F \in R_d | FR_{j-d} \subseteq
(I_j^{\perp})^{\perp} = I_j\} = I_d,
\end{equation*} the last
equality being guaranteed by Corollary \ref{I_deter}. For the
right-hand side, \begin{equation*}\{ F \in R_d |
F*(R_{j-d}*I_j^{\perp})= 0\} =
(R_{j-d}*I_j^{\perp})^{\perp}.\end{equation*}

Thus $I_d = (R_{j-d}*I_j^{\perp})^{\perp}$, so $I_d^{\perp}
=R_{j-d}*I_j^{\perp}$.
\end{proof}
\begin{cor} \label{dim_char}
If $A = R/I$ is a level algebra, then
\begin{equation*} \label {dim_value}
\dim_k[R/I]_d = \dim_kR_{j-d}*I_j^{\perp}.
\end{equation*}
$\hspace{\fill} \square$
\end{cor}
So far we have shown that, given a level algebra $A = R/I$, we can
characterize $I$ as the annihilator (in $R$) of a vector subspace
$\mathcal{W}_A \subseteq \mathcal{D}_j$.  We next turn the question
around. Given an arbitrary vector subspace $\mathcal{W} \subseteq
\mathcal{D}_j$, we can define
\begin{equation}
I_{\mathcal{W}} := Ann_R(\mathcal{W}).
\end{equation}

It is easy to see that $I_{\mathcal{W}}$ is a homogeneous ideal.
However, is $R/I_{\mathcal{W}}$ a level algebra?

\begin{lem} \label{Matlis_bij}
Let $\mathcal{W} \subseteq \mathcal{D}_j$ be a vector subspace. Then
$A_{\mathcal{W}} := R/I_{\mathcal{W}} := R/Ann_R(\mathcal{W})$ is a
level algebra.
\end{lem}
\begin{proof}
First of all, $A_{\mathcal{W}}$ is Artinian because, for $d > j,
[I_{\mathcal{W}}]_d = R_d$.  To show $A_{\mathcal{W}}$ is level, by
Lemma \ref{level_char} it is enough to establish that, given $d < j$
and $F \in R_d - [I_{\mathcal{W}}]_d,$ we have $R_{j-d}F \nsubseteq
[I_{\mathcal{W}}]_j.$ Consider such an $F$.  Since $F \notin
I_{\mathcal{W}}$, there is some $f \in \mathcal{W}$ such that $F*f =
g \neq 0$.  $g$ is a nonzero element of  $D_{j-d}$, so there exists
at least one $G \in R_{j-d}$ such that $G*g = 1$.  Thus $GF*f =
G*(F*f)  \neq 0$, and $R_{j-d}F \nsubseteq [I_{\mathcal{W}}]_j$ as
required.
\end{proof}

We are now ready to state another characterization of level
algebras.

\begin{thm}{MATLIS DUALITY.} \label{matlis_duality}
Let $k$ be a field of characteristic 0 and let the elements of $R :=
k[X_1,...X_r]$ act on $\mathcal{D} := k[x_1,...,x_r]$ as
differential operators. Then the level quotients $R/I$ of socle
degree $j$ are in bijection with the nonzero vector subspaces
$\mathcal{W} \subseteq \mathcal{D}_j$. Specifically, given $I$, take
$\mathcal{W} = I_j^{\perp}$; and given $\mathcal{W}$, take $I =
Ann_R(\mathcal{W})$, which is the unique homogeneous ideal with $I_j
= \mathcal{W}^{\perp}$ such that $R/I$ is level.
\end{thm}
\begin{proof}
We first remark that if $\mathcal{W} = \{0\}$, then
$Ann_R(\mathcal{W})= R$ and $R/Ann_R(\mathcal{W})$ is the 0-ring,
which is not of socle degree $j$. This explains the stipulation that
$\mathcal{W}$ be
nonzero.\\

We next remark that $[Ann_R(\mathcal{W})]_j = \{F \in R_j |
F*\mathcal{W} = 0\} = \mathcal{W}^{\perp}$.  By Lemma
\ref{Matlis_bij}, $R/Ann_R(\mathcal{W})$ is level, and by Corollary
\ref{I_deter}, $R/Ann_R(\mathcal{W})$ is the only
level quotient $R/I$ such that $I_j = \mathcal{W}^{\perp}$.\\

We define the maps $\alpha(I) := I_j^{\perp}$ and
$\beta(\mathcal{W}) := Ann_R(\mathcal{W})$.  We must show that
$\beta\alpha(I) = I$ for any homogeneous ideal $I$ such that $R/I$
is level of socle degree $j$, and $\alpha\beta(\mathcal{W}) =
\mathcal{W}$ for any nonzero vector subspace $\mathcal{W} \subseteq
\mathcal{D}_j$.  We have
\begin{equation*}
\beta\alpha(I) = \beta(I_j^{\perp})=Ann_R(I_j^{\perp})= I,
\end{equation*}
\noindent where the last equality follows from Lemma \ref{I_char}.
Also
\begin{equation*}
\alpha\beta(\mathcal{W}) =\alpha(Ann_R(\mathcal{W})) =
[Ann_R(\mathcal{W})]_j^{\perp} = \mathcal{W},
\end{equation*}
\noindent where the last equality follows because we showed that
$[Ann_R(\mathcal{W})]_j = \mathcal{W}^{\perp}$.
\end{proof}

\begin{lem} \label{hilb_char}
Let $\mathcal{W} \subseteq \mathcal{D}_j$ be a vector subspace. Then
\begin{equation*}
\dim_k{[A_{\mathcal{W}}]}_d = \dim_kR_{j-d}*\mathcal{W}.
\end{equation*}
\end{lem}
\begin{proof}
Set $I = Ann_R(\mathcal{W})$.  Then by Theorem \ref{matlis_duality},
$\mathcal{W} = I_j^{\perp}$.  We substitute theses values into the
formula of Corollary \ref{dim_char}, which is permitted because, by
Lemma \ref{Matlis_bij}, $A_{\mathcal{W}}$ is a level algebra.
\end{proof}

\chapter{Hilbert Functions}

\section{Definitions and Preliminaries}

As we have seen, level algebras are graded Artinian quotients of the
form $A := R/I$, where $I$ is a homogeneous ideal of $R :=
k[X_1,...,X_r]$. In considering a homogeneous ideal $I$, we will
always assume that $I$ contains no constant or linear polynomials as
elements; equivalently, $I_0 = 0$ and $I_1$ = 0.  The condition that
$I_0 = 0$ ensures that $A$ is nonzero; the condition that $I_1 = 0$
is equivalent to saying that $A$ is not isomorphic (as a graded ring
with standard grading) to any quotient of a polynomial ring with fewer
than $r$ variables.  With this understanding, we define the
\emph{codimension} of $A$ to be $r$, the number of variables.\\

For a graded Artinian quotient $A = \bigoplus_{0 \le d \le j}A_d$,
we define its \emph{Hilbert function} $h_A: \mathbb{Z}_{\ge0}
\rightarrow \mathbb{Z}_{\ge0}$ as follows: For non-negative integers
$d$, $h_A(d) := \dim_kA_d$.  When we form the level algebra
$A_{\mathcal{W}} := R/Ann_R(\mathcal{W})$, where $\mathcal{W}$ is a
vector subspace of $\mathcal{D}_j$, we may write $h_{\mathcal{W}}$
instead of $h_{A_{\mathcal{W}}}$. \\

Notationally, it is sometimes useful to express the Hilbert function as an
\emph{h-vector}, which is to say a $(j+1)$ - tuple of values taken on. In Example
\ref{socle_demo}, $h_A(0) = 1, h_A(1) = 2, h_A(2) = 1$.
As an h-vector, the Hilbert function is written (1,2,1).\\

The Hilbert function turns out to be a useful concept in algebraic
geometry.  This has been known for many years, but some recent
research has extended the applicability of the Hilbert function in
some new ways.   The details are beyond the scope of this thesis,
but we describe the concept, and refer the reader to \cite{BZ1},
\cite{Mig1}, \cite{Mig2}, and \cite{GHMS1} for
basic definitions and further details.\\

The concept is this:  one uses the Hilbert function of a graded Artinian quotient
$R/I$ and of related algebras to define certain properties of $R/I$, specifically
the \emph{Uniform Position Property} (UPP), \emph{Weak Lefschetz Property} (WLP),
\emph{Strong Lefschetz Property }(SLP), and \emph{Unimodality}.  If a projective
scheme is arthmetically Cohen-Macaulay, one can discover some of its geometric
properties by asking whether all Artinian reductions of its co-ordinate ring have
these properties.\\

We are interested in the last of these properties, unimodality, as
it applies to level algebras.  We say the Hilbert function $h_A$ of
a graded Artinian quotient $A = \bigoplus_{0 \le d \le j}A_d$ is
\emph{unimodal} if there is some degree $i$ such that $h_A$ is
nondecreasing for values of $d$ between $0$ and $i$ (inclusive), and
nonincreasing for values of $d$ between $i$ and $j$ (inclusive).
Otherwise, we say $h_A$ is \emph{non-unimodal}.  By extension, we
say $A$ itself is unimodal or non-unimodal.\\

If one works with level algebras for even a small amount of time,
either by hand or using a computer, it becomes immediately evident
that, in some sense, non-unimodal Hilbert functions are difficult to
find.  One might never find one at all, unless armed with some
particular strategy of construction. Based on this experience, we
pose the following questions:

\begin{enumerate}
\item  For $r = 1,2,3,...$, is every level algebra of codimension $r$
necessarily unimodal?
\item If not, what is the lowest possible type $t$ of a non-unimodal
level algebra of codimension $r$?
\item  For a given codimension $r$ for which type-$t$ non-unimodals
exist, what is the lowest possible socle degree $j$?
\end{enumerate}

For the first question, the answer is yes for $r =1$ or $2$, no for
$r \ge 3$.\\

For the second question, the most difficult cases are $r = 3$ and $r
= 4$.  This thesis describes non-unimodal level algebras in
codimension 3 of type 5 or more, and non-unimodal level algebras in
codimensions 4 and 5, of type 3 or more, and proves that they are in
fact non-unimodal.  The strategy used in constructing some of them
is due A. Iarrobino, who conjectured that they would turn out to be
non-unimodal; others involve minor variations on Iarrobino's
strategy of construction.  For a more detailed description of the
current state of play, please refer back to Chapter 1.\\

For the third question, very little is known, and the state of
knowledge appears to be too rudimentary to attempt a comprehensive
theory.  We will make a few observations, but prove no results, on
this subject in a later section.

\section{Splicing} \label{splic}

In trying to construct non-unimodal Hilbert functions, one strategy
is to build them up out of smaller pieces. We perform our
constructions in the polynomial rings $R$ and $\mathcal{D}$ with $r$
variables, and we focus on the $j^{th}$ graded piece $\mathcal{D}_j$
of $\mathcal{D}$.  Always, the level algebras constructed will turn
out to have codimension $r$ and socle degree $j$; so when we choose
values for $r$ and $j$ we will say we are \emph{fixing the
codimension} and \emph{fixing the socle degree}.\\

We start by fixing the codimension $r$ and the socle degree $j$. We
consider two vector subspaces $\mathcal{V},\mathcal{W} \subseteq
\mathcal{D}_j$, and for convenience we require that $\mathcal{V}
\cap \mathcal{W} = \{0\}$, so that $\mathcal{V} +\mathcal{ W} =
\mathcal{V} \bigoplus \mathcal{W}$, an internal direct sum.  If we
know the Hilbert functions $h_{\mathcal{V}}$ and $h_{\mathcal{W}}$
of the level algebras $A_\mathcal{V}$ and $A_\mathcal{W}$, it is
reasonable to hope that the Hilbert function $h_{\mathcal{V}
\bigoplus \mathcal{W}}$ of $A_{\mathcal{V} \bigoplus \mathcal{W}}$
will be related to $h_{\mathcal{V}}$ and $h_{\mathcal{W}}$. A first
step is provided by the following lemma.

\begin{lem} \label{splice_lemma}
Fix codimension $r$ and socle degree $j$. Let $\mathcal{V}$ and
$\mathcal{W}$ be two vector subspaces of $\mathcal{D}_j$ such that
that $\mathcal{V} \cap \mathcal{W} = \{0\}$.  Then for any $d$,
\begin{equation} h_{\mathcal{V} \bigoplus \mathcal{W}}(d) \le
h_{\mathcal{V}}(d) + h_{\mathcal{W}}(d),
\end{equation}
\noindent with equality if and only if $R_{j-d}*\mathcal{V} \cap
R_{j-d}*\mathcal{W} = \{0\}$.
\end{lem}
\begin{proof}
From Theorem \ref{matlis_duality} and Corollary \ref{dim_char},
\begin{enumerate}
\item[] $h_{\mathcal{V} \bigoplus \mathcal{W}}(d) = \dim_k
R_{j-d}*[\mathcal{V} \bigoplus \mathcal{W}]$.
\item[] $h_{\mathcal{V}}(d) = \dim_k
R_{j-d}*\mathcal{V}$.
\item[] $h_{\mathcal{W}}(d) = \dim_k
R_{j-d}*\mathcal{W}$.
\end{enumerate}
The lemma then follows from the observation that\\
$R_{j-d}*[\mathcal{V} \bigoplus \mathcal{W}] = R_{j-d}*\mathcal{V} +
R_{j-d}*\mathcal{W}$.
\end{proof}

\begin{examp} \label{deriv_demo}
$r = 3, R = k[X,Y,Z], \mathcal{D} = k[x,y,z], j = 6, \mathcal{V} =
\langle x^3y^3\rangle , \mathcal{W} = \langle x^3z^3\rangle $.
\end{examp}
For $d = 6,5,4$,  we have $R_{6-d}*\mathcal{V} \cap
R_{6-d}*\mathcal{W} = \{0\}$ since $y$ divides every element of
$R_{6-d}*\mathcal{V}$ but divides no nonzero element of
$R_{6-d}*\mathcal{W}$.

For $d = 3,2,1,0$, we have $R_{6-d}*\mathcal{V} \cap
R_{6-d}*\mathcal{W} \neq
\{0\}$, since $x^d$ is in the intersection.\\

One constraint on the dimension of $R_{j-d}*[\mathcal{V} \bigoplus
\mathcal{W}]$ is that it cannot exceed the dimension of
$\mathcal{D}_d$, of which it is a subspace:
\begin{equation} \label{h_constraint}
 h_{\mathcal{V} \bigoplus \mathcal{W}}(d) \le \dim_k \mathcal{D}_d.
\end{equation}
This places an immediate limitation on the choices of $\mathcal{V}$
and $\mathcal{W}$ that
give equality.\\

In \cite {Iar1}, A. Iarrobino proved a result of which the following is a special
case.  To state the result, we use the word \emph{general} in the sense of algebraic
geometry, that is, to say that a statement is true for general $f$ means that the
statement is true for $f$ lying in some dense Zariski-open subset of
$\mathcal{D}_j$, regarded as an affine variety . (See, for example, \cite{Sha1}.)
We will be more precise about this notion later on.\\

\begin{thm} {} \label{quoted}
With notation as above, let $\mathcal{V}$ be arbitrary and let
$\mathcal{W}:= \langle f\rangle  \subseteq \mathcal{D}_j$, the
one-dimensional subspace generated by the single element $f$. Then
for general $f \in \mathcal{D}_j$,
\begin{equation*}
h_{\mathcal{V} \bigoplus \mathcal{W}}(d) = min( h_{\mathcal{V}}(d) +
h_{\mathcal{W}}(d), h_{\mathcal{D}}(d)).
\end{equation*}
$\hspace{\fill} \square$
\end{thm}

In other words, for general $f \in \mathcal{D}_j$, $h_{\mathcal{V}
\bigoplus \mathcal{W}}(d)$ is as large as it could possibly be,
subject to \eqref{h_constraint}.\\

We will not rely on Theorem \ref{quoted} because we will sometimes
want to chose non-general $f \in \mathcal{D}_j$.  Instead, we will
prove similar-looking results, in contexts where $\mathcal{V}$,
rather than being arbitrary, is required to satisfy specified
conditions.\\

In order to use Theorem \ref{quoted} or anything similar, it is of course desirable
to know the Hilbert function $h_\mathcal{W}$. To this end, we quote another theorem
of Iarrobino from \cite{Iar1} (and others in \cite{FL1} and \cite{G78}).  For
details, see, for example,
\cite{IK1}.  \\

\begin{thm}{} \label{cat_max_rank}
With the notation above, let $\mathcal{W}:= \langle f\rangle
\subseteq \mathcal{D}_j$.  Then for general $f \in \mathcal{D}_j$
\begin{equation} \label{cat_max_eqn}
h_{\mathcal{W}}(d) = min( \dim_kR_{j-d}, \dim_k\mathcal{D}_d).
\end{equation}
$\hspace{\fill} \square$
\end{thm}

We will be proving this theorem (by different methods) and extending
it in a later chapter.  For now, to see better what is involved, we
work an explicit example.\\

We let $j=3$ and write $R = k[X,Y,Z]$, $\mathcal{D} = k[x,y,z]$.\\

Any $f \in \mathcal{D}_3$ can be written
\begin{equation} \label{gen_f}
f = Ax^3 + Bx^2y + Cx^2z + Dxy^2 + Exyz + Fxz^2 + Gy^3 + Hy^2z +
Iyz^2 + Jz^3,
\end{equation}

\noindent where $A, B,...,J \in k$ are the co-ordinates of $f \in
\mathcal{D}_3$, a 10-dimensional vector space of which the
monomials form a basis. \\

Setting $\mathcal{W} := \langle f\rangle $, let us compute
$h_{\mathcal{W}}(2)$, which is the same as the dimension of
$R_{j-d}*f = R_1*f =\langle X*f, Y*f, Z*f\rangle $. (Here we have
put $d =2$, so $j-d = 1$.) We compute $X*f, Y*f$, and $Z*f$.

\begin{equation*}
X*f = \frac{\partial f}{\partial x}= 3Ax^2 + 2Bxy + 2Cxz + Dy^2 +
Eyz + Fz^2.
\end{equation*}
\begin{equation*}
Y*f =
\frac{\partial f}{\partial y}= Bx^2 + 2Dxy + Exz + 3Gy^2 + 2Hyz +
Iz^2.
\end{equation*}
\begin{equation*} Z*f = \frac{\partial
f}{\partial z}= Cx^2 + Exy + 2Fxz + Hy^2 + 2Iyz +
3Jz^2.
\end{equation*}

Note that in writing the three partial derivatives, we have listed
the six monomials of degree 2 lexicographically across the page; and
we have listed the three monomials of degree 1 ($X,Y,$ and $Z$),
which determine which partial derivative is being taken,
lexicographically
down the page.\\

To determine the dimension of $R_1*f = \langle X*f, Y*f, Z*f\rangle
$, one must compute the rank of the 3 $\times$ 6 coefficient matrix

$$\left[
\begin{matrix}
3A&2B&2C&D&E&F\cr B&2D&E&3G&2H&I\cr C&E&2F&H&2I&3J
\end{matrix}
\right]$$

Of course, $3 = \dim_kR_{j-d}$, and $6 = \dim_k\mathcal{D}_d$; and
Theorem \ref{cat_max_rank} is saying that the coefficient matrix has
maximal rank for general $f \in \mathcal{D}_j$.\\

To generalize the context of Theorem \ref{cat_max_rank}, we consider
what might happen if $f$ were defined to be, not a member of the
whole space $\mathcal{D}_j$, but instead a member of some vector
subspace $\mathcal{W}_{\mathcal{M}}$ generated by monomials.  For
example, let $\mathcal{W}_{\mathcal{M}} := \langle xy^2, xyz, xz^2,
y^3, y^2z, yz^2, z^3\rangle $.  Then any member $f \in
\mathcal{W}_{\mathcal{M}}$ can be written
\begin{equation} \label{small_eq}
f = Dxy^2 + Exyz + Fxz^2 + Gy^3 + Hy^2z + Iyz^2 + Jz^3,
\end{equation}
\noindent where we have retained the same coefficient names for
purposes of comparison.  Then
\begin{equation*}
X*f = \frac{\partial f}{\partial x}= Dy^2 + Eyz + Fz^2.
\end{equation*}
\begin{equation*}
Y*f = \frac{\partial f}{\partial y}= 2Dxy + Exz + 3Gy^2 + 2Hyz +
Iz^2.
\end{equation*}
\begin{equation*}
Z*f = \frac{\partial f}{\partial z}= Exy + 2Fxz + Hy^2 + 2Iyz +
3Jz^2.
\end{equation*}

The matrix of coefficients is now

$$\left[
\begin{matrix}
0&0&D&E&F\cr 2D&E&3G&2H&I\cr E&2F&H&2I&3J
\end{matrix}
\right]$$

Alternatively, we could have observed that \eqref{small_eq} is
obtained from \eqref{gen_f} by substituting $A = 0, B = 0, C = 0$,
so the new matrix of coefficients is obtained from the old one by
making the same set of substitutions (and then deleting the
column that consists entirely of zeroes).\\

As a third example, we modify the previous example.  We retain the definition of
$\mathcal{W}_\mathcal{M}$, but this time we let $\mathcal{W} := \langle
f_1,f_2\rangle$ be generated by two vectors of $\mathcal{W}_\mathcal{M}$, denoted
\begin{equation*}
f_1 = D_1xy^2 + E_1xyz + F_1xz^2 + G_1y^3 + H_1y^2z + I_1yz^2 + J_1z^3
\end{equation*}
\noindent
and
\begin{equation*}
f_2 = D_2xy^2 + E_2xyz + F_2xz^2 + G_2y^3 + H_2y^2z + I_2yz^2 + J_2z^3.
\end{equation*}
\noindent Then, for $i = 1,2$, we have
\begin{equation*}
X*f_i = \frac{\partial f_i}{\partial x}= 3A_ix^2 + 2B_ixy + 2C_ixz + D_iy^2 + E_iyz
+ F_iz^2.
\end{equation*}
\begin{equation*}
Y*f_i = \frac{\partial f_i}{\partial y}= B_ix^2 + 2D_ixy + E_ixz + 3G_iy^2 + 2H_iyz
+ I_iz^2.
\end{equation*}
\begin{equation*} Z*f_i = \frac{\partial
f_i}{\partial z}= C_ix^2 + E_ixy + 2F_ixz + H_iy^2 + 2I_iyz + 3J_iz^2.
\end{equation*}
\noindent and the matrix of coefficients is now
$$\left[
\begin{matrix}
0&0&D_1&E_1&F_1\cr 0&0&D_2&E_2&F_2\cr 2D_1&E_1&3G_1&2H_1&I_1\cr
2D_2&E_2&3G_2&2H_2&I_2\cr E_1&2F_1&H_1&2I_1&3J_1\cr E_2&2F_2&H_2&2I_2&3J_2
\end{matrix}
\right]$$

We remark that adding another generator of $\mathcal{W}_\mathcal{M}$ created new
rows in
the matrix of coefficients without changing the number of columns.\\

Searching for a generalization of Theorem \ref{cat_max_rank}, it is logical to look
more closely at matrices of coefficients and ask whether they provide a means for
computing values of the Hilbert function of $A_{\mathcal{W}}$.

\chapter{L-Matrices}

\section{Definitions and Preliminaries}

Recall that $k$ is an algebraically closed field of characteristic
0. Let $B := k[z_1,...,z_n]$ be a polynomial ring. Then a matrix $U$
is called \emph{PV-matrix over $B$} if every nonzero entry of $U$
has the form $\lambda z_i$, where $\lambda$ is a positive integer
and $1 \le i \le n$.\\

A variable $z_i$ in a PV-matrix $U = (u_{ij})$ over $k[z_1,...,z_n]$
is said to \emph{move to the left} in $U$ if, whenever the variable
$z_i$ appears in both $u_{i_1j_1}$ and $u_{i_2j_2}$,  then $i_1 <
i_2 \Leftrightarrow j_1 > j_2$. In more precise language: if
$u_{i_1j_1} = \lambda_1z_i$ and $u_{i_2j_2} = \lambda_2z_i$ with
$\lambda_1$ and $\lambda_2$ positive integers, then
$i_1 < i_2 \Leftrightarrow j_1 > j_2$. \\

\begin{lem} \label{sub_mat}
Let $U$ be a PV-matrix in which the variable $z_i$ moves to the
left. Then $z_i$ does not appear twice in the same row of $U$ or
twice in the same column of $U$. If $z_i$ appears in two distinct
rows, its column in the lower row will be to the left of its column
in the higher row.
\end{lem}
\begin{proof}
These results follow directly from the definitions.
\end{proof}

A PV-matrix $U$ over $k[z_1,...,z_n]$ in which all variables
move to the left is called an \emph{L-matrix}.\\

For example, the three coefficient matrices worked as examples in the previous
chapter are PV-matrices over $k[A,B,C,...,J]$.  Since all variables move to the
left, they
are also L-matrices.\\

We will be proving several results about the ranks of PV-matrices.
The following lemma is crucial to their proofs.
\begin{lem} \label{pv_0}
Let $U$ be a PV-matrix over a polynomial ring $B$ and let $T$ be a
submatrix of $U$.  Then $T$ is a PV-Matrix over $B$.  Any variable
that moves to the left in $U$ moves to the left in $T$.  If $U$ is
an L-matrix, so is $T$.
\end{lem}
\begin{proof}
The definitions of PV-matrix and of variables moving to the left put
conditions on the entries of $U$, and it is immediate that $T$
inherits them from $U$.
\end{proof}

We remark that it is unusual for some useful property of a class of matrices to be
inherited by its submatrices.

\begin{lem} \label{pv_1}
Let $U$ be a square $s \times s$ PV-matrix over $k[z_1,...,z_n]$
with block decomposition
$$\left[
\begin{matrix}
A&*\cr *&Z\cr
\end{matrix}
\right]$$ where $A$ is a square $q \times q$ matrix with nonzero determinant (when
$q
> 0$), $Z$ is a square $r \times r$ matrix whose entries on the main diagonal
are all nonzero and all contain variables that move to the left in $U$, and the
entries of blocks marked $\lq\lq*"$ are not restricted. (To be precise, we assume $r
\ge 0, q \ge 0, s := q + r \ge 1$). Then the determinant $D(z_1,...,z_n)$ of $U$ is
a nonzero polynomial.
\end{lem}
\begin{proof}
For any non-negative integer $q$, we prove the lemma for matrices
$U$ for which $A$ is a $q \times q$ matrix.  Let $U = (u_{ij})$, an
$s \times s$ matrix. We perform induction on $s$. If $q = 0$, we
start the induction with $s = 1$, in which case $U$ has a single
nonzero entry, and the determinant must necessarily be nonzero. If
$q
> 0$, we start the induction with
$s$ = $q$, in which case $U$ = $A$, whose determinant is nonzero by hypothesis.\\

For the induction step, we assume the result proved for $A$ a $q
\times q$ matrix and $U$ an $(s-1) \times (s-1)$ matrix, and we
prove it for $A$ a $q
\times q$ matrix and $U$ an $s \times s$ matrix.\\

Let $\Sigma_s$ denote the symmetric group on $s$ letters, and recall
that
\begin{equation} \label{deter}
D(z_1,...,z_n) := \sum_{\sigma \in \Sigma_s} sgn(\sigma)\prod_{1 \le
i \le s}u_{i\sigma(i)}.
\end{equation}

\noindent where as usual sgn($\sigma$) is $+1$ if $\sigma$ is an
even permutation, $-1$ if $\sigma$ is an odd permutation.\\

Let $u_{ss} = \lambda z_k$.  Since the variable $z_k$ moves to the
left, we claim it can appear only in the entry $u_{ss}$ of $U$:
Suppose $z_k$ appears in $u_{ij}$.  If $i < s$, then $j > s$, which
is
impossible; if $j < s$, then $i > s$, which is again impossible.\\

If we wish to compute those terms of $D(z_1,...,z_n)$ in which $z_k$
appears, we must take, in \eqref{deter}, those $\sigma$ for which
$\sigma(s) = s$.  Collecting these terms together into a polynomial
$P(z_1,...,z_n)$, and letting $U'$ denote the $(s-1) \times (s-1)$
submatrix of $U$ formed by the first $(s-1)$ rows and the first
$(s-1)$ columns, we have
\begin{align*}
P(z_1,...,z_n) =& u_{ss}\sum_{\sigma \in \Sigma_{s-1}}
sgn(\sigma)\prod_{1 \le i \le {s-1}}u_{i\sigma(i)}\\
=&u_{ss}\det(U').
\end{align*}
Det$(U')$ is nonzero by induction and $u_{ss}$ by hypothesis. This
shows that that  $P(z_1,...,z_n)$, and hence $D(z_1,...,z_n)$, is
nonzero.
\end{proof}

\begin{cor} \label{pv_2}
Let $U$ be a square PV-matrix over $k[z_1,...,z_n]$ with block
decomposition
$$\left[
\begin{matrix}
A&*\cr *&Z\cr
\end{matrix}
\right]$$ where $A$ is a square $q \times q$ matrix with nonzero determinant (when
$q > 0$), $Z$ is a square $r \times r$ matrix with nonzero entries on the main
diagonal, and the entries of blocks marked $\lq\lq*"$ are not restricted.  We assume
$q + r \ge 1$. If $U$ is either (a) an L-matrix or (b) the result of permuting the
first $q$ rows and the first $q$ columns of an L-matrix, then the determinant
$D(z_1,...,z_n)$ of $U$ is nonzero.
\end{cor}
\begin{proof}
For case (a), if $U$ is an L-matrix, the variables on the main
diagonal of $Z$ move to the left, and the result follows immediately
from Lemma \ref{pv_1}.  For case (b), if $U$ is the result of
permuting the first $q$ rows and the first $q$ columns of an
L-matrix $U'$, $U'$ is of the form
$$\left[
\begin{matrix}
A'&*\cr *&Z\cr
\end{matrix}
\right]$$ \noindent where $\det(A') = \pm\det(A)$, which is nonzero
by hypothesis. Thus, by part (a), $\det(U')$ is nonzero.  But
$\det(U) = \pm\det(U')$.
\end{proof}

\begin{cor} \label{pv_3}
Let $U$ be a square $s \times s$ PV-matrix of block form
$$\left[
\begin{matrix}
0&A'&*\cr B'&C'&*\cr *&*&Z'\cr
\end{matrix}
\right]$$ \\
\noindent where 0 denotes a block of zeroes, $A'$ is a square $q \times q$ matrix
with nonzero determinant,  $B'$ is a square $r \times r$ matrix with nonzero
determinant, $Z'$ is a square matrix with nonzero entries on the main diagonal, and
the entries of blocks marked $\lq\lq*"$ are not restricted. If $U$ is either (a) an
L-matrix, or (b) the result of permuting the first $q +r $ rows and the first $q +
r$ columns of an L-matrix, then $U$ has nonzero determinant.
\end{cor}
\begin{proof}
This follows from Corollary \ref{pv_2}, taking $A$ to be the
submatrix formed by joining the blocks $0, A', B'$ and $C'$.
\end{proof}

\section{PV-Matrices as Parameterized Families}

Let $U$ be a $q \times r$ PV-matrix over $k[z_1,...,z_n]$ and let $C = (c_1,...,c_n)
\in k^n$.  Substituting $c_i$ for each $z_i$ in $U$, we obtain the matrix $U(C)$, a
matrix with entries in $k$. It is therefore possible to view $U$ as a family $\{U(C)
| C \in k^n \}$ of matrices with entries in $k$.  We wish to translate the notion of
$U$ having the maximal possible rank min($q,r$) into a statement about the family
$\{U(C) | C \in k^n \}$.  We use the word \emph{general} in the sense of algebraic
geometry: regarding $k^n$ as an affine variety, a statement is true for general $C
\in k^n$ if it is true for all $C$ contained in a dense Zariski-open subset of
$k^n$.
\begin{lem} \label{general_max}
Let $U$ be a matrix with entries in $k[z_1,...,z_n]$ having maximal
rank. Then for general $C \in k^n$, $U(C)$ has maximal rank.
\end{lem}
\begin{proof}
We must show there is a dense Zariski-open subset of $k^n$ on which
$U(C)$ has maximal rank.  In fact, we will show that $V := \{C \in
k^n | U(C)$ has
maximal rank$\}$ is itself a dense Zariski-open set. \\

Since $k$ is algebraically closed, $k^n$ is an irreducible affine
variety.  Since $k^n$ is irreducible, any non-empty open subset is
dense.\\

Having maximal rank is equivalent to there being at least one
maximal square submatrix with nonzero determinant.  Let
$M_1,...,M_m$ be the finitely many maximal square submatrices of
$U$.  For $i = 1,...,m$, let $D_i \in k[z_1,...,z_n]$ be the
determinant of $M_i$ and let $V_i := \{C \in k^n | D_i(C) \neq 0\}$.
Then $V = \bigcup_{1 \le i \le m}V_i$, so it is enough to show that
each $V_i$ is Zariski-open and that at least one of them is nonempty
(hence dense).\\

$V_i$ is Zariski-open because it is the complement of the zero-set
of $D_i$, a polynomial in $k[z_1,...,z_n]$.  By hypothesis, at least
one of the $D_i$ is nonzero, say $D_{i_0}$.  Since $D_{i_0}$ is a
nonzero polynomial and $k$ is infinite, there is some $C \in k^n$
such that $D_{i_0}(C) \neq 0$.  That is, $C \in V_{i_0}$, and
$V_{i_0}$ is nonempty.
\end{proof}

Before leaving this section, we remind the reader of the rule for
combining two (and by iteration, finitely many) statements, each of
which is true
for general $C \in k^n$.\\

\begin{lem} \label{zar_comp}
Let $S_1$ and $S_2$ be two statements, each of which is true for
general $C \in k^n$.  Then, for general $C \in k^n$, $S_1$ and $S_2$
are simultaneously true.
\end{lem}
\begin{proof}
$S_1$ is true on some Zariski-open dense set $V_1$; $S_2$ is true on
some Zariski-open dense set $V_2$.  $V_1$ and $V_2$ being dense,
$V_1 \cap V_2$ is dense as well.  That is, $V_1 \cap V_2$ is a
Zariski-open dense set on which $S_1$ and $S_2$ are both true.
\end{proof}

\chapter{Combinatorial Preliminaries}

\section{Partially Ordered Sets}

A \emph{partially ordered set} or \emph{poset} is a set $S$ together
with a binary relation $\succeq$ satisfying the following three
properties (See \cite{vLM1}.):
\begin{itemize}
  \item[(i)] (Reflexivity) For any $a \in S, a \succeq a$.
  \item[(ii)] (Transitivity) For any $a,b,c \in S$, if $a \succeq b$ and
  $b \succeq c$, then $a \succeq c$.
  \item[(iii)](Antisymmetry) For any $a,b \in S$, if $a \succeq b$
  and $b \succeq a$, then $a = b$.
\end{itemize}

If, for $a,b \in S$, $a \succeq b$ but $a \neq b$, we write $a \succ
b$.\\

In this work, the only partially ordered sets we will be considering
are finite and nonempty.  Whenever we use the phrase
\emph{partially ordered set}, we will mean a finite, nonempty partially
ordered set. \\

We say two partially ordered sets $S_1$ and $S_2$ are
\emph{isomorphic} if there exists a bijection $\beta: S_1
\rightarrow S_2$ such that, for all $a,b \in S_1$, $a \succeq b
\Leftrightarrow \beta(a) \succeq \beta(b)$.\\

Let $S$ be a partially ordered set and let $T \subseteq S$.  We say
$T$ is a \emph{co-ideal} or \emph{filter} or \emph{topset} if the
following condition is satisfied:
\begin{equation} \label{topset_def}
\text{If} \; a,b \in S, a \in T,\; and \;b \succeq a, \; then \; b
\in T.
\end{equation}

Similarly, let $S$ be a partially ordered set and let $B \subseteq
S$.We say $B$ is an \emph{ideal} or \emph{bottomset} if the
following condition is satisfied:
\begin{equation} \label{bottomset_def}
\text{If} \; a,b \in S, b \in B,\; and \;b \succeq a, \; then \; a
\in B.
\end{equation}

We collect some properties of topsets and bottomsets into a lemma
for future use.

\begin{lem} \label{topset_lem}
Let $S$ be a partially ordered set.
\begin{enumerate}
\item  $T$ is a topset of $S$
if and only if $B := S - T$ is a bottomset of $S$.
\item If $T_1$ and $T_2$ are topsets of
$S$, then $T_1 \cup T_2$ and $T_1 \cap T_2$ are topsets of $S$, and
$T_1 \cap T_2$ is a topset of $T_1$.
\item Let $X \subseteq S$ be any subset.  Then $T_X := \{a \in S |$ for some
$x \in X, a \succeq x\}$ is a topset of $S$ and $B_X := \{a \in S |$
for some $x \in X, x \succeq a\}$ is a bottomset of $S$.
\item Let $U$ be a topset of $S$ and let $C := S-U$.
\begin{enumerate}
\item[(a)] Let $T$ be a topset of $S$. Then $T \cap C$ is a topset of $C$.
\item[(b)] Let $W$ be a topset of $C$. Then
\begin{enumerate}
\item[(i)] $W = T_W \cap C$.
\item[(ii)] $\{$topsets of $C\}$ = $\{ T \cap C | T$ is a topset of
$S\}$.
\item[(iii)] $U \cup W$ is a topset of $S$.
\item[(iv)] Let $B$ be a bottomset of $U \cup W$.  Then $B \cap U$
is a bottomset of $U$, and $W - B$ is a topset of $C$.
\end{enumerate}
\item[(c)] Let $B$ be a bottomset of $S$.  Then
$U \cap B$ is a bottomset of $U$.
\item[(d)] Let $D$ be a bottomset of $U$. Then
\begin{enumerate}
\item[(i)]$D = B_D \cap U.$
\item[(ii)] $\{$bottomsets of $U\}$ = $\{B \cap U | B$ is a bottomset
of $S\}$.
\end{enumerate}
\end{enumerate}
\item Recalling that $S$ is finite by definition, let $\{a_1,...,a_n\}$ be the
finite set of all minimal elements of $S$.  For each $i = i,...,n$,
let $T_i$ be a topset of $S$ containing $a_i$.  Then $S = \bigcup_{1
\le i \le n} T_{i}$.

\end{enumerate}
\end{lem}
\begin{proof}
(1)  Assume $T$ is a topset.  Let $a,b \in S, b \in B, b \succeq a$.
We must show $a \in B$.  Since $T$ is a topset, $b \notin T
\Rightarrow a \notin T$.  That is, $a \in B$.

Assume $B$ is a bottomset.  Let $a,b \in S, a \in T, b \succeq a$.
We must show $b \in T$.  Since $B$ is a bottomset, $a \notin B
\Rightarrow b \notin B$.  That is, $b \in T$.\\

(2) Assume $b \succeq a$. For $i = 1,2$, $T_i$ is a topset, so if $a
\in T_i$ then $b \in T_i$.  That is, if $a$ is a member of both
$T_1$ and $T_2$, so is $b$; and if $a$ is a member of $T_1$ or
$T_2$, so is $b$.

Let $a \in T_1 \cap T_2, b \in T_1, b \succeq a$.  Then $b \in T_1
\cap T_2$ since $T_1 \cap T_2$ is a topset of $S$.\\

(3) Assume $a \in T_X$ and $b \succeq a$. We must show $b \in T_X$.
Since $a \in T_X$, $a \succeq x$ for some $x \in X$.  That is, $b
\succeq a \succeq x$, and by transitivity $b \succeq x$, thus $b \in
T_X$.

Assume $b \in B_X$ and $b \succeq a$. We must show $a \in B_X$.
Since $b \in B_X$, $x \succeq b$ for some $x \in X$.  That is, $x
\succeq b \succeq a$, and by transitivity $x \succeq a$, thus $a \in
B_X$.\\

(4)(a) Assume $a,b \in C, b \succeq a, a \in T \cap C$.  We must
show $b \in T \cap C$.  But $b \in T$ because $T$ is a topset, and
$b \in C$ by hypothesis.\\

(4)(b)(i) We first show $W \subseteq T_W \cap C$.  $W \subseteq T_W$
since, by reflexivity, $w \succeq w$ for all $w \in W$; and $W
\subseteq C$ by hypothesis.

For the other direction, let $a \in T_W \cap C$.  Since $a \in T_W$,
there is some $w \in W$ for which $a \succeq w$.  Since $a \in C$
and $W$ is a topset of $C$, this gives $a \in W$.\\

(4)(b)(ii) This follows from the previous results.  If $W$ is a
topset of $C$, then $W = T_W \cap C$ and $T_W$ is a topset of $S$.
If $T$ is a topset of $S$, then $T \cap C$ is a topset of $C$.\\

(4)(b)(iii) Let $a \in U \cup W, b \in S,$ and $b \succeq a$. We must show $b \in U$
or $b \in W$.  If $a \in U$, then since $U$ is a topset of $S$, $b \in U$. If $a \in
W$, then either $b \in C$, in which case $b \in W$, since $W$ is a topset of $C$; or
else $b
\notin C$, in which case $b \in S - C = U$.\\

(4)(b)(iv) For the first assertion, let $b \in B \cap U, a \in U, b
\succeq a$.  We must show $a \in B \cap U$, so it is enough to show
$a \in B$.  This follows from $B$ being a bottomset of $U \cup W$:
$a,b \in U \cup W, b \in B$, and $b \succeq a$.

For the second assertion, let $a \in W - B, b \in C, b \succeq a$.
We must show $b \in W - B$.  That $b \in W$ follows from $W$ being a
topset of $C$:  $a \in W, b \in C, b \succeq a$.  That $b \notin B$
follows from $B$ being a bottomset of $U \cup W$:  $a \in U \cup W,
b \succeq a, a \notin B$; if $b$ were an element of $B$, $a$ would
also have to be an element of $B$, which it is not.\\

(4)(c) Assume $a,b \in U, b \in U \cap B, b \succeq a$.  We must
show $a \in U \cap B$.  But $a \in B$ because $B$ is a bottomset,
and $a \in U$ by hypothesis.\\

(4)(d)(i) We first show $D \subseteq B_D \cap U$.  $D \subseteq B_D$
since, by reflexivity, $d \succeq d$ for all $d \in D$; and $D
\subseteq U$ by hypothesis.

For the other direction, let $a \in B_D \cap U$.  Since $a \in B_D$,
there is some $d \in D$ for which $d \succeq a$.  Since $a \in U$
and $D$ is a bottomset of $U$, this gives $a \in D$.\\

(4)(d)(ii) This follows from the previous results.  If $D$ is a
bottomset of $U$, then $D = B_D \cap U$ and $B_D$ is a bottomset of
$S$. If $B$ is a bottomset of $S$, then $B \cap U$ is a bottomset of
$U$.\\

(5) Let $a \in S$, which is a finite set.  We claim that there is a
minimal $a_i \in S$ such that $a \succeq a_i$.  Assuming the claim,
$a \in T_i$ because $T_i$ is a topset, and we are done.

To prove the claim:  If $a$ is not minimal, there is some $b_1 \in
S$ such that $a \succ b_1$; if $b_1$ is not minimal, there is some
$b_2 \in S$ such that $b_1 \succ b_2$. Continuing in this manner, we
get a chain  $a \succ b_1 \succ b_2 \succ ... \succ b_r$, which must
stop because $S$ is finite and (being a partially ordered set)
antisymmetric.
\end{proof}
\vspace{\baselineskip}

Let $S$ be a partially ordered set.  A real-valued function $\phi :
S \rightarrow \mathbb{R}$ is called \emph{order-preserving} if, for
$a,b \in S$, $\phi(a) \ge \phi(b)$ whenever $a \succeq b$.\\

We investigate the connection between order-preserving functions and
topsets.  We start by defining two properties that a partially
ordered set might or might not have.\\

We say that a partially ordered set $S$ has the \emph{Topset Positivity Property
(TPP)} if $\sum_{a \in T}\phi(a) \ge 0$ for any order-preserving function $\phi$
defined on $S$  such that $\sum_{a \in S}\phi(a) \ge 0$ and any topset $T \subseteq
S$. We say that $S$ has the \emph{Topset Average Property (TAP)} if for any
order-preserving function $\phi$ defined on $S$ and any nonempty topset $T \subseteq
S$, the average of the values of $\phi$ on $T$ is at least as large as the average
of the values of $\phi$ on $S$. In symbols, $\genfrac{}{}{}{0}{\sum_{a \in
T}\phi(a)}{\#(T)} \ge
\genfrac{}{}{}{0}{\sum_{a \in S}\phi(a)}{\#(S)}$. \\

\begin{examp}
$S = \{a,b\}$, with neither $a \succeq b$ nor $b \succeq a$.
\end{examp}

$S$ has neither TPP nor TAP, as can be verified by considering the
order-preserving function $\phi$ with $\phi(a) = 3, \phi(b) = -1$,
and the topset $\{b\}$.\\

\begin{prop} \label{cond_eq}
A partially ordered set $S$ has TAP if and only if it has TPP.
\end{prop}
\begin{proof}
Assume $S$ has TAP and let $\phi$ be an order-preserving function on
$S$ such that $\sum_{a \in S}\phi(a) \ge 0$.  Let $T \subseteq S$ be
a nonempty topset.  By TAP, $\sum_{a \in T}\phi(a) \ge
\genfrac{}{}{}{0}{\#(T)}{\#(S)}\sum_{a \in S}\phi(a) \ge 0$.  And if
$T$ is empty,
$\sum_{a \in T}\phi(a) = 0.$\\

For the other direction, assume $S$ has TPP and let $\phi$ be an
order-preserving function on $S$. Let $C := \genfrac{}{}{}{0}
{\sum_{a \in S}\phi(a)}{\#(S)}$ and define a new order-preserving
function $\psi$ on $S$ by setting $\psi(a) = \phi(a) - C$ for all $a
\in S$. We observe that $\sum_{a \in S}\psi(a) = 0$.  For any topset
$T \subseteq S$, TPP gives $\sum_{a \in T}\psi(a) \ge 0$.  So
\begin{align*}
\genfrac{}{}{}{0}{\sum_{a \in T}\psi(a)}{\#(T)} &\ge 0 =
\genfrac{}{}{}{0}{\sum_{a \in S}\psi(a)}{\#(S)}\\
\Rightarrow \genfrac{}{}{}{0}{\sum_{a \in T}\phi(a)}{\#(T)}  -
\genfrac{}{}{}{0}{\sum_{a \in T}C}{\#(T)} &\ge
\genfrac{}{}{}{0}{\sum_{a \in S}\phi(a)}{\#(S)}-
\genfrac{}{}{}{0}{\sum_{a \in S}C}{\#(S)}\\
\Rightarrow \genfrac{}{}{}{0}{\sum_{a \in T}\phi(a)}{\#(T)}  - C
&\ge
\genfrac{}{}{}{0}{\sum_{a \in S}\phi(a)}{\#(S)}-C\\
\Rightarrow \genfrac{}{}{}{0}{\sum_{a \in T}\phi(a)}{\#(T)}&\ge
\genfrac{}{}{}{0}{\sum_{a \in S}\phi(a)}{\#(S)}.
\end{align*}
\end{proof}

\section{The Partially Ordered Set $\mathcal{G}_Q$}

Let $Q = (Q_1,...,Q_n)$ be an $n$-tuple of non-negative integers.
We define the set $\mathcal{G}_Q$ as follows:
\begin{equation}
\mathcal{G}_Q := \{\text{n-tuples} \; I:= (I_1,...,I_n) \;| \; 0 \le I_i \le Q_i \;
for \; i = 1,...,n \}.
\end{equation}

\noindent We call $n$ the \emph{dimension} of $Q$ or of
$\mathcal{G}_Q$.\\

An element of $\mathcal{G}_Q$ will often be called a
\emph{multi-index}. The multi-index $(0,...,0)$ consisting of all
zeroes comes up frequently, and we sometimes refer to it as $0$.  We
give $\mathcal{G}_Q$ a partial ordering as follows:
\begin{equation}
I \succeq J \Leftrightarrow I_i \le J_i \; for \; i =1,...,n.
\end{equation}

For example, for any $I \in \mathcal{G}_Q, 0 \succeq I \succeq Q$.
We note that the partial ordering defined here is not at all the
same as lexicographic ordering, which is also defined on sets of
$n$-tuples.\\

As a first step, we consider some linearly ordered subsets of
$\mathcal{G}_Q$. Specifically, for some co-ordinate $j$, we fix the
values of all co-ordinates of $I = (I_1,...,I_n)$ except the
$j^{th}$ to be $I_i = \lambda_i$. We say that the $Q_j + 1$-element
set
\begin{equation}
X_{\lambda}:=
\{(\lambda_1,...,\lambda_{j-1},I_j,\lambda_{j+1},...\lambda_n) \;
|\; 0 \le I_j \le Q_j  \}
\end{equation}
is a \emph{one-parameter subset} of $\mathcal{G}_Q$.

\begin{lem} \label{one_parameter_lemma}
Let $\phi$ be an order-preserving function on $\mathcal{G}_Q$, let
$X$ be a one-parameter subset of $\mathcal{G}_Q$, and let $T$ be a
nonempty topset of $X$ (under the partial order inherited from
$\mathcal{G}_Q$). Then the average value of $\phi$ on $T$ is at
least as great as the average value of $\phi$ on $X$. In symbols,
$\sum_{I \in T} \phi(I) / \#(T) \ge \sum_{I \in X} \phi(I) / \#(X)$.
\end{lem}
\begin{proof}
This is an immediate consequence of the definition of
order-preserving: restricting from $X$ to $T$ removes the smallest
values of $\phi(I)$.
\end{proof}

We next prove a computational lemma.

\begin{lem} \label{ugly_sum}
Let $\phi$ be an order-preserving function defined on
$\mathcal{G}_Q$, let $P = (P_1,...,P_n) \in \mathcal{G}_Q$, and let
$0 \le j \le n$. Consider the two sums:
\begin{equation*}
\mathcal{S}_1 = \sum_{I_n \le Q_n} ... \sum_{I_{j+1} \le Q_{j+1}}
\sum_{I_{j} \le Q_{j}} \sum_{I_{j-1} \le P_{j-1}} ... \sum_{I_1 \le
P_1} \phi(I)
\end{equation*}
and
\begin{equation*}
\mathcal{S}_2 = \sum_{I_n \le Q_n} ... \sum_{I_{j+1} \le Q_{j+1}}
\sum_{I_{j} \le P_{j}} \sum_{I_{j-1} \le P_{j-1}} ... \sum_{I_1 \le
P_1} \phi(I).
\end{equation*}

Then there is a positive constant $\mu$ such that $\mu \mathcal{S}_2
\ge \mathcal{S}_1.$ In particular, if $\mathcal{S}_1 \ge 0$, then
$\mathcal{S}_2 \ge 0.$
\end{lem}

\begin{proof}
We note that the two sums are taken over the same set of values of
$I_1, ..., I_{j-1}$, and $I_{j+1}, ..., I_n$.  Only the values of
$I_j$ are different in the two sums.  The set of multi-indexes $I$
over which the first sum is taken can be subdivided into
one-parameter subsets $X_{\lambda}:=
\{(\lambda_1,...,\lambda_{j-1},I_j,\lambda_{j+1},...\lambda_n) \;
|\; 0 \le I_j \le Q_j  \}$, one for each choice of $\lambda
:=(\lambda_1,...\lambda_{j-1},\lambda_{j+1},...,\lambda_n)$; and
then the second sum can be subdivided into subsets $T_{\lambda}
\subseteq X_{\lambda}$, where $T_{\lambda}:=
\{(\lambda_1,...,\lambda_{j-1},I_j,\lambda_{j+1},...\lambda_n) \;
|\; 0 \le I_j \le P_j  \}$.  For each $\lambda$, we note that
$T_{\lambda}$ is a topset of $X_{\lambda}$ (under the partial order
inherited from $\mathcal{G}_Q$), that there are $Q_j +1$ elements in
$X_{\lambda}$, and that there are $P_j +1$ elements in
$T_{\lambda}$. By Lemma \ref{one_parameter_lemma}, for each
$\lambda$ we have
\begin{align*}
\genfrac{}{}{}{}{Q_j +1 }{P_j +1} \sum_{I \in T_{\lambda}} \phi(I)
&\ge \sum_{I \in X_{\lambda}} \phi(I),
\end{align*}
\noindent and summing over all values of $\lambda$,
\begin{equation*}
\genfrac{}{}{}{}{Q_j +1 }{P_j +1} \sum_{\lambda}\sum_{I \in
T_{\lambda}} \phi(I) \ge \sum_{\lambda}\sum_{I \in X_{\lambda}}
\phi(I)
\end{equation*} or
\begin{equation*}
\genfrac{}{}{}{}{Q_j +1 }{P_j +1}\mathcal{S}_2\ge\mathcal{S}_1.
\end{equation*}
The proof is completed by setting $\mu := \genfrac{}{}{}{0}{Q_j +1
}{P_j +1}$.
\end{proof}

Next we state a result about topsets of $\mathcal{G}_Q$ that have
the form $\mathcal{G}_P$, for some element $P = (P_1,...,P_n) \in
\mathcal{G}_Q$.

\begin{prop} \label{topset_prop}
Let $P = (P_1,...,P_n) \in \mathcal{G}_Q$ and let $\phi$ be an
order-preserving function on $\mathcal{G}_Q$ such that
\begin{equation*}
\sum_{I \in \mathcal{G}_Q}\phi(I) \ge 0.
\end{equation*}
Then
\begin{equation*}
\sum_{I \in \mathcal{G}_P}\phi(I) \ge 0.
\end{equation*}
\end{prop}
\begin{proof}
The first inequality can be rewritten
\begin{equation*}
\sum_{I_n \le Q_n} ... \sum_{I_1 \le Q_1 }\phi(I) \ge 0
\end{equation*}
and the second inequality can be written
\begin{equation*}
\sum_{I_n \le P_n} ... \sum_{I_1 \le P_1 }\phi(I) \ge 0.
\end{equation*}
The first is transformed into the second by $n$ iterations of Lemma
\ref{ugly_sum}.
\end{proof}

We next prove an extension of the previous proposition.\\

\begin{prop} \label{tpp}
The partially ordered set $\mathcal{G}_Q$ has TPP.  Equivalently,
for any topset $T \subseteq \mathcal{G}_Q$ and any order-preserving
function $\phi$ on $\mathcal{G}_Q$ such that
\begin{equation*}
\sum_{I \in \mathcal{G}_Q}\phi(I) \ge 0,
\end{equation*}
then
\begin{equation*}
\sum_{I \in T}\phi(I) \ge 0.
\end{equation*}
\end{prop}
\begin{proof}
We proceed by induction on $n$, the dimension of $\mathcal{G}_Q$. To
start the induction, we must show that TPP holds for $\mathcal{G}_Q$
of dimension 1, so we assume $Q := (Q_1)$. Any nonempty topset $T$
has the form $\{ (I_1) | 0 \le I_1 \le P_1) \} = \mathcal{G}_P$ for
some 1-tuple $P := (P_1)$.  So the dimension-1 case follows from
Proposition \ref{topset_prop}.\\

For the induction step, we assume the proposition proved for
$\mathcal{G}_Q$ of dimension $n-1$, and prove it for dimension $n$.
We first deal with the special case that $Q_n = 0$.  In this case,
$\mathcal{G}_Q$ = $\mathcal{G}_{(Q_1,...,Q_{n-1},0)}$ is isomorphic
to $\mathcal{G}_{(Q_1,...,Q_{n-1})}$ by the bijection
$(I_1,...,I_{n-1},0) \leftrightarrow (I_1,...,I_{n-1})$.  Since
$\mathcal{G}_{(Q_1,...,Q_{n-1})}$ has dimension $n-1$, the result
follows by the induction hypothesis.\\

To prove the result for arbitrary $Q$ of dimension $n$, assuming it
to be true for lower dimensions, we perform another induction. As
the induction step, we assume that the proposition is true for all
$P = (P_1,...,P_n)$ for which $P_1 \le Q_1$, ..., $P_n \le Q_n$, and
at least one inequality is strict; and we prove the proposition for
$(Q_1, ..., Q_n)$. To start the induction, we note that the result
is immediate for $Q = (0,...,0)$.  Also, we have already dealt with
the special case that $Q_n = 0$, so in proving the induction step we
are entitled to assume that $Q_n \ge 1$.\\

To prove the induction step, we assume that $\phi$ is an
order-preserving function on $\mathcal{G}_Q$ such that $\sum_{I \in
\mathcal{G}_Q}\phi(I) \ge 0$, and that $T \subseteq S$ is a topset.
Our goal is to show that $\sum_{I \in
T}\phi(I) \ge 0$.\\

We make several definitions.  Let $H_0 :=
\mathcal{G}_{(Q_0,...,Q_{n-1},0)} = \{(I_1,...,I_{n-1},0)\}
\subseteq \mathcal{G}_Q$.  Let $H_1 := \{(I_1,...,I_{n-1},1)\}
\subseteq \mathcal{G}_Q$.  We observe that $\#(H_0) = \#(H_1)$. We
observe that $\#(H_0 \cap T) \ge \#(H_1 \cap T)$ since if
$(I_1,...,I_{n-1},1)$ is
a member of the topset $T$ then $(I_1,...,I_{n-1},0)$ is also a member.\\

Let $\mathcal{G}' := \mathcal{G}_Q - H_0$. We observe that
$\mathcal{G}'$ is isomorphic to
$\mathcal{G}_{(Q_1,...,Q_{n-1},Q_n-1)}$ by the bijection
$(I_1,...,I_{n-1},I_n) \leftrightarrow (I_1,...,I_{n-1},I_n -1)$. We
let $T' := T - H_0$ and observe that
$T'$ is a topset of $\mathcal{G}'$ by Lemma \ref{topset_lem} (4)(a).\\

We let $C := \genfrac{}{}{}{0}{\sum_{I \in H_0} \phi(I)}{\#(H_0)}$,
and observe that $C \ge 0$ by Proposition \ref{topset_prop}.  We
define a new order-preserving function $\psi$ on $\mathcal{G}'$
according to the rule: if $I \in H_1, \psi(I) := \phi(I) + C$;
otherwise, $\psi(I) := \phi(I)$.  To verify that $\psi$ is
order-preserving, let $I,J \in \mathcal{G}'$ and let $I \succeq J$.
If $I \in H_1$, then $\psi(I) = \phi(I) + C \ge \phi(J) + C \ge
\psi(J)$.  If $I \notin H_1$, then $J \notin H_1$ (since $I_n > 1$
and $I \succeq J$) so $\psi(I) = \phi(I) \ge \phi(J) = \psi(J)$.  We
observe that
\begin{align*}
\sum_{I \in \mathcal{G}'}\psi(I) &= \sum_{I \in \mathcal{G}'}\phi(I)
+ \#(H_1)C\\
&= \sum_{I \in \mathcal{G}'}\phi(I) + \#(H_1)
\genfrac{}{}{}{0}{\sum_{I \in H_0} \phi(I)}{\#(H_0)}\\
&= \sum_{I \in \mathcal{G}'}\phi(I) + \sum_{I \in H_0} \phi(I)\\
&= \sum_{I \in \mathcal{G}_Q}\phi(I) \ge 0,
\end{align*}
the last inequality being true by hypothesis.\\

We now use the induction hypothesis, applied to $\mathcal{G}', T'$,
and $\psi$, which is applicable because of the previous computation.
We deduce that $0 \le \sum_{I \in T'}\psi(I)$.  We observe, for use
in the next paragraph:
\begin{equation} \label{n_p}
0 \le \sum_{I \in T'}\psi(I) = \sum_{I \in T'}\phi(I) + \#(T \cap
H_1)C = \sum_{I \in T'}\phi(I) + \#(T \cap
H_1)\genfrac{}{}{}{0}{\sum_{I \in H_0} \phi(I)}{\#(H_0)}.
\end{equation}

Recall that our goal is to show that $\sum_{I \in T}\phi(I) \ge 0$.
We have
\begin{align*}
\sum_{I \in T}\phi(I) &= \sum_{I
\in T'}\phi(I) + \sum_{I \in T \cap H_0}\phi(I)\\
&\ge \sum_{I \in T'}\phi(I) +\genfrac{}{}{}{0}{\#(T \cap
H_0)}{\#(H_0)}\sum_{I \in H_0}\phi(I),
\end{align*}
the last inequality being demonstrated as follows:
\begin{enumerate}
\item [(a)] By the special case, $H_0 =
\mathcal{G}_{(Q_0,...,Q_{n-1},0)}$ has TPP. \item[(b)] $T \cap H_0$
is a topset of $H_0$ by Lemma \ref{topset_lem}(2). \item[(c)] TPP
$\Rightarrow$ TAP by Proposition \ref{cond_eq}.
\end{enumerate}
Continuing the computation,
\begin{align*}
\sum_{I \in T}\phi(I) &\ge \sum_{I \in T'}\phi(I)
+\genfrac{}{}{}{0}{\#(T \cap H_0)}{\#(H_0)}\sum_{I \in H_0}\phi(I)\\
&= \sum_{I \in T'}\phi(I) + \#(T \cap H_0)\genfrac{}{}{}{0}{\sum_{I
\in H_0}\phi(I)}{\#(H_0)}\\
&\ge \sum_{I \in T'}\phi(I) + \#(T \cap
H_1)\genfrac{}{}{}{0}{\sum_{I \in H_0}\phi(I)}{\#(H_0)} \ge 0,
\end{align*}
the last inequality having been established as \eqref{n_p}.
\end{proof}

\section{Block L-Matrices Associated to $\mathcal{G}_Q$}

Recall that the elements of $\mathcal{G}_Q$ have two different
orderings on them.  There is lexicographic order, denoted $I \ge J$,
and the partial order, denoted $I \succeq J$.\\

We say that an L-Matrix $U$ \emph{has $\mathcal{G}_Q$ pattern} if
\begin{enumerate}
\item For each $I \in \mathcal{G}_Q$, there exist non-negative integers $r_I$
and $c_I$ such that $U$ has block form, with one $r_I \times c_J$
block $B_{IJ}$ corresponding to each ordered pair $(I,J)$ of
elements of $\mathcal{G}_Q$.
\item The block-row indices $I$ occur in lexicographic order.
\item The block-column indices $J$ occur in reverse lexicographic order.
\item All entries in the block $B_{IJ}$ are nonzero if $I \succeq
J$ and all entries are zero if $I \nsucceq J$.
\end{enumerate}

Recalling that in lexicographic order $Q := (Q_1,...,Q_n)$ comes
first and $0 := (0,...,0)$ comes last, an L-matrix with
$\mathcal{G}_Q$ pattern decomposes into blocks as follows:

$$\left[
\begin{matrix}
B_{Q0}&.&.&.&B_{QQ}\cr .&.&.&.&.\cr .&.&B_{IJ}&.&.\cr .&.&.&.&.\cr
B_{00}&.&.&.&B_{0Q}
\end{matrix}
\right]$$

Recall that the size of $B_{IJ}$ is $r_I \times c_J$. If we wish to
include this information along with the matrix, we will do it as
follows, with the understanding that the $r_I$'s and $c_J$'s are not
matrix entries:
\begin{equation*}
\begin{artmatrix}
&c_0&&c_J&&c_Q\cr
\hline r_Q&B_{Q0}&.&.&.&B_{QQ}\cr
&.&.&.&.&.\cr
r_I&.&.&B_{IJ}&.&.\cr &.&.&.&.&.\cr r_0&B_{00}&.&.&.&B_{0Q}
\end{artmatrix}
\end{equation*}

We will be interested in determining necessary and sufficient
conditions that that an L-Matrix with $\mathcal{G}_Q$ pattern have
nonzero determinant.\\

Let $U$ be an L-Matrix with $\mathcal{G}_Q$ pattern. We wish to
define the notion of a \emph{superblock} of $U$.  Let
$\mathcal{I},\mathcal{J} \subseteq \mathcal{G}_Q$ be subsets.  Then
the \emph{$\mathcal{I}\mathcal{J}$ superblock} of $U$ is the
submatrix composed of all blocks $B_{IJ}$ such that $I \in
\mathcal{I}$ and $J \in \mathcal{J}$.  A \emph{superblock of zeroes}
is a superblock of $U$ composed entirely of blocks of zeroes.  A
\emph{maximal superblock of zeroes} is a superblock of zeroes that
is not properly contained in any larger superblock of zeroes.

\begin{examp} \label{pattern}
Let $n = 2$, $Q = (Q_1,Q_2) = (1,1)$.  $\mathcal{G}_Q = \{(1,1),
(1,0), (0,1), (0,0)\}$.  Abusing notation, $\mathcal{G}_Q = \{11,
10, 01, 00\}$.
\end{examp}
Let us analyze an L-Matrix $U$ with $\mathcal{G}_Q$ pattern, for the
choice of $Q$ in Example \ref{pattern}. According to the definition
of $\mathcal{G}_Q$ pattern, the blocks $B_{IJ}$ are composed
entirely of zeroes if $I \nsucceq J$ and contain no zeroes if $I
\succeq J$. The pairs $(I,J)$ for which $I \nsucceq J$ are:
\begin{equation*}
(11,00), (11,01), (11,10), (10,00), (10,01), (01,00), (01,10).
\end{equation*} Therefore $U$ has the following block form, where an
asterisk denotes a block that contains only nonzero entries:

\begin{equation*}
\begin{artmatrix}
&c_{00}&c_{01}&c_{10}&c_{11}\cr
 \hline
r_{11}&0&0&0&*\cr r_{10}&0&0&*&*\cr r_{01}&0&*&0&*\cr r_{00}&*&*&*&*
\end{artmatrix}
\end{equation*}

The four maximal superblocks of zeroes can be demonstrated by
placing spaces into the diagram, as follows, and perhaps permuting
the rows and columns of $U$ (as was done to demonstrate the third
maximal superblock).

\begin{equation*}
\begin{artmatrix}
&c_{00}&c_{01}&c_{10}&&c_{11}\cr  \hline r_{11}&0&0&0&&*\cr \cr
r_{10}&0&0&*&&*\cr r_{01}&0&*&0&&*\cr r_{00}&*&*&*&&*
\end{artmatrix}
\end{equation*}

\begin{equation*}
\begin{artmatrix}
&c_{00}&c_{01}&&c_{10}&c_{11}\cr \hline r_{11}&0&0&&0&*\cr
r_{10}&0&0&&*&*\cr \cr r_{01}&0&*&&0&*\cr r_{00}&*&*&&*&*
\end{artmatrix}
\end{equation*}

\begin{equation*}
\begin{artmatrix}
&c_{00}&c_{10}&&c_{01}&c_{11}\cr \hline r_{11}&0&0&&0&*\cr
r_{01}&0&0&&*&*\cr \cr r_{10}&0&*&&0&*\cr r_{00}&*&*&&*&*
\end{artmatrix}
\end{equation*}

\begin{equation*}
\begin{artmatrix}
&c_{00}&&c_{01}&c_{10}&c_{11}\cr \hline r_{11}&0&&0&0&*\cr
r_{10}&0&&0&*&*\cr r_{01}&0&&*&0&*\cr \cr r_{00}&*&&*&*&*
\end{artmatrix}
\end{equation*}

\begin{lem} \label{gq_1}  Let $U$ be an L-Matrix with $\mathcal{G}_Q$
pattern. Then its maximal superblocks $Y_T$ of zeroes are determined by the proper
nonempty topsets $T \subseteq \mathcal{G}_Q$: the block $Y_T$ contains precisely
those blocks $B_{IJ}$ such that $I \notin T$ and $J \in T$.
\end{lem}
\begin{proof}  We first show that any such superblock contains only
zeroes.  By the definition of $\mathcal{G}_Q$ pattern, it is enough
to show that, for any block $B_{IJ}$ in the superblock, $I \nsucceq
J$.  But $J$ is an element of the topset $T$
and $I$ is not, so $I \nsucceq J$.\\

We next show that, given a superblock $Y$ of zeroes, it is a
subsuperblock of one of the $Y_T$'s.  Since all its constituent
blocks $B_{IJ}$ are 0, we always have $I \nsucceq J$. So if we
consider the set $T = \{K \in \mathcal{G}_Q | K \succeq J$ where
some $B_{IJ}$ is in $Y \}$, then $T$ is a topset by Lemma
\ref{topset_lem} (3), and $Y$ is a subsuperblock of $Y_T$.
\end{proof}

In the matrix $U$ with $\mathcal{G}_Q$ pattern discussed in connection with Example
\ref{pattern}, we see that the four maximal superblocks of zeroes correspond,
respectively, to the four nonempty proper subsets of $\mathcal{G}_Q$:
$\{00,01,10\},\ \{00,01\},
\{00,10\}, \{00\}$.\\

For an L-matrix $U$ with $\mathcal{G}_Q$ pattern, with block
dimensions $r_I$ and $c_I$ as above, we define, for each $I \in
\mathcal{G}_Q$, the \emph{excess} $A_I := r_I - c_I$.  Despite the
name, there is no requirement that $r_I \ge c_I$, and $A_I$ can
certainly be a negative number.

\begin{lem} \label{gq_2}
Let $U$ be an L-matrix with $\mathcal{G}_Q$ pattern and excesses
$A_I$.  Then $U$ is a square matrix if and only if $\sum_{I \in
\mathcal{G}_Q}A_I = 0$.
\end{lem}
\begin{proof}
The condition is equivalent to $\sum_{I \in \mathcal{G}_Q}r_I =
\sum_{I \in \mathcal{G}_Q}c_I.$
\end{proof}

\begin{thm}{} \label{gq_3}
Let $U$ be a square L-matrix with $\mathcal{G}_Q$ pattern and
excesses $A_I$.  Then the following are equivalent:
\begin{enumerate}
\item $\det(U) \neq 0.$
\item For any nonempty proper topset $T \subseteq \mathcal{G}_Q,
 \sum_{I \in T}A_I \ge 0$.
\item For any nonempty topset $T \subseteq \mathcal{G}_Q- \{Q\},
 \sum_{I \in T}A_I \ge 0$.
\item For any nonempty proper bottomset $B \subseteq \mathcal{G}_Q,
 \sum_{I \in B}A_I \le 0$.
\item For any nonempty bottomset $B \subseteq \mathcal{G}_Q- \{0\},
 \sum_{I \in B}A_I \le 0$.
\end{enumerate}
\end{thm}
\begin{proof}

To see that (2) and (3) are equivalent, we observe that a proper
topset of $\mathcal{G}_Q$ is the same thing as a topset of
$\mathcal{G}_Q - \{Q\}$, since the only topset of $\mathcal{G}_Q$
containing $Q$ is $\mathcal{G}_Q$ itself. Similarly, to see that (4)
and (5) are equivalent, we observe that a proper bottomset of
$\mathcal{G}_Q$ is the same thing as a bottomset of $\mathcal{G}_Q -
\{0\}$, since the only bottomset of
$\mathcal{G}_Q$ containing $0$ is $\mathcal{G}_Q$ itself.\\

To see that (2) and (4) are equivalent, we first observe that\\ $\{B
\subseteq \mathcal{G}_Q| B$ is a bottomset$\}$ = $\{\mathcal{G}_Q -T
\subseteq \mathcal{G}_Q | T$ is a topset$\}$. (See Lemma
\ref{topset_lem}(1).) By Lemma \ref{gq_2}, $\sum_{I \in
\mathcal{G}_Q -T}A_I \le 0 \Leftrightarrow \sum_{I \in T}A_I \ge
0$.\\

To see that (1) $\Rightarrow$ (2), let $T$ be a nonempty proper
topset and let $Y_T$ be the corresponding maximal superblock of
zeroes. Recall that $Y_T$ consists of blocks $B_{IJ}$ such that $I
\notin T$ and $J \in T$.  Then some matrix $U'$, formed by (perhaps)
permuting some of the rows and columns of $U$, has a decomposition
into four superblocks as follows:

\begin{equation*}
\begin{artmatrix}
&\sum_{I \in T}c_I& \sum_{I \notin T}c_I\cr \hline \sum_{I \notin
T}r_I&0&A\cr \sum_{I \in T}r_I&B&Z
\end{artmatrix}
\end{equation*}

\noindent where 0 represents $Y_T$.  Assume $\det(U) \neq 0$.  Then
$\det(U') \neq 0$, since $U'$ was formed by permuting rows of
columns of $U$. So the first $\sum_{I \in T}c_I$ columns must be
linearly independent, which means the rank of the superblock $B$
must be at least $\sum_{I \in T}c_I$. This implies $\sum_{I \in
T}r_I \ge \sum_{I \in T}c_I$, that is, $\sum_{I
\in T}A_I \ge 0$. \\

To prove that (2) $\Rightarrow$ (1) takes several pages and
constitutes the remainder of this chapter. We proceed by induction
on the size
$s \times s$ of $U$.  We fix a value of $Q$, which remains unchanged
throughout the induction.\\

We start the induction with $s = 1$, which is to say $U$ has a
single entry $u$. For $U$ to have a nonzero determinant, $u$ must be
nonzero.  We observe that, for some choice of multi-indexes $I_0$
and $J_0$, $u$ constitutes the $B_{I_0J_0}$ block of $U$. That is,
$r_{I_0} = c_{J_0} = 1$, and these are the only nonzero $r_I$ and
$c_J$.  In particular, if $I_0 \neq J_0$, the only nonzero values of
$A_I := r_I - c_I$ are $A_{I_0} =1$ and $A_{J_0} = -1$.\\

To prove the case $s=1$: we assume, for every nonempty proper topset $T \subseteq
\mathcal{G}_Q$, that $\sum_{I \in T}A_I \ge 0$; and our goal is to show that the
$B_{I_0J_0}$ block has a nonzero entry. Equivalently, since $U$ has $\mathcal{G}_Q$
pattern, we must show that $I_0 \succeq J_0$.\\

If $I_0 = J_0$ or $J_0 = Q$, this is immediate. Otherwise, we form
the nonempty proper topset $T_{\{J_0\}} := \{I \in \mathcal{G}_Q | I
\succeq J_0\}$ (See Lemma \ref{topset_lem}(3).) Since $\sum_{I \in
T_{\{J_0\}}}A_I \ge 0$ and $J_0 \in T_{\{J_0\}}$, it must be that
$I_0 \in T_{\{J_0\}}$.
That is, $I_0\succeq J_0$.\\
For the induction step, we assume that the theorem has been proved for $1,...,s-1$,
and we let $U$ be an $s\times s$ matrix such that (2) holds.  We must show that the
determinant of $U$ is nonzero.  To do this, having Corollary \ref{pv_3} in mind, we
choose an arbitrary nonempty proper topset $S \subseteq \mathcal{G}_Q$ (switching to
$S$ in order to reserve the letter $T$ for future use) and we set $C = \mathcal{G}_Q
- S$.  Then we permute the rows and columns of $U$, if necessary, so that $Y_S$, the
maximal superblock of zeroes associated to $S$, is in the upper left.  We call this
permuted matrix $U_S$, and look at its decomposition into four superblocks.

\begin{equation*}
\begin{artmatrix}
&\sum_{I \in S}c_I& \sum_{I \in C}c_I\cr \hline \sum_{I \in
C}r_I&0&A\cr \sum_{I \in S}r_I&B&Z
\end{artmatrix}
\end{equation*}

We observe that it was not necessary to permute the last $r_0$ rows
or the last $c_Q$ columns, since automatically $0 \succeq Q$, which
forces the $B_{0Q}$ block (which is contained in $Z$) to consist of
nonzero entries. Since we are assuming that (2) holds, we know that
$A$ has at least as many columns as rows, and $B$ has at least as
many rows as columns. We let $A'$ denote the leftmost maximal square
submatrix of $A$, and let $B'$ denote the uppermost maximal square
submatrix of $B$.  With this notation, the decomposition of $U_S$
can be rewritten

$$\left[
\begin{matrix}
0&A'&*\cr B'&C'&*\cr *&*&Z'\cr
\end{matrix}
\right]$$\\

\noindent This rewriting did not involve any further permuting of
the rows and columns, so we may be sure that the last $r_0$ rows and
the last $c_Q$ columns have never been permuted from the original $U$.\\

In order to show that $U$ has nonzero determinant, it is enough to
show that, for some choice of $S$, $U_S$ has nonzero determinant.
According to Corollary \ref{pv_3}, we can show that $U_S$ has
nonzero determinant by showing that (a1) $A'$ has nonzero
determinant, (b1) $B'$ has nonzero determinant, and (c) the block
$Z'$ lies entirely within the $B_{0Q}$ block, which guarantees that
its rows and columns have not been permuted and that its entries
(and hence its main diagonal entries) are all nonzero.\\

In order to show this, we fix some notations.  We recall that both
$A'$ and $B'$ were formed by first permuting some of the rows and
columns of $U$, and then deleting some of the rows and columns of
the permuted matrix. We observe that the result would have been the
same if we had first deleted the appropriate rows and columns of $U$
and then suitably permuted the rows and columns of what remained.
With regard to this equivalent alternative construction of $A'$ and
$B'$, we define $A_0'$ and $B_0'$ to be the submatrices of $U$ that
were formed by deletion of rows and columns, and were subsequently
altered by permutations of rows and columns to form, respectively,
$A'$ and $B'$. We remark that $A_0'$ and $B_0'$, being submatrices
of the L-matrix $U$, are themselves L-matrices; and that, in order
to show that $A'$ and $B'$ have nonzero determinant, it is enough to
show
that $A_0'$ and $B_0'$ have nonzero determinant.\\

In order to use induction, we need to view $A_0'$ and $B_0'$ as L-matrices with
$\mathcal{G}_Q$ pattern.  To do this, we say that an entry is in the $B_{IJ}$ block
of $A_0'$ or $B_0'$ if it was in the $B_{IJ}$ block of $U$.  Thus, if we use primes
to denote block dimensions and excessess in $A_0'$ (viz. $r_I',c_I',A_I'$) and
double primes to denote block dimensions and excesses in $B_0'$ (viz.
$r_I'',c_I'',A_I''$),

\begin{enumerate}
\item [(i)] For all $I \in \mathcal{G}_Q$, $r_I' \le r_I$, $r_I'' \le r_I$,
$c_I' \le c_I$, $c_I'' \le c_I.$
\item[(ii)] For all $I \in S, 0 = r_I' = c_I' = A_I'$, and $c_I'' =
c_I.$
\item[(iii)] For all $I \in C, 0 = r_I'' = c_I'' = A_I''$, and $r_I'
= r_I.$
\end{enumerate}

With these notations, we repeat the previous decomposition of $U_S$,
this time including some of the block dimensions.

\begin{equation*}
\begin{artmatrix}
&\sum_{I \in S}c_I& \sum_{I \in C}c_I'&\cr \hline \sum_{I \in
C}r_I&0&A'&*\cr \sum_{I \in S}r_I''&B'&C'&*\cr &*&*&Z'\cr
\end{artmatrix}
\end{equation*}

With these notations, we can break statement (c) above into two parts:\\
Statement (a2):
\begin{equation*}
\sum_{I \in C}c_I' \ge  \sum_{I \in C - \{Q\}}c_I.
\end{equation*}
\noindent and Statement (b2):
\begin{equation*}
\sum_{I \in S}r_I'' \ge  \sum_{I \in S - \{0\}}r_I.
\end{equation*}

\noindent To state (a1) with the new notations, we use the induction
hypothesis to obtain a set of conditions that $A_0'$ have nonzero
determinant:
\begin{equation*}
\text{For all nonempty topsets } T \subseteq \mathcal{G}_Q - \{Q\},
\; \sum_{I \in T}A_I' \ge 0.
\end{equation*}
By (ii), this is equivalent to:\\
\begin{equation*}
\text{For all nonempty topsets} \; T \subseteq \mathcal{G}_Q -
\{Q\}, \; \sum_{I \in T \cap C}A_I' \ge 0.
\end{equation*}
By Lemma \ref{topset_lem} (4)(b)(ii), this is equivalent to:
\begin{equation}  \label{&} \text{For all nonempty topsets} \; T \subseteq C - \{Q\}, \;
\sum_{I \in T}A_I' \ge 0.
\end{equation}
Also, we recall (a2):
\begin{equation*}
\sum_{I \in C}c_I' \ge  \sum_{I \in C - \{Q\}}c_I
\end{equation*}
which, since $A'$ is square, is equivalent to
\begin{equation*}
\sum_{I \in C}r_I \ge \sum_{I \in C - \{Q\}}c_I
\end{equation*}
or
\begin{equation} \label{&&}
r_Q + \sum_{I \in C - \{Q\}}A_I  \ge 0.
\end{equation}

We now introduce a condition on $U$ that, we claim, implies both
\eqref{&} and \eqref{&&}:
\begin{equation} \label{pou}
\text{For all nonempty topsets} \; T \subseteq C - \{Q\}, \sum_{I
\in T}A_I \ge 0.
\end{equation}
To see that \eqref{pou} implies \eqref{&}, let $T$ be a nonempty
topset of $C - \{Q\}$.  Then
\begin{align*}
0 \le \sum_{I \in T}A_I =\sum_{I \in T}r_I - \sum_{I \in T}c_I
  &=\sum_{I \in T}r'_I - \sum_{I \in T}c_I\\
&\le \sum_{I \in T}r'_I - \sum_{I \in T}c_I' =\sum_{I \in T}A_I'.
\end{align*}

To see that \eqref{pou} implies \eqref{&&}, apply \eqref{pou} to $C
- \{Q\}$ (which is a topset of itself):
\begin{equation*}
0 \le \sum_{I \in C - \{Q\}}A_I \le r_Q + \sum_{I \in C - \{Q\}}A_I.
\end{equation*}

Summarizing the progress so far, we have shown that the induction
step will follow if we can demonstrate a maximal superblock $Y_S$
for which (a1), (a2), (b1), and (b2) hold.  We have shown that (a1)
and (a2) hold if $S$ satisfies \eqref{pou}.  We now proceed in a
completely analogous fashion to establish another condition
on $S$ that will ensure (b1) and (b2) hold.\\

\noindent To state (b1) with the new notations, we argue similarly,
using the induction hypothesis to obtain a set of sufficient
conditions that $B_0'$ have nonzero determinant:
\begin{equation*}
\text{For all nonempty bottomsets} \; B \subseteq \mathcal{G}_Q -
\{0\}, \; \sum_{I \in B}A_I'' \le 0.
\end{equation*}
By (iii), this is equivalent to:\\
\begin{equation*}
\text{For all nonempty bottomsets} \; B \subseteq \mathcal{G}_Q -
\{0\}, \; \sum_{I \in B \cap S}A_I'' \le 0.
\end{equation*}
By Lemma \ref{topset_lem} (4)(d)(ii), this is equivalent to:
\begin{equation} \label{&'}
\text{For all nonempty bottomsets} \; B \subseteq S - \{0\}, \;
\sum_{I \in B}A_I'' \le 0.
\end{equation}
Also, we recall (b2):
\begin{equation*}
\sum_{I \in S}r_I'' \ge  \sum_{I \in S - \{0\}}r_I
\end{equation*}
which, since $B'$ is square, is equivalent to
\begin{equation*}
\sum_{I \in S}c_I \ge \sum_{I \in S - \{0\}}r_I
\end{equation*}
or
\begin{equation} \label{&&'}
\sum_{I \in S - \{0\}}A_I - c_0 \le 0.
\end{equation}\\

We now introduce a condition on $U$ that, we claim, implies both
\eqref{&'} and \eqref{&&'}:
\begin{equation}\label{poupou}
\text{For all nonempty bottomsets} \; B \subseteq S - \{0\}, \sum_{I
\in B}A_I \le 0.
\end{equation}
To see that \eqref{poupou} implies \eqref{&'}, let $B$ be a
bottomset of $S - \{0\}$. Then

\begin{align*}
0 \ge \sum_{I \in B}A_I
  =\sum_{I \in B}r_I - \sum_{I \in B}c_I
  &=\sum_{I \in B}r_I - \sum_{I \in B}c_I''\\
&\ge \sum_{I \in B}r''_I - \sum_{I \in B}c_I'' =\sum_{I \in B}A_I''.
\end{align*}

To see that \eqref{poupou} implies \eqref{&&'}, apply \eqref{poupou}
to $S - \{0\}$ (which is a bottomset of itself):
\begin{equation*}
0 \ge \sum_{I \in S - \{0\}}A_I \ge \sum_{I \in S - \{0\}}A_I - c_0.
\end{equation*}\\

Again summarizing the progress so far, we have shown that the
induction step will hold if we can guarantee the existence of a
nonempty topset $S$ of $\mathcal{G}_Q -\{Q\}$ for which \eqref{pou}
and
\eqref{poupou} are true.\\

Our method of proof is algorithmic.  We start with a candidate for $S$, namely
$\{0\}$, for which \eqref{poupou} is (vacuously) true, but for which \eqref{pou} may
not be true.  We describe a procedure for replacing the current candidate for $S$ by
another candidate topset that properly contains it, a process that must stop because
it can be carried out only a finite number of times.  We then show that (i) the
procedure results in a new candidate that also satisfies \eqref{poupou}, and (ii) if
the process cannot be continued, the current candidate satisfies \eqref{pou} as
well. Establishing (i) and (ii) suffices to prove the induction step,
and the theorem.\\

By way of notation, let $S$ represent the current candidate topset,
which is known to satisfy \eqref{poupou}, and let $S'$ represent the
next candidate. That is, we start with $S$ = $\{0\}$. If we have a
current $S$, the procedure to find $S'$ is as follows:  Among all
nonempty topsets of $\mathcal{G}_Q - S - \{Q\}$, pick a smallest one
$X$ (smallest by inclusion) such that $\sum_{I
\in X}A_I < 0$.  We set $S' := S \cup X$.\\

We note that if no such $X$ exists, $S$ satisfies \eqref{pou}, since
for all nonempty topsets $T \subseteq \mathcal{G}_Q -\{Q\} - S$, we
have $\sum_{I \in T}A_I \ge 0$, and $\mathcal{G}_Q - S - \{Q\} = C
-\{Q\}$.  This establishes (ii).\\

If $X$ can be found, we note that $S' := S \cup X$ is a topset of
$\mathcal{G}_Q$ by Lemma \ref{topset_lem} (4)(b)(iii).  We note that
$S'$ is nonempty because $0 \in S$ and proper because $Q \notin S$
and $Q \notin X$.  We note that, by construction, $S$ and $X$ are
disjoint.\\

We must show that $S':= S \cup X$ satisfies \eqref{poupou}.  That
is, for any bottomset $B \subseteq (S \cup X) - \{0\} = (S  - \{0\})
\cup X$, we must show that $\sum_{I \in B}A_I \le 0$.  For this, it
is enough establish the following two inequalities:

\begin{equation} \label{a}
\sum_{I \in B \cap (S  - \{0\})}A_I \le 0
\end{equation}

\noindent and

\begin{equation} \label{b}
\sum_{I \in B \cap X}A_I \le 0.
\end{equation}

Statement \eqref{a} follows immediately from the fact that
\eqref{poupou} holds for $S$, once one has verified that $B \cap (S
- \{0\})$ is a bottomset of $S  - \{0\}$, which follows from Lemma
\ref{topset_lem}
(4)(b)(iv).\\

Statement \eqref{b} follows from the following two inequalities:

\begin{equation} \label{c}
\sum_{I \in X} A_I < 0
\end{equation}

\noindent and
\begin{equation} \label{d}
\sum_{I \in X-B} A_I \ge 0
\end{equation}

\noindent since
\begin{equation*}
\sum_{I \in B \cap X}A_I = \sum_{I \in X} A_I - \sum_{I \in X-B}
A_I.
\end{equation*}\\

Statement \eqref{c} is true by construction. Statement \eqref{d} is
certainly true if $X-B$ is empty. If $X-B$ is nonempty it is a
topset of $\mathcal{G}_Q - S - \{Q\}$ by  Lemma \ref{topset_lem}
(4)(b)(iv), and then \eqref{d} follows by construction, because
$X-B$ is topset of $\mathcal{G}_Q - S - \{Q\}$ that is smaller (by
inclusion) than $X$.
\end{proof}

\chapter{Coefficient Matrices}

\section{Coefficient Matrices of $R_{j-d}*\mathcal{E}(C)$}

We fix a codimension $r$ and a socle degree $j$.  We consider a
vector subspace $\mathcal{E} \subseteq \mathcal{D}_j$.  We fix a
degree $d \le j$ and we wish to consider $\dim_k R_{j-d}*\mathcal{
E}$.  By way of notation, we make the convention that $e:=j-d$, $j =
e + d$. Also, whenever we wish to specify that generators
$f_1,...,f_s$ of $\mathcal{E}$ are to be taken from a particular
vector subspace $\mathcal{W} \subseteq \mathcal{D}_j$, we will
simply write $\mathcal{E} := \langle f_1,...,f_s\rangle
\subseteq \mathcal{W}$.\\

Let $\mathcal{M}$ be a set of multi-indexes of degree $j$, and
define $\mathcal{W}_{\mathcal{M}} := \langle \{x^J | J \in
\mathcal{M} \}\rangle$.  That is, $\mathcal{W}_\mathcal{M}$ is a
vector subspace of $\mathcal{D}_j$ generated by monomials.  We let
$\mathcal{E}:= \langle f_1,...,f_s\rangle  \subseteq
\mathcal{W}_{\mathcal{M}}$. For each generator $f_i$, we write $f_i
= \sum_{J \in \mathcal{M}}z_{iJ}x^J$,
where each $z_{iJ} \in k$.\\

We will always adopt the point of view that the $z_{iJ}$'s are allowed to vary.
Specifying a value $c_{iJ} = z_{iJ}$ for each of them, or equivalently specifying an
element $C := (...,c_{iJ},...)
\in k^{s\#(\mathcal{M})}$, determines a particular subspace\\
$\mathcal{E}(C) = \langle f_1(C),...,f_s(C)\rangle  \subseteq
\mathcal{W}_\mathcal{M} \subseteq \mathcal{D}_j$.  When we later
define the matrices $U'$ and $U$ with coefficients in
$k[\{z_{iJ}\}]$, $U'(C)$ and $U(C)$ will similarly be specific
matrices with coefficients in $k$.  In other words, from now on we
will view the $f_i$'s, $E$, $U'$, and $U$ (written without the $\lq
\lq C"$) as functions whose domain is the irreducible affine variety
$k^{s\#(\mathcal{M})}$.  We will consider the images of these
functions as families of vectors, vector subspaces, or matrices,
parameterized by elements $C \in k^{s\#(\mathcal{M})}$. Whenever we
wish to indicate a specific element of a family, we will use the
notation with the $\lq \lq C"$, sometimes without explicitly
mentioning the $f_i$'s or the $z_{iJ}$'s.\\

\begin{lem} \label{gen}
For all $C \in k^{s\#(\mathcal{M})}$, $R_e*\mathcal{E}(C)$ is
generated as a vector space by\\ $\{X^E*f_i(C) | X^E$ is a monomial
of degree $e$ and $1 \le i \le s\}$.
\end{lem}

\begin{proof}
Any element of $R_e*\mathcal{E}(C)$ can be written\\ $(\sum
a_EX^E)*(\sum b_if_i(C)) = \sum a_Eb_i(X^E*f_i(C)).$
\end{proof}

\begin{lem}
Let $X^E$ be a monomial of degree $e$ and $f_i = \sum_{J \in
\mathcal{M}}z_{iJ}x^J$ as above.  Then
\begin{equation*} \label{niJ}
X^E*f_i = \sum_{J; E \succeq J \in \mathcal{M}}n_{iJ}z_{iJ}x^{J-E} =
\sum_{D; D + E = J \in \mathcal{M}} n_{iJ}z_{iJ}x^D,
\end{equation*}
\noindent where the $n_{iJ}$'s are positive integers.
\end{lem}
\begin{proof}
This follows immediately from the definition of the operation * as
partial differentiation.
\end{proof}

We wish to translate the problem of determining $\dim_k
(R_e*\mathcal{E}(C))$ into the language of matrices.  To this end,
we define a matrix $U'$, the \emph{uncropped coefficient matrix of
$e^{th}$ partial derivatives of $\mathcal{E}$}, or simply the
\emph{$e^{th}$ uncropped matrix of $\mathcal{E}$}, as follows. (We
use the prime
to distinguish it from the $e^{th}$ cropped matrix $U$, to be defined later.)\\

The matrix $U'$ has rows indexed by ordered pairs $(E,i)$ where $E$
is a multi-index of degree $e$ and $i \in \{1,...,s\}$. The rows are
ordered according to the rule that $(E_1,i_1)$ comes before
$(E_2,i_2)$ if $E_1 > E_2$ (in lexicographic order) or if $E_1 =
E_2$ and $i_1 < i_2$.  The columns of $U'$ are indexed by
multi-indexes $D$ of degree $d$, where $D_1$ comes before $D_2$ if
$D_1 > D_2$ (in lexicographic order).  The entry of $U'$ in the $(
(E,i),D )$ position is $n_{iJ}z_{iJ}$ if $J:= D + E \in\mathcal{M}$
and 0 otherwise.

\begin{lem} \label{is_l}
Let the vector subspace $\mathcal{E} := \langle f_1,...,f_s\rangle
\subseteq \mathcal{W}_{\mathcal{M}} \subseteq \mathcal{D}_j$ be
defined as above and let $U'$ be its $e^{th}$ uncropped  matrix.
Then $U'$ is an L-matrix over\\ $k[\{z_{iJ}|i \in \{1,...,s\}$ and
$J \in \mathcal{M}\}]$.
\end{lem}
\begin{proof}
By construction, the entries of $U'$ are either 0 or positive
integer multiples of some $z_{iJ}$.  So it remains to show that
every
$z_{iJ}$ moves to the left.\\

Assume $z_{iJ}$ is the variable in two different locations
$((E_1,i_1),D_1)$ and\\ $((E_2,i_2),D_2)$.  We first note that $i_1
= i_2 = i$, since that is the only way (by the definition of $U'$)
that the variable $z_{iJ}$ can appear at all.  Since the order of
the rows $(E_1,i)$ and $(E_2,i)$ in $U'$ is determined by
lexicographical order of $E_1$ and $E_2$, and the order of the
columns is determined by lexicographical order of $D_1$ and $D_2$,
we
must show $E_1 > E_2$ if and only if $D_1 < D_2$.\\

By way of notation, let $E_1 := (...,e_{1i},...); E_2 :=
(...,e_{2i},...); D_1 := (...,d_{1i},...); D_2 := (...,d_{2i},...)$.
From the definition of $U'$, $D_1 + E_1 = J = D_2 + E_2$. So for
each co-ordinate $i$, $e_{1i} - e_{2i} = d_{2i} - d_{1i}$. Then
\begin{align*}
E_1 > E_2 &\Leftrightarrow e_{1i} = e_{2i} \; for \; i=1,...,m-1 \;
and
\; e_{1m} > e_{2m} \text{ (for some }  m)\\
&\Leftrightarrow d_{1i} = d_{2i} \; for \; i=1,...,m-1 \; and \;
d_{1m} < d_{2m} \Leftrightarrow D_1 < D_2.
\end{align*}
\end{proof}

\begin{lem} \label{dim_u}
Let the family of vector subspaces $\mathcal{E}:= \langle
f_1,...,f_s\rangle  \subseteq \mathcal{W}_{\mathcal{M}} \subseteq
\mathcal{D}_j$ be defined as above and let $U'$ be its $e^{th}$
uncropped  matrix. Then for all $C \in k^{s\#(\mathcal{M})}$,
$\rank(U'(C)) = \dim_k(R_e*\mathcal{E}(C))$.
\end{lem}
\begin{proof}
We combine previous results concerning the vector space $\mathcal{D}_d$, of which
$\{x^D\}$ is a basis and $R_e*\mathcal{E}(C)$ is a vector subspace.
$R_e*\mathcal{E}(C)$ is generated by the vectors $X_E*f_i(C)$, so to find its
dimension we express each generator as a linear combination of basis vectors and
determine the rank of the matrix of coefficients.  Since $X_E*f_i(C) = \sum_{D; D +
E = J \in \mathcal{M}} n_{iJ}c_{iJ}x^D$, its matrix of coefficients is $U'(C)$.
\end{proof}

\section{Coefficient Matrices for Constrained Subspaces of
$\mathcal{D}_j$}

As before, we assume that the codimension $r$ and socle degree $j$
have been fixed.  We now fix a nonegative number $n \le r$ and a
multi-index $Q := (Q_1,...,Q_n)$, where $0 \le Q_i \le j$ for $i =
1,...,n$.  We say an $r$-tuple $I := (I_1,...,I_r)$ of non-negative
integers, of any degree $d \le j $, is \emph{constrained by $Q$} if
$I_i \le Q_i$ for $i = 1,...,n$.  We say that a monomial $X^I$ or
$x^I$ is \emph{constrained by $Q$} if I is constrained by $Q$. We
define $\mathcal{M}_Q(d)$ to be the set of all multi-indexes of
degree $d$ that are constrained by $Q$.  In particular,
$\mathcal{M}_Q(j)$ is a set of multi-indexes of degree $j$, so we
can consider $\mathcal{W}_{\mathcal{M}_Q(j)} \subseteq
\mathcal{D}_j$. By way of notation, we will always write $m_Q(d)$
for $\#(M_Q(d))$.  As in the previous section, we fix degree $d$ and
write $e = j-d$.

\begin{lem} \label{crop}
Let $\mathcal{E} = \langle f_1,...,f_s\rangle  \subseteq
\mathcal{W}_{\mathcal{M}_Q(j)} \subseteq \mathcal{D}_j$ be a family
of vector subspaces and let $U'$ be its $e^{th}$ uncropped  matrix.
If $E \notin \mathcal{M}_Q(e)$ is a multi-index of degree $e$ and $1
\le i \le s$, the $(E,i)$ row of $U'$ consists entirely of zeroes.
If  $D \notin \mathcal{M}_Q(d)$ is a multi-index of degree $d$, the
$D$ column of $U'$ consists entirely of zeroes.
\end{lem}
\begin{proof}
For the $((E,i), D)$ entry to be nonzero, we need $E + D = J \in
\mathcal{M}_Q(j)$.  That is, writing $D = (d_1,...,d_r)$ and $E =
(e_1,...,e_r)$, we must have $d_i + e_i \le Q_i$ for $i = 1,...,n$.
This implies $d_i \le Q_i$ and $e_i \le Q_i$ for $i = 1,...,n$.
Equivalently, $D \in \mathcal{M}_Q(d)$ and $E \in \mathcal{M}_Q(e)$.
\end{proof}

Since our interest in $U'$ stems from our desire to compute its rank, we lose
nothing by deleting rows and columns that consist entirely of zeroes.  We define
$U$, the \emph{cropped coefficient matrix of $e^{th}$ partial derivatives of
$\mathcal{E}$}, or simply the \emph{$e^{th}$ cropped matrix of $\mathcal{E}$}, to be
the submatrix of $U'$ obtained by taking only those $((E,i), D)$ entries for which
$D \in \mathcal{M}_Q(d)$ and $E \in \mathcal{M}_Q(e)$.  More precisely,

\begin{deff} \label{cropped}
For a fixed choice of $r,j,Q,d,e:=j-d$, and $s$, the \emph{$e^{th}$ cropped matrix
of $\mathcal{E}$} is defined as follows:
\end{deff}
The rows are indexed by pairs $(E,i)$ where $E \in \mathcal{M}_Q(e)$ and $i \in
\{1,...,s\}$, ordered by the rule that $(E_1,i_1)$ comes before $(E_2,i_2)$ if $E_1
> E_2$ (in lexicographic order) or if $E_1 = E_2$ and $i_1 < i_2$.  The columns are
indexed by elements $D \in \mathcal{M}_Q(d)$, ordered by the rule that $D_1$ comes
before $D_2$ if $D_1 > D_2$ (in lexicographic order).  Writing
$$X^E*f_i = \sum_{D; D + E
= J \in \mathcal{M}_Q(j)} n_{iJ}z_{iJ}x^D,$$ \noindent the entry of $U$ in the
$((E,i),D )$ position is $n_{iJ}z_{iJ} = n_{i(D+E)}z_{i(D+E)}$ if $J \in
\mathcal{M}_Q(j)$ and 0 otherwise.

\begin{cor} \label{crop_rank}
Let the family of vector subspaces $\mathcal{E} = \langle
f_1,...,f_s\rangle  \subseteq \mathcal{W}_{\mathcal{M}_Q(j)}
\subseteq \mathcal{D}_j$ be defined as above, let $U'$ be its
$e^{th}$ uncropped matrix, and let $U$ be its $e^{th}$ cropped
matrix.  Then $\rank(U) = \rank(U')$.
\end{cor}
\begin{proof}
Removing rows and columns of zeroes does not affect the rank of a
matrix.
\end{proof}

\begin{cor} \label{crop_rank_c}
For $C \in k^{sm_Q(j)}$, let $\mathcal{E}(C) = \langle
f_1(C),...,f_s(C)\rangle \subseteq \mathcal{W}_{\mathcal{M}_Q(j)}
\subseteq \mathcal{D}_j$, let $U'(C)$ be its $e^{th}$ uncropped
matrix, and let $U(C)$ be its $e^{th}$ cropped matrix.  Then
$\rank(U(C)) = \rank(U'(C)) = \dim_kR_e*\mathcal{E}(C)$.
\end{cor}
\begin{proof}
Again, removing rows and columns of zeroes does not affect the rank
of a matrix. The second equality simply repeats the statement of
Lemma \ref{dim_u}.
\end{proof}

\begin{lem} \label{is_l_too}
Let the family of vector subspaces $\mathcal{E} := \langle
f_1,...,f_s\rangle  \subseteq \mathcal{W}_{{\mathcal{M}}_Q(j)}
\subseteq \mathcal{D}_j$ be defined as above and let $U$ be the
$e^{th}$ cropped matrix of $\mathcal{E}$. Then $U$ is an L-matrix
over $k[\{z_{iJ}|i \in \{1,...,s\}$ and $J \in \mathcal{M}_Q(j)\}]$.
\end{lem}
\begin{proof}
By Lemma \ref{is_l}, the uncropped matrix is an L-matrix over\\
$k[\{z_{iJ}|i \in \{1,...,s\}$ and $J \in \mathcal{M}_Q(j)\}]$.  By
Lemma \ref{pv_0}, $U$ is as well, since it is defined to be a
submatrix of the uncropped matrix.
\end{proof}

The matrix $U$ is an L-Matrix, and we describe a scheme for
subdividing it into blocks $B_{IJ}$, where $I,J \in \mathcal{G}_Q$,
that makes $U$ an L-Matrix with $\mathcal{G}_Q$ pattern.\\

Given a multi-index $I \in \mathcal{G}_Q$, we must designate $r_I$
row indices $(E,i)$ and $c_I$ column indices $D$ to associate with
$I$. Writing $I = (I_1,...,I_n)$, $D = (d_1,...,d_r)$ and $E =
(e_1,...,e_r)$, we associate with $I$ those row indices $(E,i)$ for
which $E = (I_1,...,I_n,e_{n+1},...,e_r)$, and we associate with $I$
those column indices $D$ for which $D =
(Q_1-I_1,...,Q_n-I_n,d_{n+1},...,d_r)$.  We call this assignment of
rows and columns the \emph{standard assignment}.

\begin{lem} \label{crop_l}
Let the family of vector subspaces $\mathcal{E} := \langle f_1,...,f_s\rangle
\subseteq \mathcal{W}_{{\mathcal{M}}_Q(j)} \subseteq \mathcal{D}_j$ be defined as
above and let $U$ be the $e^{th}$ cropped matrix of $\mathcal{E}$. Then the standard
assignment makes $U$ an L-matrix with $\mathcal{G}_Q$ pattern.
\end{lem}
\begin{proof}
We must verify the following statements.
\begin{enumerate}
\item[(i)] Every row and every column of $U$ is associated to a
unique $I$.
\item[(ii)] The standard assignment subdivides $U$ into blocks.  That
is, for any multi-index $I \in \mathcal{G}_Q$, all rows associated
to $I$ are consecutive in $U$, and all columns associated to $I$ are
consecutive in $U$.
\item[(iii)] The block-row indices $I$ occur in lexicographic order,
and the block-column indices $I$ occur in reverse lexicographic
order.
\item[(iv)] The entries in the block $B_{IJ}$ are nonzero if $I
\succeq J$ and 0 otherwise.
\end{enumerate}
For (i), consider a row $(E,i)$ of $U$ and write $E =
(e_1,...,e_r)$. Then the only possible candidate for $I$ is
$(e_1,...,e_n)$, and we must verify that it is an element of
$\mathcal{G}_Q$.  Since $U$ is the cropped matrix, $E \in
\mathcal{M}_Q(e)$, which implies $e_i \le Q_i$ for each $i$; hence
$I \in \mathcal{G}_Q$.\\

Similarly, let $D$ be a column of $U$ and write $D = (d_1,...,d_r)$.
Then the only possible candidate for $I$ is $(Q_1-d_1,...,Q_n -
d_n)$, and we must verify that it is an element of $\mathcal{G}_Q$.
This is true because $D \in \mathcal{M}_Q(d)$, so $d_1 \le
Q_1,...,d_n \le Q_n$, which is to say $0 \le Q_1 - d_1 \le Q_1, ...,
0 \le Q_n - d_n \le Q_n$.\\

 For (ii), recall that the ordering of the rows and columns of
$U$ is given in Defintion \ref{cropped}.  Assume that $(E_1,i_1)$ and $(E_3,i_3)$
are two row indices associated to $I$, and that $(E_2,i_2)$ comes between them.
Write $E_1 = (I_1,...,I_n,...), E_2 = (e_1,...,e_n,...), E_3 = (I_1,...,I_n,...)$.
We must show that $e_i = I_i$ for $i = 1,...,n$, and we argue by contradiction.  If
not, let $m$ be the first co-ordinate in which this is not so.  By the rule for
ordering the rows of $U$, $E_1 \ge E_2 \ge E_3$, so $I_m \ge e_m \ge I_m$,
which is impossible if $e_m \neq I_m$.\\

The argument for columns is similar.  Assume that $D_1$ and $D_3$
are two column indices associated to $I$, and that $D_2$ comes
between them. Write $D_1 = (Q_1 - I_1,...,Q_n -I_n,...), D_2 =
(d_1,...,d_n,...), D_3 = (Q_1 - I_1,...,Q_n -I_n,...)$.  We must
show that $d_i = Q_i-I_i$ for $i = 1,...,n$, and we argue by
contradiction.  If not, let $m$ be the first co-ordinate in which
this is not so.  By the rule for ordering the columns of $U$, $D_1
\ge D_2 \ge D_3$, so $Q_m - I_m \ge d_m \ge Q_m - I_m$,
which is impossible if $d_m \neq Q_m - I_m$.\\

For (iii), let row index $(E_1,i_1)$ associated to $I =
(I_1,...,I_n) \in \mathcal{G}_Q$ come before row index $(E_2,i_2)$
associated to $J = (J_1,...,J_n) \in \mathcal{G}_Q$. Then $E_1 \ge
E_2$, or as expanded by co-ordinates, $(I_1,...,I_n,...) \ge
(J_1,...,J_n,...)$. That is, either $I_i = J_i$ for $i = 1,...,n$,
or else, for some $m$, $I_i = J_i$ for $i = 1,...,m-1$ and $I_m >
J_m$; equivalently, $I \ge
J$.\\

The argument for columns is similar.  Let column index $D_1$
associated to $I = (I_1,...,I_n) \in \mathcal{G}_Q$ come before
column index $D_2$ associated to $J = (J_1,...,J_n) \in
\mathcal{G}_Q$. Then $D_1 > D_2$, or as expanded by co-ordinates,
$(Q_1 -I_1,...,Q_n -I_n,...) \ge (Q_1 -J_1,...,Q_n-J_n,...)$.  That
is, either $Q_i - I_i = Q_i -J_i$ for $i = 1,...,n$, or else, for
some $m$, $Q_i - I_i = Q_i -J_i$ for $i = 1,...,m-1$ and $Q_m -I_m >
Q_m -J_m$.  So either $I_i = J_i$ for $i = 1,...,n$ or $I_i = J_i$
for $i = 1,...,m-1$ and $J_m
> I_m$; equivalently, $J \ge
I$.\\

For (iv), we recall from the definition of $U$ that the $((E,i),D)$
entry is nonzero if and only if $E + D = J \in \mathcal{M}_Q(j)$. So
if $(E,i)$ is associated to $I = (I_1,...,I_n)$ and $D$ is
associated to $J = (J_1,...,J_n)$, the condition for being nonzero
becomes $I_i + (Q_i -J_i) \le Q_i$ for $i = 1,...,n$, or
equivalently, $I_i \le J_i$ for $i = 1,...,n$.  But that is the
definition of $I \succeq J$.
\end{proof}

\begin{prop} \label{u_count}
Let the family of vector subspaces $\mathcal{E} := \langle
f_1,...,f_s\rangle  \subseteq \mathcal{W}_{{\mathcal{M}}_Q(j)}
\subseteq \mathcal{D}_j$ be defined as above and let $U$ be the
$e^{th}$ cropped matrix of $\mathcal{E}$. Then, with the standard
assignment, the block dimensions $r_I$ and $c_I$  of $I =
(I_1,...,I_n) \in \mathcal{G}_Q$ are given as follows.  Setting $p =
I_1 + ... + I_n$ and $q = Q_1 + ... + Q_n$, we have:
\begin{equation}
\text{If} \;\; p \le e, \; then \; r_I =s
\genfrac(){0cm}{0}{e-p+r-n-1}{r-n-1}; \;\; otherwise, \; r_I = 0.
\end{equation}
\begin{equation}
\text{If} \;\; q-p \le d, \; then \;  c_I = \genfrac(){0cm}{0}{d -
(q - p)+r-n-1}{r-n-1}; \;\; otherwise, \; c_I = 0.
\end{equation}
\end{prop}
\begin{proof}
To find $r_I$, we count the number of ways of forming multi-indexes
$(E,i)$ associated to $I$. Writing $E =
(I_1,...,I_n,e_{n+1},...,e_r)$, and recalling that $E$ must be of
degree $e$, we see immediately that this is impossible unless $e \ge
I_1 + ... + I_n = p$.  In this case, we must assign non-negative
integer values of $e_{n+1},...,e_r$ that bring the total degree up
to $e$. Equivalently, we must count the number of monomials of
degree $e-p$ in $r-n$ variables (and then multiply by $s$ to account
for all possible choices of $i$).  As is well-known, there are
$\genfrac(){0cm}{0}{t + u- 1}{u-1}$ monomials of degree $t$ in $u$
variables, so $r_I = s
\genfrac(){0cm}{0}{e-p+r-n-1}{r-n-1}$ when $p \le e$.\\

Similarly, to find $c_I$ we count the number of ways of forming
multi-indexes $D$ associated to $I$.  Writing $D = (Q_1
-I_1,...,Q_n-I_n,d_{n+1},...,d_r)$, and recalling that $D$ must be
of degree $d$, we see immediately that this is impossible unless $d
\ge (Q_1 -I_1) + ... +(Q_n-I_n) = q -p$. In this case, we must
assign non-negative integer values of $d_{n+1},...,d_r$ that bring
the total degree up to $d$. That is, we must count the number of
monomials of degree $d-(q-p)$ in $r-n$ variables.  This gives $c_I
=\genfrac(){0cm}{0}{d - (q - p)+r-n-1}{r-n-1}$ when $ d \ge q-p$.
\end{proof}

\begin{cor} \label{ord_pres}
Let the family of vector subspaces\\ $\mathcal{E} := \langle
f_1,...,f_s\rangle  \subseteq \mathcal{W}_{{\mathcal{M}}_Q(j)}
\subseteq \mathcal{D}_j$ be defined as above and let $U$ be the
$e^{th}$ cropped matrix of $\mathcal{E}$. Under the standard
assignment, denote the block dimensions $r_I$ and $c_I$. If $I, J
\in \mathcal{G}_Q$ and $I \succeq J$, then $r_I \ge r_J$ and $c_I
\le c_J$.  In particular, both $r_I$ and the excess $A_I = r_I -
c_I$ are order-preserving functions on $\mathcal{G}_Q$.
\end{cor}
\begin{proof}
This is a consequence of the formulas in Proposition \ref{u_count}.
Write $I = (I_1,...,I_n)$, $J = (J_1,...,J_n)$, $p_I = I_1+...+I_n$,
$p_J = J_1+...+J_n$.  We remark that if $I \succeq J$, then the
definition of partial order gives $p_I \le
p_J$ and $q - p_I \ge q - p_J$.\\

To see that $r_I$ is order-preserving, assume that $I \succeq J$.
If $e < p_I \le p_J $, then $r_I = r_J = 0$ and $r_I \ge r_J$ as
required. If $p_I \le e < p_J$, then $r_I = s
\genfrac(){0cm}{0}{e-p_I+r-n-1}{r-n-1}$ and $r_J = 0$, and again
$r_I \ge r_J$.  Finally, if $p_I \le p_J \le e$, then $r_I = s
\genfrac(){0cm}{0}{e-p_I+r-n-1}{r-n-1}$ and $r_J = s
\genfrac(){0cm}{0}{e-p_J+r-n-1}{r-n-1}$.  Since $p_I \le p_J$, this
gives $r_I \ge r_J$.\\

The argument for $c_I$ is similar. If $q - p_I \ge q - p_J > d$,
then $c_I = c_J = 0$ and $c_I \le c_J$ as required. If $q - p_I > d
\ge q - p_J$, then $c_I = 0$ and $c_J = \genfrac(){0cm}{0}{d - (q -
p_J)+r-n-1}{r-n-1}$, and again $c_I \le c_J$.  Finally, if $d \ge q
- p_I \ge q - p_J$, then $c_I = \genfrac(){0cm}{0}{d - (q -
p_I)+r-n-1}{r-n-1}$ and $c_J = \genfrac(){0cm}{0}{d - (q -
p_J)+r-n-1}{r-n-1}$.  Since $p_I \le p_J$, this gives $c_I \le c_J$.
\end{proof}

\begin{thm} {} \label{max_rank}
Let the family of vector subspaces $\mathcal{E} := \langle
f_1,...,f_s\rangle  \subseteq \mathcal{W}_{{\mathcal{M}}_Q(j)}
\subseteq \mathcal{D}_j$ be defined as above and let $U$ be the
$e^{th}$ cropped matrix of $\mathcal{E}$. If $U$ has at least as
many rows as columns, then it has maximal rank.
\end{thm}
\begin{proof}
We use the standard assignment to regard $U$ as an L-matrix with
$\mathcal{G}_Q$ pattern.  Since it has at least as many rows as
columns, $$0 \le \sum_{I \in \mathcal{G}_Q}r_I  -  \sum_{I \in
\mathcal{G}_Q}c_I = \sum_{I \in \mathcal{G}_Q}A_I.$$  By Corollary
\ref{ord_pres}, $A_I$ is an order-preserving function on
$\mathcal{G}_Q$, so according to Proposition \ref{tpp}:
\begin{equation} \label{cur}
\text{For any topset} \; T \subseteq \mathcal{G}_Q, \; \sum_{I \in
T} A_I \ge 0.
\end{equation}
By Theorem \ref{gq_3}, this would settle the matter, if only $U$
were square. So our goal is to show that we can delete rows, one at
a time, in such a way that, at every stage, \eqref{cur} remains true
for the new values of $A_I$ corresponding to the submatrix
(still with $\mathcal{G}_Q$ pattern) formed by deleting the row.\\

At each stage, we consider the subset $\mathcal{G'} \subseteq
\mathcal{G}_Q$, consisting of all multi-indexes $I$ for which $r_I$
remains nonzero. When we start out, $\mathcal{G'}$ is a topset,
because, by Corollary \ref{ord_pres}, $r_I$ is order-preserving:
given $I,J \in \mathcal{G}_Q$ such that $r_I > 0$ and $J \succeq I$,
we have $r_J \ge r_I > 0$ and $J \in \mathcal{G}'$.  When we delete
a row, we always choose a row associated with an $I$ that is minimal
in $\mathcal{G'}$, and claim that $\mathcal{G'}$ remains a topset:
If before the deletion, $r_I
> 1$, $\mathcal{G'}$ is unchanged. If before the deletion, $r_I = 1$,
the deletion will remove $I$ from $\mathcal{G}'$, and the result
will remain a topset. (See Lemma \ref{topset_lem} (3), setting $X :=
\mathcal{G}' - \{I\}$).\\

So assume at some stage that we have deleted some number of rows
from $U$, each time diminishing $r_I$ by 1 for some minimal $I \in
\mathcal{G'}$, and that \eqref{cur} remains true.  If we are not yet
done, by Lemma \ref{gq_2} it must be that $\sum_{I \in
\mathcal{G'}}A_I
> 0$, that is, the inequality is strict. We seek to find a minimal
multi-index $I \in \mathcal{G'}$ with the property that any topset
$T \subseteq \mathcal{G'}$ that contains $I$ has $\sum_{I \in T}A_I
> 0$.  If such an $I$ exists, we can delete a row associated to $I$
(thus diminishing $r_I$, and therefore also $A_I$, by 1) and
\eqref{cur} will remain true. The assertion is
that such a minimal multi-index $I$ can always be found.\\

To prove the assertion, we argue by contradiction.  Assume there are
$m$ minimal elements $I_1,...,I_m$ of $\mathcal{G'}$ and that each
$I_i$ lies in a topset $T_i$ for which $\sum_{I \in T_i}A_I = 0$. We
claim $\sum_{I \in T_1 \cup ... \cup T_m}A_I = 0$.  But by Lemma
\ref{topset_lem}(5), $T_1 \cup ... \cup T_m = \mathcal{G'}$, and we
are assuming $\sum_{I \in \mathcal{G'}}A_I
> 0$.  This contradiction proves the theorem, once the claim is
established.\\

To establish the claim, we prove by induction on $p$ the statement
that\\ $\sum_{I \in T_1 \cup ... \cup T_p}A_I = 0$.  For $p=1$, this
is true because we have assumed $\sum_{I \in T_1}A_I = 0$. For the
induction step, assume $\sum_{I \in T_1 \cup ... \cup T_{p-1}}A_I =
0$.  Write $X := T_1 \cup ... \cup T_{p-1}$, and observe that $X$ is
a topset by Lemma \ref{topset_lem}(2). Then $X \cap T_p$ and $X \cup
T_p$ are also topsets, again by Lemma \ref{topset_lem}(2), and
\begin{equation*} \sum_{I \in X \cap
T_p}A_I + \sum_{I \in X \cup T_p}A_I = \sum_{I \in X}A_I + \sum_{I
\in T_p}A_I = 0 + 0 = 0.
\end{equation*}
This forces $\sum_{I \in X \cup T_p}A_I= 0$, since both terms on the
left are non-negative.  But of course $X \cup T_p = T_1 \cup ...
\cup T_p$.
\end{proof}

We collect several results together into one theorem.

\begin{thm}{} \label{omnibus}
Let the family of vector subspaces $\mathcal{E} := \langle
f_1,...,f_s\rangle  \subseteq \mathcal{W}_{{\mathcal{M}}_Q(j)}
\subseteq \mathcal{D}_j$ be defined as above and let $U$ be the
$e^{th}$ cropped matrix of $\mathcal{E}$.  If $U$ has at least as
many rows as columns, or more generally if $U$ has maximal rank,
then for general $C \in k^{sm_Q(j)}$, $h_{\mathcal{E}(C)}(d) =
\rank(U) = \dim_kR_e*\mathcal{E}(C)$.
\end{thm}
\begin{proof}
By Theorem \ref{max_rank}, $U$ having at least as many rows as
columns guarantees that $U$ has maximal rank.\\

In any event, $U$ is an L-matrix by Lemma \ref{is_l_too}.  By Lemma
\ref{general_max}, $U(C)$ has maximal rank = $\rank(U)$ for general
$C \in k^{sm_Q(j)}$, which is the same as $\dim_kR_e*\mathcal{E}(C)$
by Corollary \ref{crop_rank_c}.  Finally, by Lemma \ref{hilb_char},
this is the same as $h_{\mathcal{E}(C)}(d)$.
\end{proof}

\begin{cor} \label{full_type}
Let the family of vector subspaces \\
$\mathcal{E} := \langle f_1,...,f_s\rangle  \subseteq
\mathcal{W}_{{\mathcal{M}}_Q(j)} \subseteq \mathcal{D}_j$ be defined
as above, where $s \le m_Q(j)$.  Then for general $C \in
k^{sm_Q(j)},\dim_k\mathcal{E}(C) = s$.
\end{cor}
\begin{proof}
Let $N := m_Q(j)$ and let $\mathcal{E}' := \langle
f_1,...,f_N\rangle \subseteq \mathcal{W}_{{\mathcal{M}}_Q(j)}
\subseteq \mathcal{D}_j$. Then setting $e = 0$, the $e^{th}$ cropped
matrix $U$ of $\mathcal{E}'$ is $N \times N$, and we apply Theorem
\ref{omnibus}. We find that, for general $C \in k^{N^2}$,
$\dim_k\mathcal{E}'(C) := \dim_kR_0*\mathcal{E}'(C) = \rank(U)= N$;
thus $f_1(C),...,f_N(C)$ are linearly independent, and perforce
$f_1(C),...,f_s(C)$ are also linearly independent.  Let $V \subseteq
k^{N^2}$ be the Zariski-open dense set on which $f_1(C),...,f_N(C)$
are linearly independent. Then, as a subset of $k^{sN}$, $V \cap
k^{sN}$ is Zariski-open; and it is nonempty, thus dense.
Equivalently, for general $C \in k^{sN}$, $f_1(C),...,f_s(C)$ are
linearly independent, and $\dim_k\mathcal{E}(C) = s$.
\end{proof}

\section{Intersections of Subspaces of $\mathcal{D}_j$}

For the theorem proved in this section, we select a notation that
will be convenient later. We define, as above, a family of vector
subspaces $\mathcal{F} := \langle g_1,...,g_u\rangle  \subseteq
\mathcal{W}_{{\mathcal{M}}_Q(j)} \subseteq \mathcal{D}_j$. We let
$U$ be the $e^{th}$ cropped matrix of $\mathcal{F}$ and let $T$ be a
submatrix of $U$. Recalling that the columns of $U$ are indexed by
$\mathcal{M}_Q(d)$, let $\mathcal{M}_T \subseteq \mathcal{M}_Q(d)$
be the set of all column indices in $T$, and let $\mathcal{M}_{T^c}
\subseteq \mathcal{M}_Q(d)$ be the set of all column indices not in
$T$.  We define, for use in this section and in later sections,
\begin{equation} \label{vlabel}
 \mathcal{V}_{\mathcal{M}_Q(d)}:= \langle \{x^D | D \in
\mathcal{M}_Q(d)\}\rangle \subseteq\mathcal{D}_d.
\end{equation}
\noindent We remark a peculiarity of the notation, namely, that
$\mathcal{V}_{\mathcal{M}_Q(j)}$ is the same as
$\mathcal{W}_{\mathcal{M}_Q(j)}$.  The difference in notation
highlights the distinction that $\mathcal{W}_{\mathcal{M}_Q(j)}$ is
a vector subspace of $\mathcal{D}_j$, whereas
$\mathcal{V}_{\mathcal{M}_Q(d)}$ is a vector subspace of
$\mathcal{D}_d$.  We also define
\begin{equation*}
 \mathcal{V}_{\mathcal{M}_T}:= \langle \{x^D | D \in
\mathcal{M}_T\}\rangle  \subseteq\mathcal{D}_d.
\end{equation*}
\noindent and
\begin{equation} \label{vtclabel}
\mathcal{V}_{\mathcal{M}_{T^c}}:= \langle \{x^D | D \in
\mathcal{M}_{T^c}\}\rangle  \subseteq\mathcal{D}_d.
\end{equation}

\begin{thm}{} \label{intersect}
Let the family of vector subspaces $\mathcal{F} := \langle
g_1,...,g_u\rangle  \subseteq \mathcal{W}_{{\mathcal{M}}_Q(j)}
\subseteq \mathcal{D}_j$ be defined as above and let $U$ be the
$e^{th}$ coefficient matrix of $\mathcal{F}$. Assume that $U$ is of
dimension $p \times q$ and that $T$ is a $t \times t$ square
submatrix of $U$ whose determinant is nonzero.  Let $\mathcal{M}'$
be the set of all row indices $(E,i)$ of $T$. Let $\mathcal{W}
\subseteq \mathcal{D}_d$ be a vector subspace such that
$(R_e*\mathcal{F}(C)) \cap \mathcal{W} \subseteq
\mathcal{V}_{\mathcal{M}_{T^c}}$ for all $C \in k^{um_Q(j)}$. Then
\begin{enumerate}
\item[(i)] For general $C \in k^{um_Q(j)},
\langle \{X^E*g_i(C)| (E,i) \in \mathcal{M}'\}\rangle \cap
\mathcal{V}_{\mathcal{M}_{T^c}} = \{0\}.$
\item[(ii)] If $q \ge p =t$, then for general $C \in k^{um_Q(j)}$,
$(R_e*\mathcal{F}(C)) \cap \mathcal{V}_{\mathcal{M}_{T^c}} = \{0\}$
and $(R_e*\mathcal{F}(C)) \cap \mathcal{W} = \{0\}.$
\end{enumerate}
\end{thm}

\begin{proof}
We can express $\mathcal{V}_{\mathcal{M}_Q(d)}$ as an internal
direct sum:
\begin{equation*}
\mathcal{V}_{\mathcal{M}_Q(d)} =\mathcal{V}_{\mathcal{M}_T}
\bigoplus \mathcal{V}_{\mathcal{M}_{T^c}}.
\end{equation*}

If we now focus on a particular $X^E*g_i(C) \in
\mathcal{V}_{\mathcal{M}_Q(d)} \subseteq \mathcal{D}_d$, we have

\begin{align*}
 X^E*g_i(C) &=\sum_{D\in \mathcal{M}_Q(d)} u_{(E,i)D}(C)x^D\\
&= \sum_{D\in \mathcal{M}_T} u_{(E,i)D}(C)x^D + \sum_{D\in
\mathcal{M}_{T^c}}
u_{(E,i)D}(C)x^D\\
&\subseteq \mathcal{V}_{\mathcal{M}_T} \bigoplus
\mathcal{V}_{\mathcal{M}_{T^c}},
\end{align*}
\noindent and we observe from Definition \ref{cropped} that $u_{(E,i)D}$ is the
entry of $U$ appearing in the $( (E,i),D )$ position. For a linear combination
$\sum_{(E,i) \in \mathcal{M}'} a_{Ei}X^E*g_i(C)$, we have

\begin{align*}
\sum_{(E,i) \in \mathcal{M}'} a_{Ei}X^E*g_i(C) &=\sum_{(E,i) \in
\mathcal{M}'} a_{Ei}\sum_{D\in \mathcal{M}_Q(d)} u_{(E,i)D}(C)x^D\\
 =\sum_{(E,i) \in \mathcal{M}'} a_{Ei}\sum_{D\in \mathcal{M}_T}
u_{(E,i)D}(C)x^D &+ \sum_{(E,i) \in \mathcal{M}'} a_{Ei}\sum_{D\in
\mathcal{M}_{T^c}}
u_{(E,i)D}(C)x^D\\
 &\subseteq \mathcal{V}_{\mathcal{M}_T} \bigoplus
\mathcal{V}_{\mathcal{M}_{T^c}}.
\end{align*}

The non-vanishing of $\det(T)$ (as a polynomial in the coefficients
$z_{iJ}$ of the $g_i$'s), being a Zariski-open condition, guarantees
that the row vectors \\
$\{\sum_{D \in \mathcal{M}_T}u_{(E,i)D}(C)x^D|(E,i) \in
\mathcal{M}'\}$ of $T$ are linearly independent for general $C \in
k^{um_Q(j)}$.  For such a choice of $C$, the linear combination \\
$\sum_{(E,i) \in \mathcal{M}'} a_{Ei}\sum_{D \in
\mathcal{M}_T}u_{(E,i)D}(C)x^D \in \mathcal{V}_{\mathcal{M}_T}$ is
never 0 unless all of the $a_{Ei}$'s are 0, in which case
$\sum_{(E,i) \in \mathcal{M}'} a_{Ei}X^E*g_i(C) = 0 $.  This gives
$\langle \{X^E*g_i(C)| (E,i) \in \mathcal{M}'\}\rangle  \cap
\mathcal{V}_{\mathcal{M}_{T^c}} = \{0\}$,
which proves (i).\\

For (ii), we are considering the special case that $q \ge p =t$,
which is to say that $U$ has at least as many columns as rows, and
that $T$ is a maximal square submatrix.  In this case,
$\mathcal{M}'$ comprises all the rows of $U$, which by construction
are indexed by $\{(E,i) | E \in \mathcal{M}_Q(e)\}$. By Lemma
\ref{gen}, $R_e*\mathcal{F}(C) = \langle \{X^E*g_i(C) | E$ is of
degree $e\}\rangle $, and we have seen in Lemma \ref{crop} that
nothing is lost by considering only those $E$ that lie in
$\mathcal{M}_Q(e)$. Thus $R_e*\mathcal{F}(C) = \langle \{X^E*g_i(C)
| (E,i) \in \mathcal{M}'\}\rangle $, and the first statement of (ii)
follows
from (i).\\

Finally, for general $C \in k^{um_Q(j)}$, we have assumed
$(R_e*\mathcal{F}(C)) \cap \mathcal{W} \subseteq
\mathcal{V}_{\mathcal{M}_{T^c}}$, and of course
$(R_e*\mathcal{F}(C)) \cap \mathcal{W} \subseteq
R_e*\mathcal{F}(C)$.  Thus\\ $(R_e*\mathcal{F}(C)) \cap \mathcal{W}
\subseteq (R_e*\mathcal{F}(C)) \cap \mathcal{V}_{\mathcal{M}_{T^c}}
= \{0\}$.
\end{proof}

\chapter{Special Cases for Interesting Choices of Q}

In this chapter we examine some special cases that result from
particular choices of constraints $Q := (Q_1,...,Q_n)$, some of
which will be used later to construct non-unimodal level algebras.
In the first three sections of this chapter, the following outline
will be followed. We assume that a choice of codimension $r$ and
socle degree $j$ has been made, and we state the constraint $Q :=
(Q_1,...,Q_n)$ that is to be studied in the section. We assume that
a degree $d$ has been chosen and we set $e = j-d$.  We study the
situation that some number $u$ of polynomials $g_1,...,g_u$ have
been selected from $\mathcal{D}_j$ to generate a vector subspace
$\langle g_1,...,g_u\rangle \subseteq \mathcal{D}_j$, subject to the
condition that all monomials appearing in these generators are to be
constrained by $Q$. That is, we consider the family of subspaces
$\mathcal{F} := \langle g_1,...,g_u\rangle \subseteq
\mathcal{W}_{\mathcal{M}_Q(j)} \subseteq \mathcal{D}_j$,
parameterized by elements $C' = (...,c_{iJ}',...) \in k^{um_Q(j)}$.
We remark that the choice of notation has been influenced by
context: we will be applying these results in a context where we
have already defined, for some other constraint $P$, a family
$\mathcal{E} := \langle f_1,...,f_s\rangle \subseteq
\mathcal{W}_{\mathcal{M}_P(j)} \subseteq \mathcal{D}_j$,
parameterized by elements $C = (...,c_{iJ},...) \in k^{sm_P(j)}$. We
let $U$ denote the $e^{th}$ cropped matrix of $\mathcal{F}$.

\section{Absence of Constraints}

If we set $n = r$ and $Q_1 = ... = Q_r = j$, there are no actual
constraints imposed. For any degree $d$, we have
$\mathcal{V}_{\mathcal{M}_Q(d)} =\mathcal{D}_d$ (recalling the
definition in $\eqref{vlabel} )$; and in particular
$\mathcal{W}_{\mathcal{M}_Q(j)} =\mathcal{D}_j$. We can easily count
$m_Q(d):= \#(\mathcal{M}_Q(d))$ as the number of monomials of degree
$d$ in $r$ variables, namely, $\genfrac(){0cm}{0}{d+r-1}{r-1}$.\\

With the results obtained so far, we are in a position to prove one
of the two
theorems of A. Iarrobino quoted earlier, although we now restate it slightly
different language.\\

\begin{thm}{}
Consider the family of vector subspaces $\mathcal{F}:= \langle
g_1\rangle \subseteq \mathcal{D}_j$. Then for general $C' \in
k^{m_Q(j)}$
\begin{equation*}
h_{\mathcal{F}(C')}(d) = min( \dim_kR_{j-d}, \dim_k\mathcal{D}_d).
\end{equation*}
\end{thm}
\begin{proof}
Given $d$, we set $e := j-d$ and we construct $U$, the $e^{th}$ cropped matrix of
$\mathcal{F}$, which by Lemma \ref{is_l_too} is an L-matrix.  We remark that all
entries of $U$ are nonzero, since for any two multi-indexes $E$ of degree $e$ and
$D$ of degree $d$, $E + D \in \mathcal{M}_Q(j)$. Thus, by Lemma \ref{pv_1} every
square submatrix of $U$ has nonzero determinant, and $U$ has maximal rank. That is,
its rank is either $m_Q(e) = \dim_kR_{e}$, the number of rows, or $m_Q(d)
=\dim_k\mathcal{D}_d$, the number of columns, whichever is smaller; equivalently,
the rank is $min( \dim_kR_{j-d}, \dim_k\mathcal{D}_d)$.  By Theorem \ref{omnibus},
for general $C' \in k^{m_Q(j)}$,
this is the same as $h_{\mathcal{F}(C')}(d)$. \\
\end{proof}

\section{$k[x_1,...,x_m]_j$}

In this section we choose $m$ such that $r \ge m \ge 3$, and
consider the constraint $Q:= (j,...,j,0,...,0)$ of dimension $r$, in
which the first $m$ constraints are $j$ and the remaining $r-m$
constraints are 0.  We do not exclude the possibility that $m = r$.

\begin{prop} \label{k_3}
Let $r =n \ge m \ge 3$, fix a socle degree $j$, and consider the
constraint $Q:= (j,...,j,0,...,0)$ in which the first $m$
constraints are $j$ and the remaining $r-m$ constraints are $0$. Let
$p := m-1$.  Fix a degree $d$, and let $e := j-d$. Let $U$ be the
$e^{th}$ cropped matrix of the family of vector subspaces
$\mathcal{F} := \langle g_1,...g_u\rangle  \subseteq
\mathcal{W}_{\mathcal{M}_Q(j)} = k[x_1,...,x_m]_j \subseteq
\mathcal{D}_j$, where $u$ is chosen such that
$u\genfrac(){0cm}{0}{e+p}{p} \le \genfrac(){0cm}{0}{d+p}{p}$.  Then
\begin{enumerate}
\item [(i)] $U$ is a $u\genfrac(){0cm}{0}{e+p}{p} \times
\genfrac(){0cm}{0}{d+p}{p}$ matrix with at least as many columns as
rows, all of whose entries are nonzero.
\item [(ii)] $U$ is of maximal rank $u\genfrac(){0cm}{0}{e+p}{p}$.  For general
$C' \in k^{um_Q(j)},\\ \dim_kR_e*\mathcal{F}(C') =
u\genfrac(){0cm}{0}{e+p}{p}$.
\end{enumerate}
If also $\mathcal{W} \subseteq \mathcal{D}_d$ is a vector subspace
for which $\mathcal{W} \cap k[x_1,...,x_m]_d \subseteq \mathcal{Z}$,
where $\mathcal{Z} \subseteq k[x_1,...,x_m]_d$ is a vector subspace,
generated by monomials, such that\\ $\dim_k\mathcal{Z}
 \le \genfrac(){0cm}{0}{d+p}{p} - u\genfrac(){0cm}{0}{e+p}{p}$, then
\begin{enumerate}
\item [(iii)] For general $C' \in k^{um_Q(j)}$,
$\mathcal{W} \cap (R_e*\mathcal{F}(C')) = \{0\}$.
\end{enumerate}
\end{prop}
\begin{proof}
Since $Q_1 = ... = Q_m = j$, no effective constraint is placed on
the first $m$ variables $x_1,...,x_m$.  Since $Q_{m+1} = ...= Q_n =
0$, no other variables are allowed to appear at all.  So for any
degree $d, \mathcal{V}_{\mathcal{M}_Q(d)}$ (defined in
$\eqref{vlabel} )$ is spanned by all monomials (of degree d) in
which no variables other than $x_1,...,x_m$ appear, of which there
are $\genfrac(){0cm}{0}{d+p}{p}$. If we regard $k[x_1,...,x_m]_d$ as
a vector subspace of $\mathcal{D}_d$,
we have $\mathcal{V}_{\mathcal{M}_Q(d)} = k[x_1,...,x_m]_d$.\\

To show (i): By definition, $U$ has $um_Q(e) =
u\genfrac(){0cm}{0}{e+p}{p}$ rows and $m_Q(d) =
\genfrac(){0cm}{0}{d+p}{p}$ columns.  Since
$u\genfrac(){0cm}{0}{e+p}{p} \le \genfrac(){0cm}{0}{d+p}{p}$, $U$
has at least as many columns as rows. By construction, the
$((E,i),D)$ row has a nonzero entry whenever $E + D \in
\mathcal{M}_Q(j)$; this always happens because if the only nonzero
co-ordinates of $E$ and $D$ occur among the first $m$, the same is
true for $E + D$.\\

For (ii), we apply Lemma \ref{is_l_too} to show $U$ is an L-matrix,
and then Lemma \ref{pv_1} to show $U$ has maximal rank.  In fact,
since all entries of $U$ are nonzero, any maximal square submatrix
of $U$ has nonzero determinant, and $U$ has maximal rank.  This rank
is of course $u\genfrac(){0cm}{0}{e+p}{p}$, the number of
rows.\\

For (iii), we recall the definition of $\mathcal{V}_{\mathcal{M}_{T^c}}$ from
$\eqref{vtclabel}$ and we seek to apply Theorem \ref{intersect}(ii). In order to do
so, we must find a maximal square submatrix $T$ of $U$ whose determinant is nonzero
and whose columns are indexed by multi-indexes $D \in \mathcal{M}_Q(d)$ for which
the monomial $x^D \in k[x_1,...,x_m]_d$ is not among the monomial generators of
$\mathcal{Z}$.  This would ensure that $\mathcal{Z} \subseteq
\mathcal{V}_{\mathcal{M}_{T^c}}$. Also, the hypothesis that $\mathcal{W} \cap
k[x_1,...,x_m]_d \subseteq \mathcal{Z}$ ensures that, for all $C' \in k^{um_Q(j)}$,
$\mathcal{W} \cap R_e*\mathcal{F}(C') \subseteq \mathcal{Z} \subseteq
\mathcal{V}_{\mathcal{M}_{T^c}}$ . Thus, provided a suitable $T$ can be found, the
conditions of Theorem \ref{intersect} are met, and we conclude that $\mathcal{W}
\cap (R_e*\mathcal{F}(C'))
= \{0\}$ for general $C' \in k^{um_Q(j)}$.\\

Finding $T$ is easy: create $T$ from \underline{any} $u\genfrac(){0cm}{0}{e+p}{p}$
columns corresponding to indices $D$ for which $x^D$ is not among the generators of
$\mathcal{Z}$.  This can be done because we have assumed
$u\genfrac(){0cm}{0}{e+p}{p} \le \genfrac(){0cm}{0}{d+p}{p} - \dim_k\mathcal{Z}$.
As remarked above, any maximal square submatrix of $U$ has nonzero determinant, so
in particular $\det(T)$ is nonzero.

\end{proof}

\section{Q = (1)}

In this section, we consider the case that a single variable $x_1$
is constrained so that if it appears in any term, it does so with
exponent 1.

\begin{prop} \label{sf_1}
Let $n =1$ and let $Q := (1)$ for some socle degree $j$. Fix a
degree $d$, let $e := j-d$, and assume $d \ge e \ge 1$. Let $U$ be
the $e^{th}$ cropped matrix of $\mathcal{F} := \langle g_1\rangle
\subseteq \mathcal{W}_{\mathcal{M}_Q(j)} \subseteq \mathcal{D}_j$.
Then
\begin{enumerate}
\item [(i)] $U$ is a block matrix of the form

\begin{equation*}
\begin{artmatrix}
&c_0&c_1\cr \hline r_1&0&A\cr r_0&B&C
\end{artmatrix}
\end{equation*}

\noindent where $0$ denotes a block of zeroes and blocks $A, B$, and
$C$ consist entirely of nonzero entries.  The dimensions of the
blocks are
\begin{align*}
&r_1 = \genfrac(){0cm}{0}{(e-1)+(r-2)}{r-2}.\\
&r_0 = \genfrac(){0cm}{0}{e+(r-2)}{r-2}.\\
&c_0 = \genfrac(){0cm}{0}{(d-1)+(r-2)}{r-2}.\\
&c_1 = \genfrac(){0cm}{0}{d+(r-2)}{r-2}.
\end{align*}
\item[(ii)] $U$ has maximal rank.
\end{enumerate}

If also $\mathcal{W} \subseteq \mathcal{D}_d$ is a vector subspace
for which $\mathcal{W} \cap \mathcal{V}_{\mathcal{M}_Q(d)} \subseteq
\mathcal{Z}$, where $\mathcal{Z} \subseteq
\mathcal{V}_{\mathcal{M}_Q(d)}$ is a vector subspace generated by
$2c$ monomials, in $c$ of which $x_1$ appears (with exponent 1) and
in $c$ of which $x_1$ does not appear, then
\begin{enumerate}
\item[(iii)] $ \mathcal{W} \cap (R_e*\mathcal{F}(C')) = \{0\}$ for general $C' \in
k^{m_Q(j)}$ if both
\begin{equation} \label{sf_11}
m_Q(d) - 2c \ge m_Q(e)
\end{equation}
\noindent and
\begin{equation} \label{sf_12}
 \genfrac(){0cm}{0}{d+(r-2)}{r-2} - c \ge
\genfrac(){0cm}{0}{(e-1)+(r-2)}{r-2}.
\end{equation}
\end{enumerate}
\end{prop}
\begin{proof}
For (i), we apply Lemma \ref{crop_l} to establish that $U$ is an
L-matrix with $\mathcal{G}_Q$ pattern.  Since $\mathcal{G}_Q =
\{(0),(1)\}$, which for simplicity we will write as $\{0,1\}$, the
order of rows, which must be lexicographic, is 1 then 0, and the
order of columns, which must be reverse lexicographic, is 0 then 1.
The entries of block $B_{IJ}$ are zero if and only if $I \nsucceq
J$, that is, only when $I = 1$ and
$J = 0$.\\

To establish the dimensions of the blocks, we use the formulas in
Proposition \ref{u_count}, with $s = 1, p_{0} = 0, p_{1} = 1$, and
$q =
1$.\\

For (ii):  From the formulas in (i) and the hypothesis that $d \ge
e$, we see that $U$ has at least as many columns as rows, so it
suffices to show that the rightmost square submatrix $T$ has nonzero
determinant.  If $T$ contains no entries from the 0 block, its
entries are all nonzero, and $\det(T)$ is nonzero by Lemma
\ref{pv_1}.  Otherwise, $T$ contains blocks $A$ and $C$ in their
entirety, and has the form

\begin{equation*}
\begin{artmatrix}
&c_0'&c_1\cr \hline r_1&0'&A\cr r_0&B'&C
\end{artmatrix}
\end{equation*}

By Theorem \ref{gq_3}, $T$ will be nonsingular if, for every nonempty proper
bottomset $B \subseteq \mathcal{G}_Q$, $\sum_{I \in B}A_I \le 0$, where $A_I$ is the
excess $r_I - c_I$.  Since $\mathcal{G}_Q = \{0,1\}$, its only nonempty proper
bottomset is $\{1\}$, so the condition reduces to $A_1 \le 0$, that is, $r_1 \le
c_1$.  This last condition follows from the
formulas in (i) because we have assumed $d \ge e$.\\

For (iii), we apply Theorem \ref{intersect}(ii).  To do so, we must
construct a square submatrix $T$ of $U$ with nonzero determinant,
such that the $2c$ monomial generators of $\mathcal{Z}$ lie in
$\mathcal{V}_{T^c}$.  Assuming this, the hypothesis that
$\mathcal{W} \cap \mathcal{V}_{\mathcal{M}_Q(d)} \subseteq
\mathcal{Z}$ ensures that $\mathcal{W} \cap R_e*\mathcal{F}(C')
\subseteq \mathcal{Z} \subseteq \mathcal{V}_{\mathcal{M}_{T^c}}$ for
all $C' \in k^{m_Q(j)}$. Thus, provided a suitable $T$ can be found,
the conditions of Theorem \ref{intersect} are met, and we conclude
that $\mathcal{W} \cap (R_e*\mathcal{F}(C'))
 = \{0\}$ for general $C' \in k^{m_Q(j)}$.\\

So we ask under what circumstances a suitable submatrix $T$ of $U$
can be found.  One requirement is that $U$ have enough columns so
that, when $2c$ of them are not used, there are still enough columns
left to form a square $m_Q(e) \times m_Q(e)$ submatrix $T$.  That
is, we require $m_Q(d) - 2c \ge m_Q(e)$.  Assuming this, we must
still ask whether we can find a suitable submatrix $T$ whose
determinant is nonzero. To this end, we delete from $U$ the $2c$
columns corresponding to the generators of $\mathcal{Z}$, to obtain
a submatrix $U_0$ of the form

\begin{equation*}
\begin{artmatrix}
&c_0-c&c_1-c\cr \hline r_1&0'&A'\cr r_0&B'&C'
\end{artmatrix}
\end{equation*}

\noindent and we argue as in part (ii):  If the rightmost square
submatrix $T_0$ of $U_0$ has no entries from the $0'$ block,
$\det(T_0)$ is nonzero by Lemma \ref{pv_1}.  Otherwise, the
condition from Theorem \ref{gq_3} is that $c_1 - c \ge r_1$, or
equivalently\\ $\genfrac(){0cm}{0}{d+(r-2)}{r-2} - c \ge
\genfrac(){0cm}{0}{(e-1)+(r-2)}{r-2}$.
\end{proof}

The following special case will be of interest later.

\begin{cor} \label{vspec_4}
Let $r=4$, $n =1$ and $Q := (1)$ for some socle degree $j$. Fix a
degree $d$, let $e := j-d$, and assume $d \ge e \ge 1$. Let $U$ be
the $e^{th}$ cropped matrix of the family of vector subspaces
$\mathcal{F} := \langle g_1\rangle  \subseteq
\mathcal{W}_{\mathcal{M}_Q(j)} \subseteq \mathcal{D}_j$. Then
\begin{enumerate}
\item[(i)] $U$ has maximal rank $(e+1)^2$.
\item[(ii)] For general $C' \in k^{m_Q(j)},
\dim_kR_e*\mathcal{F}(C') = (e+1)^2$.
\end{enumerate}
\end{cor}
\begin{proof}
We use Proposition \ref{sf_1} for the case that $r=4$.\\

For (i): $U$ is of maximal rank, which is the number of rows.
Substituting $r=4$ into the formulas for the number of rows in each
block gives
\begin{align*}
r_1 + r_0 &= \genfrac(){0cm}{0}{e+1}{2} + \genfrac(){0cm}{0}{e+2}{2}\\
&= \genfrac{}{}{}{0}{e(e+1)}{2} + \genfrac{}{}{}{0}{(e+2)(e+1)}{2}\\
&=(2e+2)\genfrac{}{}{}{0}{(e+1)}{2}\\
&= (e+1)^2.
\end{align*}

For (ii), $U(C')$ has maximal rank for general $C' \in k^{m_Q(j)}$
by Lemma \ref{general_max},
and this rank is $\dim_kR_e*\mathcal{F}(C')$ by Corollary \ref{crop_rank_c}.\\

\end{proof}

\section{Essentially n-fold-constrained}

In this section, we find it convenient to assume that exactly $n \ge
1$ of the variables are constrained to have less than the full range
of exponents.  In this case, we will say that, for any degree $d$,
the multi-indexes in $\mathcal{M}_Q(d)$ and their corresponding
monomials are \emph{essentially n-fold-constrained}, or simply
\emph{n-fold-constrained}.\\

We wish to compute $m_Q(d)$, or equivalently the dimension of the vector space
generated by monomials of degree $d$ constrained by $Q$.  For the purposes of this
computation, we may as well assume that the constrained variables are listed first;
that is, we assume $Q$ is a constraint of dimension $n \ge 1$ where $Q_i < j$ for $i
= 1,...,n$. The following lemma makes this precise.

\begin{lem} \label{permute}
Fix codimension $r$ and socle degree $j$, and let $Q :=
(Q_1,...,Q_m)$ be a constraint of dimension $m$, such that
$Q_{i_1},...,Q_{i_n}$ are all strictly less than $j$ and the rest of
the $Q_i$'s are equal to $j$.  Let $P :=
(Q_{i_1},...,Q_{i_n},j,...j) := (Q_{\sigma(1)},...,Q_{\sigma(m)})$
be another constraint of dimension $m$ whose entries are related to
those of $Q$ via some permutation $\sigma$ of $\{1,...,m\}$, such
that the entries equal to $j$ all come last.  Then for any degree
$d$, $m_Q(d) = m_P(d)$.
\end{lem}
\begin{proof}
We define a function $b_{\sigma}$, evidently a bijection, from the
set of all $r$-tuples of degree $d$ to itself, induced by $\sigma$,
as follows. If $D := (d_1,...,d_r)$, then\\ $b_{\sigma}(D):=$
$(d_{\sigma(1)},...,d_{\sigma(m)},d_{m+1},...,d_r)$.\\

We now claim that a multi-index $D := (d_1,...,d_r)$ of degree $d$
lies in $\mathcal{M}_Q(d)$ if and only if the multi-index
$b_{\sigma}(D)$ lies in $\mathcal{M}_P(d)$: the first condition is
that, for each $i = 1,...,m$, $d_i \le Q_i$; the second is that, for
each $i = 1,...,m$,  $d_{\sigma(i)} \le P_i = Q_{\sigma(i)}$, which
is the same as the first condition because $\{1,...,m\} =
\{\sigma(1),...,\sigma(m)\}$.\\

Since $b_{\sigma}$ is a bijection, its restriction to
$\mathcal{M}_Q(d)$ is a bijection from $\mathcal{M}_Q(d)$ to its
image $\mathcal{M}_P(d)$.
\end{proof}

\begin{lem} \label{h_inc}
Fix codimension $r$ and socle degree $j$, and let $Q :=
(Q_1,...,Q_n)$ be a constraint on $r$-tuples that fails to constrain
the value of at least one co-ordinate.  Then the function $m_Q(d)$
is a non-decreasing function of $d$ for $0 \le d \le j$.
\end{lem}
\begin{proof}
We must show that if $d_1 < d_2$, then $m_Q(d_1):= \#(\mathcal{M}_Q(d_1)) \le
m_Q(d_2) := \#(\mathcal{M}_Q(d_2))$. Let the value of the $i^{th}$ coordinate not be
constrained.  Then for any element $I = (I_1,...,I_{i-1},I_i,I_{i+1},...,I_r) \in
\mathcal{M}_Q(d_1)$, there is a corresponding element $(I_1,...,I_{i-1},I_i + (d_2 -
d_1),I_{i+1},...,I_r)  \in \mathcal{M}_Q(d_2)$.
\end{proof}

Getting closed formulas of a simple form will sometimes not be
possible for some values of $d$, but we will typically find that
patterns emerge when $d$ is large enough. We classify the results by
$n$, the number of constraints.  In order to state the results more
concisely, we define $q := Q_1 + ... + Q_n$ and $a_i := Q_i + 1$ for
$i = 1,...,n$.  We denote a multi-index of degree $d$ constrained by
$Q$ as $D := (d_1,...,d_r)$.\\

\subsection{$r$-fold-constrained}

\begin{prop} \label{r_fold}
When multi-indexes are $r$-fold constrained by $Q$, $m_Q(d) = 0$ for
$d > q$.
\end{prop}
\begin{proof}
The largest possible degree of any multi-index is $q := Q_1 + ... +
Q_r$.  So for $d > q$, no multi-indexes are possible.
\end{proof}

\subsection{$(r-1)$-fold-constrained}

\begin{prop} \label{r-1_fold}
When multi-indexes are $(r-1)$-fold constrained by $Q$,\\ $m_Q(d) =
a_1a_2 \cdots a_{r-1}$ for $d \ge q$.
\end{prop}

\begin{proof}
All but one of the variables are constrained, and the last is
allowed to vary.  Thus, for any degree $d \ge q$, any choice of
$(d_1,...,d_{r-1})$ can be augmented by $d_r$ in a unique way to
create a monomial of degree $d$.  Since each $d_i$ can be chosen in
$Q_i + 1 = a_i$ ways, we have $m_Q(d) = a_1a_2 \cdots a_{r-1}$.
\end{proof}

\subsection{$(r-2)$-fold-constrained}

\begin{prop} \label{r-2_fold}
Let $r \ge 3$ and let multi-indexes be $(r-2)$-fold constrained by
$Q$. Let $S_n = \sum_{1 \le i \le n}a_i$ and $P_n = \prod_{1 \le i
\le n}a_i$. Then for $d \ge q$, $m_Q(d) = \genfrac{}{}{}{0}{P_n(2d -
S_n + r)}{2}.$
\end{prop}
\begin{proof}
We proceed by induction on $r$, starting with $r = 3, n = 1$.  For
the initial case, we are counting multi-indexes $D =(d_1,d_2,d_3)$
such that $d_1 \le Q_1$.  For any choice of $d_1$, we can complete
$D$ by choosing values of $d_2$ and $d_3$ whose sum is $d - d_1$,
and there are $\genfrac(){0cm}{0}{(d-d_1)+(2-1)}{2-1} = d - d_1 +1$
ways to do it. So for $d \ge q = Q_1$,
\begin{align*}
m_Q(d) &= \sum_{0 \le d_1 \le Q_1}(d - d_1 + 1)\\
&=\sum_{0 \le d_1 \le Q_1}(d+ 1) - \sum_{0 \le d_1 \le Q_1}d_1\\
&=a_1(d+1) - a_1(a_1-1)/2\\
&=\genfrac{}{}{}{0}{a_1(2d - a_1 + 3)}{2}.
\end{align*}
In the induction step, we assume the proposition has been proved for
$r-1$ and we prove it for $r$.  For a fixed choice of $d_n$, we ask
how many choices of\\ $(d_1,...,d_{n-1},d_{n+1},d_r)$ are
permissible. This amounts to asking the value of\\ $m_{Q'}(d-d_n)$,
for dimension $r' = r-1$ and constraint $Q' = (Q_1,...,Q_{n-1})$,
which we claim satisfies the hypothesis of the proposition:
$r'$-tuples are $(r'-2)$-fold constrained by $Q'$; and we have $d_1
+ ... + d_{n-1} \ge Q_1 + ... + Q_{n-1}$ since $d \ge q$ and $d_n
\le Q_n$. Applying the induction hypothesis,
\begin{align*}
m_Q(d) &= \sum_{0 \le d_n \le Q_n}m_{Q'}(d-d_n)\\
&=\sum_{0 \le d_n \le Q_n}\genfrac{}{}{}{0}{P_{n-1}[2(d-d_n) -
S_{n-1} +
(r-1)]}{2}\\
&=a_n \genfrac{}{}{}{0}{P_{n-1}[2d - S_{n-1} + (r-1)]}{2} -
P_{n-1}\sum_{0 \le d_n \le
Q_n} d_n\\
&=a_n \genfrac{}{}{}{0}{P_{n-1}[2d - S_{n-1} + (r-1)]}{2} -
P_{n-1}\genfrac{}{}{}{0}{a_n(a_n-1)}{2}\\
&=a_n \genfrac{}{}{}{0}{P_{n-1}[2d - S_{n-1} + (r-1) - (a_n-1)]}{2}\\
&=a_n \genfrac{}{}{}{0}{P_{n-1}[2d - S_n + r]}{2}\\
&=\genfrac{}{}{}{0}{P_n[2d - S_n + r]}{2}.
\end{align*}
\end{proof}

For the cases $r=3$ and $r=4$, we will need to know
$\#(\mathcal{M}_Q(d))$ for smaller values of $d$ than those covered
by Proposition \ref{r-2_fold}.

\begin{prop} \label{spec_3}
Let $r= 3$ and let multi-indexes be once-constrained by\\ $Q =
(Q_1).$ Then
\begin{enumerate}
\item[(i)] $m_Q(d)= \genfrac{}{}{}{0}{a_1(2d-a_1+3)}{2}$
for $d$ $\ge$ $q-1 = a_1-2.$
\item[(ii)] $m_Q(d)= \genfrac{}{}{}{0}{a_1(2d-a_1+3)}{2} + 1$
for $d = q-2 = a_1-3.$
\item[(iii)] $m_Q(d)= \genfrac(){0cm}{0}{d+2}{2}$
for $d < a_1.$
\end{enumerate}
\end{prop}
\begin{proof}
We start with the formulas from Proposition \ref{r-2_fold},
\begin{equation*}
\genfrac{}{}{}{0}{a_1(2d-a_1+3)}{2} = m_Q(d) = \sum_{0 \le d_1 \le
Q_1}(d - d_1 + 1) \; \; \; \text{when} \; d \ge Q_1.
\end{equation*}
\noindent Viewing the left- and right-hand sides of this equation as
polynomials in $d$, we see they agree for the infinitely many
integer values of $d$ such that $d \ge Q_1$, and hence must agree
for all $d$.\\

We remark that, for values of $d < Q_1$, a valid expression for
$m_Q(d)$ can be obtained as before, by summing terms of the form $d
- d_1 + 1$, for a suitable range of values of $d_1$.  We now
investigate how to do this for $d = Q_1 -1 $ and $d = Q_1 -2$.\\

For $d = Q_1 -1$, the summation ends with $d_1 = Q_1 -1$.
Equivalently, one can take the previous summation and subtract the
term for $d_1 = Q_1$.  That is,
\begin{align*}
m_Q(d) &= \sum_{0 \le d_1 \le Q_1 - 1}(d - d_1 + 1)\\
&= \sum_{0 \le d_1
\le Q_1}(d - d_1 + 1) -
[(Q_1-1) -Q_1 + 1]\\
&=\genfrac{}{}{}{0}{a_1(2d-a_1+3)}{2} - [(Q_1-1) -Q_1 + 1]\\
&=\genfrac{}{}{}{0}{a_1(2d-a_1+3)}{2}.
\end{align*}

For $d = Q_1 -2$, the terms with $d_1$ having the values $Q_1$ and
$Q_1 -1$ must be omitted.  That is,
\begin{align*}
m_Q(d) &=\sum_{0 \le d_1 \le Q_1 -2}(d - d_1 + 1)\\
&= \sum_{0 \le d_1 \le Q_1}(d - d_1 + 1) -
[(Q_1-2) -Q_1 + 1] - [(Q_1-2) - (Q_1-1) + 1] \\
&=\genfrac{}{}{}{0}{a_1(2d-a_1+3)}{2} - [(Q_1-2) -Q_1 + 1] -
[(Q_1-2) - (Q_1-1) + 1]\\
&=\genfrac{}{}{}{0}{a_1(2d-a_1+3)}{2} + 1.
\end{align*}

When $d < a_1$, we have $d \le Q_1$. So all of the
$\genfrac(){0cm}{0}{d+2}{2}$ multi-indexes of degree $d$ satisfy the
condition of being constrained by $(Q_1)$.  Equivalently, $m_Q(d) =
\genfrac(){0cm}{0}{d+2}{2}$.
\end{proof}

\begin{prop} \label{spec_4}
Let $r= 4$ and let multi-indexes be twice-constrained by $Q =
(Q_1,Q_2)$ Then
\begin{enumerate}
\item[(i)] $m_Q(d)= \genfrac{}{}{}{0}{a_1a_2(2d-a_1-a_2+4)}{2}$
for $d$ $\ge$ $q-1 = a_1 + a_2 -3.$
\item[(ii)] $m_Q(d)= \genfrac{}{}{}{0}{a_1a_2(2d-a_1-a_2+4)}{2} + 1$
for $d = q-2 = a_1 + a_2 -4.$
\item[(iii)] $m_Q(d)= \genfrac(){0cm}{0}{d+3}{3}$
for $d < min(a_1,a_2).$
\end{enumerate}
\end{prop}
\begin{proof}
The formula for $d \ge q$ is given by Proposition \ref{r-2_fold}. An
alternative formula is derived as follows, in a manner similar to
the codimension $3$ case.  We are counting multi-indexes $D
=(d_1,d_2,d_3,d_4)$ such that $d_1 \le Q_1$ and $d_2 \le Q_2$. For
any choice of values of $d_1$ and $d_2$, we can complete $D$ by
choosing values of $d_3$ and $d_4$ whose sum is $d - d_1 -d_2$, and
there are $\genfrac(){0cm}{0}{(d-d_1-d_2)+(2-1)}{2-1} = d - d_1
-d_2+1$ ways to do it. So for $d \ge q = Q_1 + Q_2$,
\begin{align*}
m_Q(d) &= \sum_{0 \le d_1 \le Q_1}\sum_{0 \le d_2 \le Q_2} (d - d_1
-d_2 + 1).
\end{align*}

As in the previous theorem, we can regard the left- and right-hand
sides of
\begin{align*}
\genfrac{}{}{}{0}{a_1a_2(2d-a_1-a_2+4)}{2} = m_Q(d)= \sum_{0 \le d_1
\le Q_1}\sum_{0 \le d_2 \le Q_2} (d - d_1 -d_2 + 1)
\end{align*}

\noindent as an identity in $d$.  And again, the same argument
justifies the right-hand side as an expression for $m_Q(d)$ when $d
= q-1$ or $q -2$, except that the
range of summation must change.\\

For $d = q -1$, the term with $d_1 = Q_1$ and $d_2 = Q_2$ must be
omitted. That is,
\begin{align*}
m_Q(d) &= \sum_{0 \le d_1 \le Q_1}\sum_{0 \le d_2 \le
Q_2} (d - d_1 -d_2 + 1) - [(Q_1 + Q_2-1) - Q_1 - Q_2+ 1]\\
&= \genfrac{}{}{}{0}{a_1a_2(2d-a_1-a_2+4)}{2} - [(Q_1 + Q_2-1) - Q_1 - Q_2+ 1]\\
&= \genfrac{}{}{}{0}{a_1a_2(2d-a_1-a_2+4)}{2}.
\end{align*}

For $d = q -2$, there are three terms that must be omitted: those
with $d_1 = Q_1$ and $d_2 = Q_2$; with $d_1 = Q_1$ and $d_2 = Q_2
-1$; and with $d_1 = Q_1-1$ and $d_2 = Q_2$; that is,

\begin{align*}
&m_Q(d) = \sum_{0 \le d_1 \le Q_1}\sum_{0 \le d_2 \le Q_2} (d - d_1
-d_2 + 1) - [(Q_1 + Q_2-2) - Q_1 - Q_2+ 1]\\ &- [(Q_1 + Q_2-2) - Q_1
- (Q_2 -1)+ 1] - [(Q_1 +
Q_2-2) - (Q_1-1) - Q_2+ 1]\\
&= \genfrac{}{}{}{0}{a_1a_2(2d-a_1-a_2+4)}{2} - [(Q_1 + Q_2-2) - Q_1
- Q_2+ 1]\\
&- [(Q_1 + Q_2-2) - Q_1 - (Q_2 -1)+ 1] - [(Q_1 +
Q_2-2) - (Q_1-1) - Q_2+ 1]\\
&= \genfrac{}{}{}{0}{a_1a_2(2d-a_1-a_2+4)}{2} + 1.
\end{align*}

When $d < min(a_1,a_2)$, we have $d \le Q_1$ and $d \le Q_2$. So all
of the $\genfrac(){0cm}{0}{d+3}{3}$ multi-indexes of degree $d$
satisfy the condition of being constrained by $(Q_1, Q_2)$.
Equivalently, $m_Q(d) = \genfrac(){0cm}{0}{d+3}{3}$.
\end{proof}

\subsection{$(r-3)$-fold-constrained}

For $(r-3)$-fold-constrained monomials, a\\ closed-form expression
would be complicated.  We give a formula involving summations, and
then obtain closed-form expressions for the cases that $r$ is 4 or
5.

\begin{prop} \label{r-3_fold}
Let $r \ge 4$ and let multi-indexes be $(r-3)$-fold constrained by
$Q$. Then for $d \ge q$,
\begin{equation*}
m_Q(d) = \sum_{0 \le d_1 \le Q_1}...\sum_{0 \le d_{r-3} \le Q_{r-3}
} \genfrac(){0cm}{0}{d -d_1 - ... -d_{r-3} +2}{2}.
\end{equation*}
\end{prop}
\begin{proof}
Once again, the approach is similar to the previous cases. We are
counting multi-indexes $D =(d_1,...,d_r)$ such that $d_1 \le
Q_1,...,d_{r-3} \le Q_{r-3}$. For any choice of values of $d_1,...,
d_{r-3}$, we can complete $D$ by choosing values of $d_{r-2},
d_{r-1},$ and $d_r$ whose sum is $d - d_1 -...-d_{r-3}$, and there
are\\ $\genfrac(){0cm}{0}{(d-d_1-...-d_{r-3})+(3-1)}{3-1}$ ways to
do it.
\end{proof}

\begin{cor} \label{r-3_fold_cor}
Let $r = 4$ and let multi-indexes be once constrained by $Q$. Then
for $d \ge q = Q_1$,
\begin{equation*}
m_Q(d) = \genfrac{}{}{}{0}{a_1}{2}[d^2 + (4-a_1)d +
\genfrac{}{}{}{0}{a_1^2 - 6a_1 +11}{3}].
\end{equation*}
\end{cor}
\begin{proof}
We apply the formula from Proposition \ref{r-3_fold}.
\begin{align*}
m_Q(d) &=\sum_{0 \le d_1 \le Q_1} \genfrac(){0cm}{0}{d -d_1
+2}{2}\\
&=\sum_{0 \le d_1 \le Q_1}\genfrac{}{}{}{0}{(d+2-d_1)(d+1-d_1)}{2}\\
&=\genfrac{}{}{}{0}{1}{2}[ \sum_{0 \le d_1 \le Q_1}(d^2+3d+2) -
(2d+3) \sum_{0 \le d_1 \le Q_1}d_1  + \sum_{0 \le d_1 \le Q_1}d_1^2 ]\\
&= \genfrac{}{}{}{0}{1}{2}[a_1(d^2+3d+2) -
(2d+3)\genfrac{}{}{}{0}{a_1(a_1-1)}{2} +
\genfrac{}{}{}{0}{a_1(a_1-1)(2a_1 - 1)}{6}]\\
&=\genfrac{}{}{}{0}{a_1}{2}[d^2 + (4-a_1)d +
(2-\genfrac{}{}{}{0}{3a_1-3}{2} +
\genfrac{}{}{}{0}{2a_1^2-3a_1+1}{6})]\\
&=  \genfrac{}{}{}{0}{a_1}{2}[d^2 + (4-a_1)d +
\genfrac{}{}{}{0}{12 -9a_1+9 +2a_1^2-3a_1+1}{6}]\\
&= \genfrac{}{}{}{0}{a_1}{2}[d^2 + (4-a_1)d +
\genfrac{}{}{}{0}{a_1^2 - 6a_1 +11}{3}].
\end{align*}
\end{proof}

We remark that substituting $a_1 = 2$ gives a confirmation of the
formula in Corollary \ref{vspec_4}.\\

\begin{cor} \label{r-3_fold4}
Let $r = 5$ and let multi-indexes be twice constrained by $Q$. Then
for $d \ge q = Q_1 + Q_2$,
\begin{equation*}
m_Q(d) = \genfrac{}{}{}{0}{a_1a_2}{2}[d^2 + (5-a_1-a_2)d +
\genfrac{}{}{}{0}{2(a_1^2+a_2^2) - 15(a_1+a_2) + 3a_1a_2 + 35}{6}].
\end{equation*}
\end{cor}
\begin{proof}
We apply the formula from Proposition \ref{r-3_fold}.  The
computation is similar in nature to that of the previous corollary.
 We omit the details.
\end{proof}

\chapter{Construction of New Non-Unimodal Level Algebras}

In this chapter we construct several families of non-unimodal level
algebras.  As was mentioned in an earlier chapter, some of the
algebras described here were described and conjectured to be
non-unimodal by A. Iarrobino in 2005 following lines suggested by F.
Zanello in \cite{Z1}, and all of them use the same general framework
of Iarrobino and Zanello.  What is new here is
that we prove these algebras to be non-unimodal.\\

\section{Overview of the Construction of Level Algebras}

In this section we co-ordinate earlier results and describe our
framework for constructing new non-unimodal level algebras.  We
review some notation from previous sections, establish some new
notation, and give an overall description of the process. This
section is meant to be a qualitative overview. The quantitative
statements that ensure non-unimodality are proved in later
sections.\\

We let $k$ be an algebraically closed field of characteristic $0$
and define $R := k[X_1,...,X_r]$ and $\mathcal{D} :=
k[x_1,...,x_r]$. As previously described, the elements of $R$ act as
differential operators on $\mathcal{D}$.  Specifically, we will let
$r$ take the value $3, 4$, or $5$, and it will be convenient to
reduce the number of subscripts by defining $X := X_1, Y := X_2, Z
:= X_3, W := X_4$, and $V := X_5$; also $x := x_1, y := x_2, z :=
x_3, w
:= x_4$, and $v := x_5$.\\

We fix a positive integer $j$ that will become the socle degree of
the level algebra of codimension $r$ being constructed. For any
constraint $K := (K_1,...,K_n)$ of dimension $n \le r$ such that $0
\le K_i \le j$ for $i = 1,...,n$, we let $\mathcal{M}_K(d)$ denote
the set of multi-indexes of dimension $r$ and degree $d$ constrained
by $K$. We define $m_K(d) := \#(\mathcal{M}_K(d))$.  We define
$\mathcal{W}_{\mathcal{M}_K(j)}$ to be the vector subspace of
$\mathcal{D}_j$ spanned by all monomials $x^J$ such that $J \in
\mathcal{M}_K(j)$.  For a degree $d \le j$, we define
$\mathcal{V}_{\mathcal{M}_K(d)}$ to be the vector subspace of
$\mathcal{D}_d$ spanned by all monomials $x^D$ such that $D \in
\mathcal{M}_K(d)$.\\

We consider two constraints $P := (P_1,...,P_{n_1})$ of dimension
$n_1$ and\\ $Q := (Q_1,...,Q_{n_2})$ of dimension $n_2$, and use
them
as follows.\\

We choose a positive integer $s$ and specify a family of vector
subspaces $\mathcal{E} := \langle f_1,...,f_s\rangle  \subseteq
\mathcal{W}_{\mathcal{M}_P(j)} \subseteq \mathcal{D}_j$
parameterized by elements of the irreducible affine variety
$k^{sm_P(j)}$, where such an element represents a choice $C$ of
coefficients for the polynomials $f_1,...,f_s$.  We construct, for
general $C \in k^{sm_P(j)}$, the graded level algebra
$A_{\mathcal{E}(C)} :=R/Ann_R(\mathcal{E}(C))$.   We remark that our
previous discussion of Matlis Duality (in Chapter 2, section 3)
motivates this construction, and in particular
that Theorem $\ref{matlis_duality}$ guarantees $A_{\mathcal{E}(C)}$ is level.\\

If $s$ is sufficiently large, then for general $C$ the Hilbert
function $h_{\mathcal{E}(C)}$ of $A_{\mathcal{E}(C)}$ is computed,
according to Theorem \ref{omnibus}, by the rule
$h_{\mathcal{E}(C)}(d) = \dim_kR_e*\mathcal{E}(C) = \rank(U) =
m_P(d)$, where $e := j-d$ and $U$ is the
$e^{th}$ cropped matrix of $\mathcal{E}$.  \\

We will always choose $P$ so that the monomials are $(r-2)$-fold
constrained.  That is, for $r =3$, $P := (P_1)$, where for technical
reasons we require that $P_1 \ge 3$.  For $r=4$, $P := (P_1,P_2)$ or
$(j,P_2,P_3)$, where we require respectively that $P_2 \ge P_1 \ge
2$ or $P_3 \ge P_2 \ge 2$. For $r=5$, $P := (P_1,P_2,P_3)$, where we
require that $P_3 \ge P_2 \ge P_1 \ge 2$. To compute values of
$m_P(d)$, we rely on Propositions \ref{r-2_fold}, \ref{spec_3}, and
\ref{spec_4}, which deal with monomials that are $(r-2)$-fold
constrained. In these propositions, the definition is made that $a_i
:= P_i + 1$. However, again to reduce the number of subscripts, we
define $a := a_1, b := a_2, c:= a_3$.  Also, we define $\Delta$ to
be either $a, ab$, or $abc$,
according to whether $r = 3,4$, or $5$.\\

We next perform a similar construction to specify another family of
vector subspaces of $\mathcal{D}_j$.  This time we choose a positive
integer $u$, which for the examples here will always be either $1$
or $2$, and specify a family of vector subspaces $\mathcal{F} :=
\langle g_1,...,g_u\rangle  \subseteq \mathcal{W}_{\mathcal{M}_Q(j)}
\subseteq \mathcal{D}_j$ parameterized by elements of the
irreducible affine variety $k^{um_Q(j)}$, where such an element
represents a choice $C'$ of coefficients for the polynomials
$g_1,...,g_u$. Then we construct, for general $C' \in k^{um_Q(j)}$,
the graded level algbra $A_{\mathcal{F}(C')}
:=R/Ann_R(\mathcal{F}(C'))$.\\

The specific constraint $Q$ varies according to the family of
non-unimodals being constructed, as follows. For $r = 3$, $Q :=
(j,j,j)$.  For $r=4$, $Q := (j,j,j,0), (1)$, or
$(j,j,j,j)$.  For $r=5$, $Q := (j,j,j,0,0)$.\\

For these choices of constraints, the Hilbert function
$h_{\mathcal{F}(C')}$ of $A_{\mathcal{F}(C')}$ is computed according
to
Proposition \ref{k_3} or Corollary \ref{vspec_4}.\\

Having constructed families of vector subspaces $\mathcal{E}$ and
$\mathcal{F}$, we can construct the family $\mathcal{E} +
\mathcal{F}$ and consider the family of level algebras
$A_{\mathcal{E} + \mathcal{F}}$ with corresponding Hilbert functions
$h_{\mathcal{E} + \mathcal{F}}$.  For general $C$ and $C'$, we will
prove that $\mathcal{E}(C) + \mathcal{F}(C') = \mathcal{E}(C)
\bigoplus \mathcal{F}(C')$ and that $h_{\mathcal{E}(C)
\bigoplus \mathcal{F}(C')}$ is non-unimodal.\\

We establish some terminology for discussing non-unimodality of a Hilbert function
$h$.
\begin{deff} \label{crit} The terms \emph{single drop, double drop, initial degree, final degree}, and \emph{critical
range} are defined as follows.
\end{deff}If for some degree $i$, $h(i) > h(i+1) < h(i+2)$, we say
that $h$ has a \emph{single drop} with \emph{initial degree} $i$ and \emph{final
degree} $i_f := i+2$. If for some degree $i$, $h(i)
> max\{h(i+1),h(i+2)\} < h(i+3)$ we say that $h$ has a \emph{double
drop} with \emph{initial degree} $i$ and \emph{final degree} $i_f
:=i+3$. In this chapter, we will use the variables $i$ and $i_f$ to
represent the candidates for the initial and final degrees of a
single or double drop.  We say that a degree $d$
is in the \emph{critical range} if $i \le d \le i_f$.\\

To establish, for general $C \in k^{sm_P(j)}$ and $C' \in
k^{um_Q(j)}$, that $h_{\mathcal{E}(C) \bigoplus \mathcal{F}(C')}$
exhibits a single drop with initial degree $i$, our method will be
to show that
\begin{equation*}h_{\mathcal{E}(C)}(i) +
h_{\mathcal{F}(C')}(i)
> h_{\mathcal{E}(C)}(i + 1) + h_{\mathcal{F}(C')}(i + 1) <
h_{\mathcal{E}(C)}(i+2) + h_{\mathcal{F}(C')}(i+2) \end{equation*}
\noindent and then to establish that $h_{\mathcal{E}(C) \bigoplus
\mathcal{F}(C')}(d) = h_{\mathcal{E}(C)}(d) +
h_{\mathcal{F}(C')}(d)$ for $d = i, i+1,\\ i+2$; similarly for a
double drop.  To this end, we define differences\\ $\Delta_d =
h_{\mathcal{E}(C)}(d + 1) - h_{\mathcal{E}(C)}(d)$ and $\delta_d =
h_{\mathcal{F}(C')}(d + 1) - h_{\mathcal{F}(C')}(d)$.

\begin{lem} \label{add_deltas}
Assume, for some degree $d$, that \begin{equation*}
h_{\mathcal{E}(C) \bigoplus \mathcal{F}(C')}(d) =
h_{\mathcal{E}(C)}(d) + h_{\mathcal{F}(C')}(d)
\end{equation*}
\noindent and that \begin{equation*} h_{\mathcal{E}(C) \bigoplus
\mathcal{F}(C')}(d+1) = h_{\mathcal{E}(C)}(d+1) +
h_{\mathcal{F}(C')}(d+1).
\end{equation*} \noindent Then
\begin{equation*}
h_{\mathcal{E}(C) \bigoplus \mathcal{F}(C')}(d+1)= h_{\mathcal{E}(C)
\bigoplus \mathcal{F}(C')}(d) + \Delta_d + \delta_d .
\end{equation*}
\end{lem}
\begin{proof} This follows immediately from the definitions.
\end{proof}

Finally, we will need to establish, for appropriate values of $d$,
that the Hilbert functions $h_{\mathcal{E}(C)}$ and
$h_{\mathcal{F}(C')}$ do indeed add as desired.  To this end, we
will be using Lemma \ref{splice_lemma} together with Propositions
\ref{k_3} and \ref{sf_1}.

\section{Computations by Computer}

Direct computation is difficult with polynomials of high degree
having many terms; instead, we use the \cite{Macaulay2} computer
program. We set the field $k$ equal to $\mathbb{Z}/32749\mathbb{Z}$,
a large finite field (as suggested in \cite{E2}); the finiteness
permits rapid calculations and gives access to some special
applications that are implemented only for finite fields. Following
suggestions of A. Iarrobino, to simulate the selection of general
members of a vector space we first generate pseudorandom scalars on
the computer; for a fixed basis, we then use these scalars as
coefficients to produce members of the vector space; and we hope
that this procedure does in fact approximate the selection of
general members.  We use the command  \lq \lq
fromDual($\mathcal{W}$)" to compute $R/Ann(\mathcal{W})$ for a
vector subspace $\mathcal{W} \subseteq \mathcal{D}_j$. \\

Computers are useful for comprehension and they sometimes provide persuasive
plausibility arguments. But the proofs of non-unimodality given here are entirely
independent of computer results.

\section{Six Families of Level Algebras, together with
Computer-Calculated Hilbert Functions}

In this section we define six parameterized families $A_{\mathcal{E}
+ \mathcal{F}}$ of level algebras according to the program of the
previous section.  That is, for each choice of parameters we obtain
a family of algebras.  We will show, in a later section, that each
choice of parameters yields a family of algebras that are
non-unimodal for general $C$ and $C'$. In this section, we confine
ourselves to definitions, examples, and display of computer results.
\\

For each parameterized family, one of the parameters is $i$, which
denotes the initial degree of the single or double drop that (we
will subsequently prove) occurs in the Hilbert function.  It is not
necessary to specify the final degree $i_f$ as another parameter,
because its value can be calculated from $i$, once we make the claim
that all algebras in families $F_2, G_2$, and $G_3$ have a single
drop (so that $i_f = i + 2$) and all algebras in families $F_1,
G_1$,
and $H_1$ have a double drop (so that $i_f = i + 3$).\\

For each family we specify that the parameter $s$, the number of vectors generating
$\mathcal{E}$, be \emph{$h_{\mathcal{E}}$-sufficient}, by which we mean that the
$(j-i_f)^{th}$ cropped matrix of $\mathcal{E}$ should have at least as many rows as
columns, a condition motivated by Theorem \ref{omnibus}. For now, we do not state
the precise values of $s$ that are $h_{\mathcal{E}}$-sufficient, postponing the
discussion until Lemma \ref{best_s}. The number $u$ of vectors generating
$\mathcal{F}$ is not a parameter, since it is fixed within each family. The
type of the resulting level algebra is then min$(s, \dim_k\mathcal{M}_P(j)) + u$.\\

When displaying computer results, we will simplify notation by
writing $h$ instead of $h_{\mathcal{E}(C) \bigoplus
\mathcal{F}(C')}$.

\begin{deff}
The family $F_1(a,i,s)$ is obtained by setting $r=3$, $j = i + a$,
$P =(a-1)$, $Q = (j,j,j)$, $u = 1$.  We require that $a \ge 4$, that
$i \ge 2a$,
and that $s$ be $h_{\mathcal{E}}$-sufficient.\\
\end{deff}

EXAMPLE $F_1(21,42,4)$:  Let $r =3$, $t = 4+1 = 5$, $R = k[X,Y,Z]$, $\mathcal{D} =
k[x,y,z]$. We set socle degree $j = 63$. We define constraints $P := (20)$ with
$n_1=1$ and $Q := (63,63,63)$ with $n_2 = 3$. We define the vector space
$\mathcal{E}(C)$ as the span of 4 general members of
$\mathcal{W}_{\mathcal{M}_P(63)}$;  $\mathcal{F}(C')$ as the span of one general
member of $\mathcal{W}_{\mathcal{M}_Q(63)} = \mathcal{D}_{63}$. Then
$F_1(21,42,4)(C,C')$ := $R/Ann(\mathcal{E}(C) \bigoplus \mathcal{F}(C'))$. According
to Macaulay2:
\begin{equation*}
h(42) = 946; \; h(43) = 945; \; h(44) = 945; \; h(45) = 946.\\
\end{equation*}

\begin{deff}
The family $F_2(a,i,s)$ is obtained by setting $r=3$,\\
$j = i + (a-1)/2$, $P =(a-1)$, $Q = (j,j,j)$, $u = 2$.  We require
that $a \ge 7$ be odd, that
\begin{equation} \label{f2_ineq}
i \ge \genfrac{}{}{}{0}{2a - 3 + \sqrt{2a^2+8a+7}}{2},
\end{equation}
\noindent and that $s$ be $h_{\mathcal{E}}$-sufficient.
\end{deff}

EXAMPLE $F_2(21,36,14)$:  Let $r =3$, $t = 14+2 = 16$, $R = k[X,Y,Z]$,\\
$\mathcal{D} = k[x,y,z]$. We set socle degree $j =46$. We define constraints $P :=
(20)$ with $n_1=1$ and $Q := (46,46,46)$ with $n_2 = 3$. We define the vector space
$\mathcal{E}(C)$ as the span of 14 general members of
$\mathcal{W}_{\mathcal{M}_P(46)}$; $\mathcal{F}(C')$ as the span of two general
members of $\mathcal{W}_{\mathcal{M}_Q(46)} = \mathcal{D}_{46}$.  Then
$F_1(21,36,14)(C,C')$ := $R/Ann(\mathcal{E}(C) \bigoplus \mathcal{F}(C'))$.
According to Macaulay2:
\begin{equation*}
h(36) = 699; \; h(37) = 698; \; h(38) = 699.
\end{equation*}

\begin{deff}
The family $G_1(a,b,i,s)$ is obtained by setting $r=4$, $j = i + ab$, $P
=(a-1,b-1)$, $Q = (j,j,j,0)$, $u = 1$.  We require that $b \ge a \ge 2$, that $i \ge
ab + 1$, and that $s$ be $h_{\mathcal{E}}$-sufficient.
\end{deff}

EXAMPLE $G_1(3,4,13,2)$:  Let $r =4$, $t = 2+1 = 3$, $R = k[X,Y,Z,W]$,\\
$\mathcal{D} = k[x,y,z,w]$. We set socle degree $j = 25$.  We define constraints $P
:= (2,3)$ with $n_1=2$ and $Q:=(25,25,25,0)$ with $n_2 = 4$. We define the vector
space $\mathcal{E}(C)$ as the span of two general members of
$\mathcal{W}_{\mathcal{M}_P(25)}$; $\mathcal{F}(C')$ as the span of one general
member of $\mathcal{W}_{\mathcal{M}_Q(25)}=k[x,y,z]_{25}$.  Then
$G_1(3,4,13,2)(C,C')$ := $R/Ann(\mathcal{E}(C) \bigoplus \mathcal{F}(C'))$.
According to Macaulay2:
\begin{equation*}
h(13) = 229; \; h(14) = 228; \; h(15) = 228; \; h(16) = 229.
\end{equation*}

\begin{deff}
The family $G_2(a,b,i,s)$ is obtained by setting $r=4$, $j = i + ab/2$, $P
=(j,a-1,b-1)$, $Q = (1)$, $u = 1$.  We require that $b \ge a \ge 2$, that $ab$ be
even, that $i \ge ab/2 + 2$, and that $s$ be $h_{\mathcal{E}}$-sufficient.
\end{deff}

EXAMPLE $G_2(4,6,14,2)$: Let $r = 4$, $t = 2+1 = 3$, $R = k[X,Y,Z,W]$,\\
$\mathcal{D} = k[x,y,z,w]$.  We set socle degree $j = 26$.  We define constraints $P
:= (26,3,5)$ with $n_1=3$ and $Q := (1)$ with $n_2 = 1$. We define the vector space
$\mathcal{E}(C)$ as the span of two general members of
$\mathcal{W}_{\mathcal{M}_P(26)}$;  $\mathcal{F}(C')$ as the span of one general
member of $\mathcal{W}_{\mathcal{M}_Q(26)}$. Then $G_2(4,6,14,2)(C,C')$ :=
$R/Ann(\mathcal{E}(C) \bigoplus \mathcal{F}(C'))$.  According to Macaulay2:
\begin{equation*}
h(14) = 433; \; h(15) = 432; \; h(16) = 433.
\end{equation*}

\begin{deff}
The family $G_3(a,b,i,s)$ is obtained by setting $r=4$; $j = i + m$,
where $m$ is the largest integer such that
$\genfrac(){0cm}{0}{m+1}{2} < ab$; $P =(a-1,b-1)$; $Q = (j,j,j,j)$;
$u = 1$. We require that $b \ge a \ge 2$, that $ab$ \underline{not}
be equal to a binomial coefficient of the form
$\genfrac(){0cm}{0}{N}{2}$, that $i \ge a + b -3$; that
\begin{equation} \label{g3_ineq}
\genfrac{}{}{}{0}{ab[2i-a-b+4]}{2} + \genfrac(){0cm}{0}{m+3}{3} \le
\genfrac(){0cm}{0}{i+3}{3};
\end{equation}
\noindent and that $s$ be $h_{\mathcal{E}}$-sufficient.
\end{deff}

EXAMPLE $G_3(4,4,8,7)$:  Let $r = 4$, $t = 7+1 = 8$, $R = k[X,Y,Z,W]$,\\
$\mathcal{D} = k[x,y,z,w]$.  We set socle degree $j = 13$.  We define constraints $P
:= (3,3)$ with $n_1=2$ and $Q := (13,13,13,13)$ with $n_2 = 4$. We define the vector
space $\mathcal{E}(C)$ as the span of 7 general members of
$\mathcal{W}_{\mathcal{M}_P(13)}$; $\mathcal{F}(C')$ as the span of one general
member of $\mathcal{W}_{\mathcal{M}_Q(13)}=\mathcal{D}_{13}$.  Then
$G_3(4,4,8,7)(C,C')$ := $R/Ann(\mathcal{E}(C) \bigoplus \mathcal{F}(C'))$. According
to Macaulay2:
\begin{equation*}
h(8) = 152; \; h(9) = 147; \; h_4(10) = 148.
\end{equation*}

\begin{deff}
The family $H_1(a,b,c,i,s)$ is obtained by setting $r=5$, $j = i + abc$, $P
=(a-1,b-1,c-1)$, $Q = (j,j,j,0,0)$, $u = 1$.  We require that $c \ge b \ge a \ge 2$,
that $i \ge abc$, and that $s$ be $h_{\mathcal{E}}$-sufficient.
\end{deff}

EXAMPLE $H_1(2,2,3,12,2)$:  Let $r =5$, $t = 2+1 = 3$, $R = k[X,Y,Z,W,V]$,\\
$\mathcal{D} = k[x,y,z,w,v]$. We set socle degree $j = 24$.  We define constraints
$P := (1,1,2)$ with $n_1=3$ and $Q := (24,24,24,0,0)$ with $n_2 = 5$. We define the
vector space $\mathcal{E}(C)$ as the span of two general members of
$\mathcal{W}_{\mathcal{M}_P(24)}$; $\mathcal{F}(C')$ as the span of one general
member of $\mathcal{W}_{\mathcal{M}_Q(24)} = k[x,y,z]_{24}$.  Then
$H_1(2,2,3,12,2)(C,C')$ := $R/Ann(\mathcal{E}(C) \bigoplus \mathcal{F}(C'))$.
According to Macaulay2:
\begin{equation*}
h(12) = 223; \; h(13) = 222; \; h(14) = 222; \; h(15) = 223.
\end{equation*}

Before entering a discussion of the six families defined here, we stop to check that
they are all nonempty.
\begin{prop}
For any of the families $F_1(a,i,s), F_2(a,i,s), G_1(a,b,i,s),\\ G_2(a,b,i,s),
G_3(a,b,i,s), H_1(a,b,c,i,s)$, it is possible to find values of the parameters that
satisfy the requirements set forth in their definitions.
\end{prop}
\begin{proof}
For each of the families, it is immediate that all parameters except $s$ can be
chosen consistent with the requirements of the definitions of the families. So it is
enough to show that, for any such choice, $s:= m_P(j)$ is
$h_{\mathcal{E}}$-sufficient.\\

Setting $d = i_f$ and $e = j-i_f$, the size of the $(j-i_f)^{th}$ cropped matrix is
$sm_P(e) \times m_P(d)$.  To verify it has at least as many rows as columns when
$s:= m_P(j)$, we observe $$sm_P(e)=m_P(j)m_P(e) \ge m_P(j) \ge m_P(d),$$ the last
inequality following from Lemma \ref{h_inc}.
\end{proof}

\section{Formulas for $h_{\mathcal{ E}(C)}(d)$ and
$\Delta_d$}

Recall that in definition \ref{crit} we have defined degrees $d$ to lie in the
critical range if $i \le d \le i_f$, where $i$ and $i_f$ are the initial and final
degrees of a proposed single or double drop; and we have specified that $F_2, G_2$
and $G_3$ are candidates for having a single drop with initial degree $i$, whereas
$F_1, G_1$, and $H_1$ are candidates for having a double drop with initial degree
$i$.  Also recall that $\Delta$ was defined to be $a, ab,$ or $abc$, depending on
whether the codimension is 3, 4, or 5.
\begin{lem} \label{mp_values}
For any of the families $F_1(a,i,s), F_2(a,i,s), G_1(a,b,i,s),
G_2(a,b,i,s),\\ G_3(a,b,i,s), H_1(a,b,c,i,s)$, the values of
$m_P(d)$ are given as follows for degrees $d$ in the critical range.
\begin{align*}
&[For \; F_1 \; and \; F_2]:\; m_P(d) = \Delta(2d - a +3)/2. \\
&[For \;G_1, \; G_2, \; G_3]: \; m_P(d) = \Delta(2d - a - b + 4)/2.\\
&[For \; H_1]: \; m_P(d) = \Delta(2d - a -b -c+5)/2.
\end{align*}
\end{lem}
\begin{proof}
Propositions  \ref{spec_3}, \ref{spec_4}, and \ref{r-2_fold} yield
the formulas above, provided we verify that $d$ is large enough.\\

For $r = 3$, the condition is that $d \ge a-2$.  This is true for
$F_1$, since $d \ge i \ge 2a \ge a-2$; and for $F_2$, since $d \ge i
\ge
\genfrac{}{}{}{0}{2a - 3 + \sqrt{2a^2+8a+7}}{2} \ge a \ge a-2$.\\

For $r = 4$, the condition is that $d \ge a + b - 3$.  Recall that for $G_1$, $b \ge
a \ge 2$ and $i \ge ab + 1$;  for $G_2$, $b \ge a \ge 2$ and $i \ge ab/2 + 2$.  In
either case, we use the fact that, for $a \ge 2$ and $b \ge 2$, $ab/2 \ge a+ b - 2$.
For $G_1$, $d \ge i \ge ab+1 \ge 2a +2b -3 \ge a+b-3$.  For $G_2$, $d \ge i \ge ab/2
+2 \ge (a+b-2) +2 \ge a+b-3$.  For $G_3$, $d \ge i \ge a+b-3$, where the last
inequality
was required to hold in the definition of $G_3$.\\

For $r = 5$, the condition is that $d \ge a + b + c - 3$.  For
$H_1$, $c \ge b \ge a \ge 2$ and $i \ge abc$, so $d \ge i \ge abc
\ge a + b + c \ge a + b + c - 3$.
\end{proof}

\begin{prop} \label{hw_values}
For any of the families $F_1(a,i,s), F_2(a,i,s), G_1(a,b,i,s),\\
G_2(a,b,i,s), G_3(a,b,i,s), H_1(a,b,c,i,s)$, let $d$ lie in the
critical range. Then for general $C$,
\begin{align*}
&[For \; F_1 \; and \; F_2]:\; h_{\mathcal{E}(C)}(d) = \Delta(2d - a
+3)/2.\\
&[For \;G_1, \; G_2, \; G_3]: \; h_{\mathcal{E}(C)}(d) = \Delta(2d -
a - b + 4)/2.\\
&[For \; H_1]: \; h_{\mathcal{E}(C)}(d) = \Delta(2d - a -b -c+5)/2.
\end{align*}
\end{prop}
\begin{proof}
Recall that the hypothesis that $s$ is $h_{\mathcal{E}}$-sufficient means that the
$(j - i_f)^{th}$ cropped matrix of $\mathcal{E}$ has at least as many rows as
columns.  This matrix has $sm_P(j-i_f)$ rows and $m_P(i_f)$ columns.  We observe
that, for any $d$ in the critical range, the $(j - d)^{th}$ cropped matrix of
$\mathcal{E}$ also has at least as many rows as columns, since by Lemma \ref{h_inc}
it has $sm_P(j-d) \ge sm_P(j-i_f)$ rows and $m_P(d) \le m_P(i_f)$ columns. So for
all values of $d$ in the critical range, Theorem \ref{omnibus} applies, and for
general $C$ we have $h_{\mathcal{E}(C)}(d) = m_P(d)$, the rank of the $e^{th}$
cropped matrix $U$ of $\mathcal{E}$, or equivalently the number of columns in $U$.
\end{proof}

\begin{cor} \label{Delta}
For any of the families $F_1(a,i,s), F_2(a,i,s), G_1(a,b,i,s),\\
G_2(a,b,i,s), G_3(a,b,i,s), H_1(a,b,c,i,s)$, and for $d :=
i,...,i_f-1$, $\Delta_d = \Delta$ for general $C$.
\end{cor}
\begin{proof}
Recalling that $\Delta_d := h_{\mathcal{E}(C)}(d + 1) -
h_{\mathcal{E}(C)}(d)$, we obtain values for\\
$h_{\mathcal{E}(C)}(d+1)$ (for general $C$) and
$h_{\mathcal{E}(C)}(d)$ (for general $C$) from Proposition
\ref{hw_values}.  Since these formulas both hold for general $C$,
Lemma \ref{zar_comp} guarantees that they hold simultaneously for
general $C$, so subtracting them gives a formula for their
difference that holds for general $C$.
\end{proof}

\section{Formulas for $h_{\mathcal{F}(C')}(d)$ and
$\delta_d$}

\begin{prop} \label{hv_values}
For any of the families $F_1(a,i,s), F_2(a,i,s), G_1(a,b,i,s),\\
G_2(a,b,i,s), G_3(a,b,i,s), H_1(a,b,c,i,s)$, let $d$ lie in the
critical range. Then for general $C'$, the following formulas for
$h_{\mathcal{F}(C')}(d)$ apply. (Recall that $e := j-d$.)
\begin{align*}
&[For \; F_1, \; G_1, \; H_1]: \; h_{\mathcal{F}(C')}(d) =
\genfrac(){0cm}{0}{e+2}{2}.\\
&[For \; F_2\;]: \; h_{\mathcal{F}(C')}(d) =
2\genfrac(){0cm}{0}{e+2}{2}.\\
&[For \; G_2\;]: \;
h_{\mathcal{F}(C')}(d)= (e+1)^2.\\
&[For \; G_3\;]: \; h_{\mathcal{F}(C')}(d) =
\genfrac(){0cm}{0}{e+3}{3}.
\end{align*}
\end{prop}
\begin{proof}
For all of the families except $G_2$, we apply Proposition
\ref{k_3}, which requires us to verify that
$u\genfrac(){0cm}{0}{e+2}{2} \le
\genfrac(){0cm}{0}{d+2}{2}$.  We consider each family in turn.\\

For $F_1$, $u=1$ and $e \le a \le 2a\le i \le d$.\\

For $G_1$, $u=1$ and $e \le ab \le ab+1 \le i \le d$.\\

For $H_1$, $u=1$ and $e \le abc \le i \le d$.\\

For $G_3$, $u=1$ and $e \le m \le i \le d$, where $m \le i$ follows
immediately
from \eqref{g3_ineq}.\\

For $F_2$, $u=2$ and
\begin{align*}
2\genfrac(){0cm}{0}{e + 2}{2} &\le 2\genfrac(){0cm}{0}{(a-1)/2 + 2}{2}\\
&= 2\genfrac{}{}{}{0}{[(a+3)/2][(a+1)/2]}{2}\\
&=\genfrac{}{}{}{0}{(a+3)(a+1)}{4}\\
&= \genfrac{}{}{}{0}{a^2+4a+3}{4}\\
&\le\genfrac{}{}{}{0}{a^2+3a+2}{2}\\
&=\genfrac(){0cm}{0}{a+2}{2} \le\genfrac(){0cm}{0}{i+2}{2}
\le\genfrac(){0cm}{0}{d+2}{2},
\end{align*}
\noindent where $a \le i$ follows immediately from
\eqref{f2_ineq}.\\

For $G_2$, we apply Proposition \ref{sf_1}, which requires that $e
\le d$.  We have\\ $e \le ab/2 \le ab/2 + 2 \le i \le d$.
\end{proof}

\begin{cor} \label{delta}
For the families $F_1(a,i,s), F_2(a,i,s), G_1(a,b,i,s),
G_2(a,b,i,s),\\ G_3(a,b,i,s), H_1(a,b,c,i,s)$, for $d :=
i,...,i_f-1$, and for general $C'$ the formulas for $\delta_d$ are
as follows.
\begin{align*}
&[For \; F_1, \; G_1, \; H_1]: \; \delta_d = -(e+1).\\
&[For \;F_2\;]: \delta_d = -2(e+1).\\
&[For \; G_2\;]: \delta_d = -(2e+1).\\
&[For \; G_3\;]: \delta_d =-\genfrac(){0cm}{0}{e+2}{2}.
\end{align*}
\end{cor}
\begin{proof}
Recalling that $\delta_d := h_{\mathcal{F}(C')}(d + 1) -
h_{\mathcal{F}(C')}(d)$, we obtain values for\\
$h_{\mathcal{F}(C')}(d+1)$ (for general $C'$) and
$h_{\mathcal{F}(C')}(d)$ (for general $C'$) from Proposition
\ref{hv_values}.  Since these formulas both hold for general $C'$,
Lemma \ref{zar_comp} guarantees that they hold simultaneously for
general $C'$, so subtracting them gives a formula for their
difference that holds for general $C'$.  In performing the
subtraction, we use the following well-known formula for binomial
coefficients.
\begin{equation*}
\genfrac(){0cm}{0}{N}{M} = \genfrac(){0cm}{0}{N-1}{M} +
\genfrac(){0cm}{0}{N-1}{M-1}.
\end{equation*}
\end{proof}

\section{Computing $h_{\mathcal{E}(C)\bigoplus\mathcal{F}(C')}(d)$}

\begin{thm}{} \label{hwv}
For any of the families $F_1(a,i,s), F_2(a,i,s), G_1(a,b,i,s),\\
G_2(a,b,i,s), G_3(a,b,i,s), H_1(a,b,c,i,s)$, for general $C$ and
$C'$, we have \begin{enumerate}
\item [(i)] $\mathcal{E}(C)\cap\mathcal{F}(C') = \{0\}$ and
\item[(ii)] $h_{\mathcal{E}(C)\bigoplus\mathcal{F}(C')}(d) =
h_{\mathcal{E}(C)}(d)+h_{\mathcal{F}(C')}(d)$, simultaneously for
all degrees $d$ in the critical range.\end{enumerate}
\end{thm}
\begin{proof}
Since we must verify (ii) for only finitely many values of $d$, by
Lemma \ref{zar_comp} it is enough to verify it separately for each
degree $d$. If (i) has been established, to verify (ii) it is
enough, by Lemma \ref{splice_lemma}, to show that
\begin{equation} \label{lastl}
\text{For general }C\text{ and }C',\text{ }R_e*\mathcal{E}(C) \cap
R_e*\mathcal{F}(C') = \{0\}. \end{equation} We remark that
$\eqref{lastl}$ also implies (i), since the existence of a nonzero
polynomial  $f \in \mathcal{E}(C) \cap \mathcal{F}(C')$ would imply
the existence of a nonzero $e^{th}$ partial derivative of $f$, which
would lie in $R_e*\mathcal{E}(C) \cap
R_e*\mathcal{F}(C')$. Thus, to prove the theorem, it is enough to prove
$\eqref{lastl}$.\\

To show $\eqref{lastl}$, for all families except $G_2$, we apply
Proposition \ref{k_3} as follows. Let $\mathcal{W} :=
\mathcal{V}_{\mathcal{M}_P(d)}$ (defined in $\eqref{vlabel} )$. We
observe that, for each family other than $G_2,\\
\mathcal{V}_{\mathcal{M}_P(d)} = k[x_1,...,x_m]_d$, where $m$ has
the value 3 or 4.  Let $\mathcal{Z} :=
\mathcal{V}_{\mathcal{M}_P(d)} \cap \mathcal{V}_{\mathcal{M}_Q(d)}$.
To use Proposition \ref{k_3} with these values of $\mathcal{W}$ and
$\mathcal{Z}$, we must verify that
\begin{equation} \label{numerical}
\dim_k\mathcal{Z} \le \genfrac(){0cm}{0}{d+p}{p}
-u\genfrac(){0cm}{0}{e+p}{p},
\end{equation}
\noindent where $p := m-1$.  Assuming this verification has been
done, we conclude from part (iii) of Proposition \ref{k_3} that, for
general $C'$, $\mathcal{W} \cap R_e*\mathcal{F}(C') = \{0\}$.  Since
$R_e*\mathcal{E}(C) \subseteq \mathcal{W}$, we conclude that
$R_e*\mathcal{E}(C) \cap R_e*\mathcal{F}(C') =
\{0\}$, as required.\\

Before proceeding to the numerical verifications of $\eqref{numerical}$, we consider
the family $G_2$, and apply Proposition \ref{sf_1} to verify $\eqref{lastl}$ as
follows.  As with the other five families, we again let $\mathcal{W} :=
\mathcal{V}_{\mathcal{M}_P(d)}$ and $\mathcal{Z} := \mathcal{V}_{\mathcal{M}_P(d)}
\cap \mathcal{V}_{\mathcal{M}_Q(d)}$. To use Proposition \ref{sf_1} with these
values of $\mathcal{W}$ and $\mathcal{Z}$, we must verify $\eqref{sf_11}$ and
$\eqref{sf_12}$. Assuming these verifications have been done, we proceed as before,
concluding from part (iii) of Proposition \ref{sf_1} that, for general $C'$,
$\mathcal{W} \cap R_e*\mathcal{F}(C') = \{0\}$. Again, since $R_e*\mathcal{E}(C)
\subseteq \mathcal{W}$, we conclude that $R_e*\mathcal{E}(C) \cap
R_e*\mathcal{F}(C') = \{0\}$, as required. We now proceed to the verifications of
$\eqref{numerical},
\eqref{sf_11},$ and $\eqref{sf_12}$.\\

For $F_1$, \begin{align*}
\mathcal{Z} :=& \mathcal{V}_{\mathcal{M}_P(d)} \cap \mathcal{V}_{\mathcal{M}_Q(d)}\\
=& \mathcal{V}_{\mathcal{M}_P(d)} \cap \mathcal{D}_d\\
=& \mathcal{V}_{\mathcal{M}_P(d)}.
\end{align*}

From Proposition \ref{mp_values}, $\dim_k\mathcal{Z} =
\genfrac{}{}{}{0}{a(2d-a+3)}{2}$, so to use Proposition \ref{k_3} we
must verify that

\begin{equation*}
\genfrac{}{}{}{0}{a(2d-a+3)}{2} \le \genfrac(){0cm}{0}{d+2}{2} -
\genfrac(){0cm}{0}{e+2}{2},
\end{equation*}

\noindent or equivalently that

\begin{equation*}
\genfrac(){0cm}{0}{d+2}{2}  - \genfrac{}{}{}{0}{a(2d-a+3)}{2} -
\genfrac(){0cm}{0}{e+2}{2}\ge 0.
\end{equation*}\\

We are considering values of $d \ge 2a$ and $e \le a$, so

\begin{align*}
\genfrac(){0cm}{0}{d+2}{2}  - \genfrac{}{}{}{0}{a(2d-a+3)}{2} -
\genfrac(){0cm}{0}{e+2}{2}&\ge \\
\genfrac(){0cm}{0}{d+2}{2} - \genfrac{}{}{}{0}{a(2d-a+3)}{2}  -
\genfrac(){0cm}{0}{a+2}{2}&= \\
[ (d^2 + 3d + 2) -(2ad - a^2 + 3a) -(a^2 + 3a + 2)]/2&=\\
[d(d-2a+3) - 6a]/2&\ge\\
[ 2a(2a-2a+3) -6a ]/2&= 0.
\end{align*}

For $F_2$, again
\begin{align*}
\mathcal{Z} :=& \mathcal{V}_{\mathcal{M}_P(d)} \cap \mathcal{V}_{\mathcal{M}_Q(d)}\\
=& \mathcal{V}_{\mathcal{M}_P(d)} \cap \mathcal{D}_d\\
=& \mathcal{V}_{\mathcal{M}_P(d)},
\end{align*}
\noindent and again $\dim_k\mathcal{Z} =
\genfrac{}{}{}{0}{a(2d-a+3)}{2}$.\\

To use Proposition \ref{k_3} we must verify that

\begin{equation*}
\genfrac{}{}{}{0}{a(2d-a+3)}{2} \le \genfrac(){0cm}{0}{d+2}{2} -
2\genfrac(){0cm}{0}{e+2}{2},
\end{equation*}

\noindent or equivalently that

\begin{equation*}
\genfrac(){0cm}{0}{d+2}{2}  - \genfrac{}{}{}{0}{a(2d-a+3)}{2} -
2\genfrac(){0cm}{0}{e+2}{2}\ge 0.
\end{equation*}\\

We are considering values of $e \le (a-1)/2$, and $d \ge i$, so

\begin{align*}
\genfrac(){0cm}{0}{d+2}{2}  - \genfrac{}{}{}{0}{a(2d-a+3)}{2} -
2\genfrac(){0cm}{0}{e+2}{2}&\ge \\
\genfrac(){0cm}{0}{d+2}{2}  - \genfrac{}{}{}{0}{a(2d-a+3)}{2} -
2\genfrac(){0cm}{0}{(a-1)/2+2}{2}&= \\
[ (d^2 + 3d + 2) -(2ad - a^2 + 3a) -2((a-1)^2/4 + 3(a-1)/2 + 2)]/2&=\\
[ (d^2 + 3d + 2) -(2ad - a^2 + 3a) -((a^2 - 2a +1)/2 + 3(a-1) + 4)]/2&=\\
[d(d-2a+3) + a^2/2 -5a + 1/2]/2&=\\
[d(d-2a+3) + \genfrac{}{}{}{0}{a^2+1}{2} -5a]/2]&\ge\\
[i(i-2a+3) +\genfrac{}{}{}{0}{a^2+1}{2} -5a]/2&=\\
[i^2 + (3-2a)i +(\genfrac{}{}{}{0}{a^2+1}{2} -5a)]/2.
\end{align*}

\noindent We must demonstrate that the expression within square
brackets is always non-negative.  We recall that, in defining the
family $F_2$, we have required that
\begin{equation*}
i \ge \genfrac{}{}{}{0}{2a-3 + \sqrt{2a^2+8a + 7}}{2}.
\end{equation*}
Using the quadratic formula to solve the quadratic inequality
\begin{equation*}i^2 + (3-2a)i
+(\genfrac{}{}{}{0}{a^2+1}{2} -5a) \ge 0,
\end{equation*}
\noindent and noting that we are only interested in positive values
of $i$ as solutions, we have:
\begin{align*}
i \ge &\genfrac{}{}{}{0}{2a-3 + \sqrt{(4a^2-12a + 9) - (2a^2 +2
-20a)}}{2}, \text{ or} \\
 i \ge &\genfrac{}{}{}{0}{2a-3 + \sqrt{2a^2+8a + 7}}{2},
\end{align*}
\noindent which has been assumed true for the family $F_2$.\\

For $G_1$,
\begin{align*}
\mathcal{Z} :=& \mathcal{V}_{\mathcal{M}_P(d)} \cap \mathcal{V}_{\mathcal{M}_Q(d)}\\
=& \mathcal{V}_{\mathcal{M}_P(d)} \cap k[x,y,z]_d,
\end{align*}
\noindent or equivalently the vector subspace of $k[x,y,z]_d$
spanned by monomials constrained by $P := (a-1,b-1)$. Its dimension
is $ab$ by Proposition \ref{r-1_fold} since\\ $d \ge ab \ge a+b >
(a-1) + (b-1)$. We must verify that

\begin{equation*}
ab \le \genfrac(){0cm}{0}{d+2}{2} - \genfrac(){0cm}{0}{e+2}{2},
\end{equation*}

\noindent or equivalently that

\begin{equation*}
\genfrac(){0cm}{0}{d+2}{2} - \genfrac(){0cm}{0}{e+2}{2} - ab \ge 0.
\end{equation*}\\

We are considering values of $d \ge ab+1$ and $e \le ab$, so

\begin{align*}
\genfrac(){0cm}{0}{d+2}{2} -
\genfrac(){0cm}{0}{e+2}{2} - ab &\ge \\
\genfrac(){0cm}{0}{ab+3}{2} -
\genfrac(){0cm}{0}{ab+2}{2} - ab &= \\
[ (a^2b^2 + 5ab + 6) - (a^2b^2 + 3ab + 2) - 2ab]/2&=\\
[ 4]/2&\ge 0.
\end{align*}

For $G_3$,
\begin{align*}
\mathcal{Z} :=& \mathcal{V}_{\mathcal{M}_P(d)} \cap \mathcal{V}_{\mathcal{M}_Q(d)}\\
=& \mathcal{V}_{\mathcal{M}_P(d)} \cap \mathcal{D}_d\\
=& \mathcal{V}_{\mathcal{M}_P(d)}.
\end{align*}
\noindent By Proposition \ref{mp_values}, the dimension of
$\mathcal{Z}$ is $\genfrac{}{}{}{0}{ab[2d-a-b+4]}{2}$.  To use
Proposition \ref{k_3}, we must verify for $d = i, i+1, i+2$ that
\begin{equation*}
\genfrac{}{}{}{0}{ab[2d-a-b+4]}{2} + \genfrac(){0cm}{0}{j-d+ 3}{3}
\le \genfrac(){0cm}{0}{d+3}{3}.
\end{equation*}

\noindent For the case that $d = i, e = j-d = m$, this is just
\eqref{g3_ineq}.  Moving to $d = i + 1,\\ e = j-d = m-1$, the first
term on the left increases by $ab$, the second term decreases, and
the term on the right increases by
\begin{align*}\genfrac(){0cm}{0}{i + 4}{3} - \genfrac(){0cm}{0}{i +
3}{3} &=\\
 \genfrac(){0cm}{0}{i + 3}{2}  &\ge \\
 \genfrac(){0cm}{0}{m + 3}{2} &\ge\\
 \genfrac(){0cm}{0}{m + 2}{2} &\ge ab,
\end{align*}

\noindent so the required inequality holds for the case $d = i + 1,
e = j-d = m-1$.  A similar computation establishes the inequality
for $d = i + 2, e = j-d = m-2$.\\

For $H_1$,
\begin{align*}
\mathcal{Z} :=& \mathcal{V}_{\mathcal{M}_P(d)} \cap \mathcal{V}_{\mathcal{M}_Q(d)}\\
=& \mathcal{V}_{\mathcal{M}_P(d)} \cap k[x,y,z]_d,
\end{align*}
\noindent or equivalently the vector subspace of $k[x,y,z]_d$
spanned by monomials constrained by $P := (a-1,b-1,c-1)$. Its
dimension is $0$ by Proposition \ref{r_fold} since $d \ge abc \ge
a+b+c
> (a-1) + (b-1) + (c-1)$. We must verify that

\begin{equation*}
0 = \dim_k\mathcal{Z} \le \genfrac(){0cm}{0}{d+2}{2} -
\genfrac(){0cm}{0}{e+2}{2},
\end{equation*}

\noindent or equivalently that

\begin{equation*}
\genfrac(){0cm}{0}{d+2}{2} - \genfrac(){0cm}{0}{e+2}{2} \ge 0.
\end{equation*}\\

Since we are considering values of $d \ge abc$ and $e \le abc$, this
is immediate.\\

For $G_2$, we use Proposition \ref{sf_1}, taking $\mathcal{Z} :=
\mathcal{V}_{\mathcal{M}_P(d)} \cap \mathcal{V}_{\mathcal{M}_Q(d)}$,
which is the vector subspace of $D_d$ constrained by $K :=
(1,a-1,b-1)$. Its dimension is $2ab$ by Proposition \ref{r-1_fold}
since $d \ge ab/2 +2 \ge (a + b -2) + 2 \ge 1 + (a-1) + (b-1)$.
Looking further, we can see that exactly $ab$ of the generators do
not contain the variable $x$ by again applying Proposition
\ref{r-1_fold}, this time to the constraint $K' :=(0,a-1,b-1)$.  So
to apply Proposition \ref{sf_1} with this choice of $\mathcal{Z}$,
we use $ab$ for the
parameter $c$ of that proposition.\\

According to Proposition \ref{sf_1} (iii), there are two
inequalities to verify for $d$ in the critical range.  The first is
that

\begin{equation*}
m_P(d) -2c \ge m_P(e),
\end{equation*}

\noindent or equivalently that
\begin{equation*}
m_P(d) - m_P(e) -2c \ge 0.
\end{equation*}\\

We are considering values of $d \ge ab/2 + 2$ and $e \le ab/2$, with
$c=ab$, so

\begin{align*}
m_P(d)- m_P(e) -2c &=\\
(d+1)^2 - (e+1)^2 -2ab &\ge \\
(ab/2 + 3)^2 - (ab/2 + 1)^2 -2ab &=\\
(a^2b^2/4 + 3ab + 9) - (a^2b^2/4 + ab + 1) -2ab&=\\
8 &> 0.
\end{align*}\\

The second inequality to be verified is that
\begin{equation*}
 \genfrac(){0cm}{0}{d+(r-2)}{r-2} - c \ge
\genfrac(){0cm}{0}{(e-1)+(r-2)}{r-2},
\end{equation*}

\noindent or equivalently that
\begin{equation*}
 \genfrac(){0cm}{0}{d+(r-2)}{r-2} -
\genfrac(){0cm}{0}{(e-1)+(r-2)}{r-2} - c \ge 0.
\end{equation*}

Once again we are considering values of $d \ge ab/2 + 2$ and $e \le
ab/2$, with $c=ab$ and $r=4$, so

\begin{align*}
 \genfrac(){0cm}{0}{d+(r-2)}{r-2} -
\genfrac(){0cm}{0}{(e-1)+(r-2)}{r-2} - c &=\\
\genfrac(){0cm}{0}{d+2}{2} - \genfrac(){0cm}{0}{e+1}{2} - ab
&\ge\\
\genfrac(){0cm}{0}{ab/2+4}{2} - \genfrac(){0cm}{0}{ab/2+1}{2} - ab
&=\\
[(ab/2+4)(ab/2+3) - (ab/2+1)(ab/2) -2ab]/2 &=\\
[(a^2b^2/4 + 7ab/2 + 12) -(a^2b^2/4 + ab/2) -2ab]/2 &=\\
[ab+12]/2 &> 0.
\end{align*}
\end{proof}

\section{Proof of Non-Unimodality}

\begin{thm}{} \label{non_uni}
For general $C$ and $C'$, all of the families $F_1(a,i,s),
F_2(a,i,s),\\ G_1(a,b,i,s), G_2(a,b,i,s), G_3(a,b,i,s),
H_1(a,b,c,i,s)$, are non-unimodal. The Hilbert functions of
$F_2,G_2$, and $G_3$ have single drops with initial degree $i$. The
Hilbert functions of $F_1,G_1$, and $H_1$ have double drops with
initial degree $i$.
\end{thm}
\begin{proof}
By Theorem \ref{hwv}, we are entitled to use Lemma \ref{add_deltas},
in which we use the values of $\Delta_d$ and $\delta_d$ given in
Corollaries \ref{Delta} and \ref{delta}. We let $h :=
h_{\mathcal{E}(C)\bigoplus\mathcal{F}(C')}$.  For each family, our
approach is to determine relationships between consecutive values of
$h(d)$ that demonstrate the non-unimodality of $h$ in the critical
range.

For $F_2$, we have $j-i = (a-1)/2$, so
\begin{align*}
h(i+1) &= h(i) + \Delta_i + \delta_i\\
&=h(i) + a - 2[(a-1)/2 +1]\\
&=h(i) -1,
\end{align*}
\noindent and
\begin{align*}
h(i+2) &= h(i+1) + \Delta_{i+1} + \delta_{i+1}\\
&=h(i+1) + a - 2([(a-1)/2 -1] +1)\\
&=h(i+1) + 1.
\end{align*}

For $G_2$, we have $j-i = ab/2$, so
\begin{align*}
h(i+1) &= h(i) + \Delta_i + \delta_i\\
&=h(i) + ab - [2(ab/2) + 1]\\
&=h(i) -1,
\end{align*}
\noindent and
\begin{align*}
h(i+2) &= h(i+1) + \Delta_{i+1} + \delta_{i+1}\\
&=h(i+1) + ab - [2(ab/2 -1) + 1]\\
&=h(i+1) + 1.
\end{align*}

For $G_3$, we have $j-i = m$, where by definition
\begin{equation*}
\genfrac(){0cm}{0}{m+1}{2} < ab < \genfrac(){0cm}{0}{m+2}{2}.
\end{equation*}
\noindent We have
\begin{align*}
h(i+1) &= h(i) + \Delta_i + \delta_i\\
&=h(i) + ab - \genfrac(){0cm}{0}{m+2}{2}\\
&< h(i),
\end{align*}
\noindent and
\begin{align*}
h(i+2) &= h(i+1) + \Delta_{i+1} + \delta_{i+1}\\
&=h(i+1) + ab - \genfrac(){0cm}{0}{(m-1)+2}{2}\\
&=h(i+1) + ab - \genfrac(){0cm}{0}{m+1}{2}\\
&> h(i+1).
\end{align*}

For $F_1, G_1$, and $H_1$, we have $j-i = \Delta$, so
\begin{align*}
h(i+1) &= h(i) + \Delta_i + \delta_i\\
&=h(i) + \Delta - (\Delta +1)\\
&=h(i) -1,
\end{align*}
\noindent and
\begin{align*}
h(i+2) &= h(i+1) + \Delta_{i+1} + \delta_{i+1}\\
&=h(i+1) + \Delta - ([\Delta -1] +1)\\
&=h(i+1),
\end{align*}
\noindent and
\begin{align*}
h(i+3) &= h(i+2) + \Delta_{i+2} + \delta_{i+2}\\
&=h(i) + \Delta - ([\Delta -2] +1)\\
&=h(i+1) + 1.
\end{align*}
\end{proof}

\section{Computation of Types}

\begin{lem} \label{type_wv}
For any of the families $F_1(a,i,s), F_2(a,i,s), G_1(a,b,i,s),
G_2(a,b,i,s),\\ G_3(a,b,i,s), H_1(a,b,c,i,s)$, let $s$ be chosen
such that $s \le m_P(j)$.  Then for general $C$ and $C'$, the type
of $A_{\mathcal{E}(C)\bigoplus\mathcal{F}(C')}$ is $s + u$.
\end{lem}
\begin{proof}
By Corollary \ref{full_type}, for general $C$ and $C'$, the
dimension of $\mathcal{E}(C)$ is $s$ and the dimension of
$\mathcal{F}(C')$ is $u$.  By Theorem \ref{hwv}, $\mathcal{E}(C)
\cap \mathcal{F}(C') = \{0\}$.
\end{proof}

Among the defining parameters, we will call $a,b$, and $c$ the
\emph{P-parameters}, since they define the constraint $P$.

\begin{lem}
For any of the families $F_1(a,i,s), F_2(a,i,s), G_1(a,b,i,s),
G_2(a,b,i,s),\\ G_3(a,b,i,s), H_1(a,b,c,i,s)$, for any fixed choice
of the $P$-parameters, the value of $j-i_f$ is constant, given by
the following formulas:
\begin{align*}
&[For \; F_1 \; G_1, \; and \; H_1]: \; \Delta -3.\\
&[For \; F_2]: \; (a-1)/2 - 2.\\
&[For \; G_2]: \; ab/2 - 2.\\
&[For \; G_3]: \; m-2, \; where \; m \; is \; the \; greatest \;
integer \; such \; that \genfrac(){0cm}{0}{m+1}{2} < ab.
\end{align*}
\end{lem}
\begin{proof}
From the definitions, the value of $j-i$ is $\Delta$ for $F_1, G_1$,
and $H_1$; $(a-1)/2$ for $F_2$; $ab/2$ for $G_2$; and the greatest
integer $m$ such that $\genfrac(){0cm}{0}{m+1}{2} < ab$ for $G_3$.
For $F_1, G_1$, and $H_1$, which have double drops, $i_f = i + 3$;
for the others, which have single drops, $i_f = i + 2$.
\end{proof}

We define the \emph{ceiling function} $\ceil(x)$ as follows.  For
any non-negative real number $x$, $\ceil(x)$ is the the smallest
integer
$c$ such that $c \ge x$.\\

\begin{lem} \label{best_s}
For any of the families $F_1(a,i,s), F_2(a,i,s), G_1(a,b,i,s), G_2(a,b,i,s),\\
G_3(a,b,i,s), H_1(a,b,c,i,s)$, for general $C$ and $C'$, the requirement that $s$ is
$h_{\mathcal{E}}$-sufficient is equivalent to the condition that $s \ge \ceil
\genfrac{}{}{}{0}{m_P(i_f)}{m_P(j-i_f)}$.  For any fixed choice of the
$P$-parameters, this formula attains the smallest possible value when $i$ is as
small as permitted by the definition of the family.
\end{lem}
\begin{proof}
The definition of $h_{\mathcal{E}}$-sufficient is that the $(j-i_f)^{th}$ cropped
matrix of $\mathcal{E}$ has at least as many rows as columns.  It has $sm_P(j-i_f)$
rows and $m_P(i_f)$ columns, so the condition is that $sm_P(j-i_f) \ge m_P(i_f)$, or
equivalently that $s \ge \genfrac{}{}{}{0}{m_P(i_f)}{m_P(j-i_f)}$.
We introduce the ceiling function because $s$ must be an integer.\\

For any of the families, once we have fixed a choice of
$P$-parameters, the lemma follows if we show that $\ceil
\genfrac{}{}{}{0}{m_P(i_f)}{m_P(j-i_f)}$ is a non-decreasing
function of $i$.  Since the ceiling function is nondecreasing, it is
enough to show that $\genfrac{}{}{}{0}{m_P(i_f)}{m_P(j-i_f)}$ is
non-decreasing as a function of $i$, and by the previous lemma it is
enough that $m_P(i_f)$ be non-decreasing as a function of $i$. By
Lemma \ref{h_inc}, it is enough that $i_f$ be a non-decreasing
function of $i$.  But $i_f := i + 2$ or $i + 3$, depending on
whether the family has a single drop or a double drop.
\end{proof}

\begin{lem}
For any of the families $F_1(a,i,s), F_2(a,i,s), G_1(a,b,i,s),
G_2(a,b,i,s),\\ G_3(a,b,i,s), H_1(a,b,c,i,s)$, for any fixed choice
of the $P$-parameters and $i$,\\ let $s = \ceil
\genfrac{}{}{}{0}{m_P(i_f)}{m_P(j-i_f)}$.  Then for general $C$ and
$C'$, the type of the algebra so obtained is $u + s$.
\end{lem}
\begin{proof}
By Lemma \ref{best_s}, the specified value of $s$ yields an algebra
in the family.  To verify the formula for the type, by Lemma
\ref{type_wv} we must verify that $s \le m_P(j)$. We have
\begin{equation*}
\ceil \genfrac{}{}{}{0}{m_P(i_f)}{m_P(j-i_f)} \le \ceil(m_P(i_f))=
m_P(i_f) \le m_P(j),
\end{equation*}
\noindent where the last inequality follows from Lemma \ref{h_inc}.
\end{proof}

Up to this point, we have stated results that emphasized the
similarities between the families.  Now, we focus on the particulars
of each family in turn.

\begin{thm}{} \label{f1}
In the family $F_1(a,i,s)$, we have
\begin{enumerate}
\item [(i)] For a fixed choice of $a$ and $i$,
the smallest possible type $t = 1 + s$ is achieved by taking $s =
\ceil \genfrac{}{}{}{0}{a(2i-a+9)}{a^2-3a+2}$.
\item[(ii)] For a fixed choice of $a$, the smallest possible
type $t = 1 + s$ is achieved by taking $i = 2a$, $s = \ceil
\genfrac{}{}{}{0}{a(3a+9)}{a^2-3a+2}$.  With these choices, the type
is greater than 5 for $a \le 20$, and the type is exactly 5 for $a
\ge 21$.
\item[(iii)] For any choice of $a$, $i$, and $s$, the type is at least 5.
\end{enumerate}
\begin{proof}
For (i), we combine the various lemmas in this section with Lemma
\ref{spec_3}.
\begin{align*}
s &= \ceil \genfrac{}{}{}{0}{m_P(i_f)}{m_P(j-i_f)}\\
 &= \ceil \genfrac{}{}{}{0}{m_P(i + 3)}{m_P(a - 3)}\\
&= \ceil \genfrac{}{}{}{0}{a[2(i+3)-a+3]/2}{a([2(a-3)-a+3]/2) + 1}\\
&= \ceil \genfrac{}{}{}{0}{a(2i-a+9)}{a(a-3) +2}\\
&= \ceil \genfrac{}{}{}{0}{a(2i-a+9)}{a^2-3a+2}.
\end{align*}

For (ii), we observe that $s$ is an increasing function of $i$, and
we substitute the smallest permissible value $i = 2a$ to obtain $t =
1 + s = 1 + \ceil \genfrac{}{}{}{0}{a(3a+9)}{a^2-3a+2}$. We observe
that the fraction $\alpha := \genfrac{}{}{}{0}{3a^2+9a}{a^2-3a+2}$
is a decreasing function of $a$ that is always strictly greater than
3. Evaluating for $a=20$ gives $\alpha =
\genfrac{}{}{}{0}{1380}{342}
> 4$, and for $a = 21$,
$\alpha = \genfrac{}{}{}{0}{1512}{380} < 4$.\\

Part (iii) follows immediately from part (ii).
\end{proof}
\end{thm}

\begin{thm}{} \label{f2}
In the family $F_2(a,i,s)$, we have
\begin{enumerate}
\item [(i)]For a fixed choice of $a$ and $i$, the
smallest possible type $t = 2 + s$ is achieved by taking $s
=\ceil\genfrac{}{}{}{0}{4a(2i-a+7)}{(a-1)(a-3)}$.
\item[(ii)]For a fixed choice of $a$, the smallest possible type $t = 2+s$
is achieved by taking  $i = M := \ceil\genfrac{}{}{}{0}{2a - 3 +
\sqrt{2a^2+8a+7}}{2}$, $s =
\ceil\genfrac{}{}{}{0}{4a(2M-a+7)}{(a-1)(a-3)}$. With these choices,
the type is 12 for $a= 205$, $a = 209$, and $a \ge 213$. The
corresponding values of $M$ are 350, 357, and 364.  For any other
value of $a$, the type is greater than 12.
\item[(iii)] For any choice of $a$, $i$, and $s$, the type is at least 12.
\end{enumerate}
\end{thm}
\begin{proof}
For (i), we combine the various lemmas in this section with Lemma
\ref{spec_3}. To evaluate $m_P((a-1)/2 - 2)$, we observe $(a-1)/2 -
2 \le a$.

\begin{align*}
s &= \ceil \genfrac{}{}{}{0}{m_P(i_f)}{m_P(j-i_f)}\\
 &= \ceil \genfrac{}{}{}{0}{m_P(i + 2)}{m_P((a-1)/2 - 2)}\\
&= \ceil \genfrac{}{}{}{0}{a[2(i+2)-a+3]/2}{\genfrac(){0cm}{0}{(a-1)/2 - 2+2}{2}}\\
&= \ceil \genfrac{}{}{}{0}{a[2i - a + 7]/2}{\genfrac(){0cm}{0}{(a-1)/2}{2}}\\
&=\ceil \genfrac{}{}{}{0}{a[2i - a + 7]/2}{[(a-1)/2][(a-1)/2-1]/2}\\
&= \ceil\genfrac{}{}{}{0}{a(2i-a+7)}{(a-1)(a-3)/4}\\
&=\ceil\genfrac{}{}{}{0}{4a(2i-a+7)}{(a-1)(a-3)}.
\end{align*}

For (ii), we observe that $s$ is a non-decreasing function of $i$,
and we substitute the smallest permissible value $i = M$ to obtain
\begin{align*}
s &= \ceil\genfrac{}{}{}{0}{4a(2M-a+7)}{(a-1)(a-3)} \\
&=\ceil\genfrac{}{}{}{0}{4a(2\ceil(\genfrac{}{}{}{0}{2a - 3 +
\sqrt{2a^2+8a+7}}{2})-a+7)}{(a-1)(a-3)}.
\end{align*}
Because of the effect of the inner ceiling function, we do not claim that $s$ is
non-increasing as a function of $a$. In fact, using a computer, we calculated $s$
for integer values of $a$ up to 220, and found that $s= 10$ for $a$ = 198, 202, 205,
206, 208, 209, and 210, and for all values of $a \ge 212$; and that $s > 10$ for all
other values of $a$.  However, the formula for $s$ is sandwiched between
\begin{equation*}
L(a) := \ceil\genfrac{}{}{}{0}{4a(2(\genfrac{}{}{}{0}{2a - 3 +
\sqrt{2a^2+8a+7}}{2})-a+7)}{(a-1)(a-3)}
\end{equation*}
\noindent and
\begin{equation*}
U(a) := \ceil\genfrac{}{}{}{0}{4a(2(\genfrac{}{}{}{0}{2a - 3 +
\sqrt{2a^2+8a+7}}{2}+1)-a+7)}{(a-1)(a-3)},
\end{equation*}
\noindent which are both nonincreasing as functions of $a$ and which
both approach the value $\ceil [4(2 + \sqrt{2} - 1)] =  \ceil [4 +
4\sqrt{2}] = 10$ as a limit for large values of $a$.  Since $L(220)
= U(220) = 10$, it must be that $s=10$ for $a \ge 220$.\\

Our construction of $F_2$ requires $a$ to be odd, so we have type $2
+ s = 12$ for $a = 205$ and 209, and for any odd $a \ge 213$. For $a
= 205$, $M = 350, j = 452$. For $a = 209$, $M = 357, j = 461$. For
$a = 213$, $M =
364, j = 470$.\\

Part (iii) follows immediately from part (ii).
\end{proof}

\begin{thm}{} \label{g1}
In the family $G_1(a,b,i,s)$, we have
\begin{enumerate}
\item [(i)] For a fixed choice of $a$, $b$, and $i$,
the smallest possible type $t = 1 + s$ is achieved by taking $s =
\ceil \genfrac{}{}{}{0}{2i-a-b+10}{2ab-a-b-2}$.
\item[(ii)] For a fixed choice of $a$ and $b$, the smallest possible
type $t = 1 + s$ is achieved by taking $i = ab + 1$, $s = \ceil
\genfrac{}{}{}{0}{2ab-a-b+12}{2ab-a-b-2}$.  With these choices, the
type is 3 for $(a,b)$ if and only if either $a \ge 3, b\ge4$ or $a
\ge 2, b \ge 6$; the lowest values of $(a,b)$ for which the type is
4 are $(2,4)$ and $(3,3)$.
\item[(iii)] For any choice of $a$, $b$, $i$, and $s$, the type is at least 3.
\end{enumerate}
\end{thm}
\begin{proof}
For (i), we combine the various lemmas in this section with Lemma
\ref{spec_4}.  To evaluate $m_P(ab-3)$, we observe $ab-3 \ge a+b -
3$.
\begin{align*}
s &= \ceil \genfrac{}{}{}{0}{m_P(i_f)}{m_P(j-i_f)}\\
 &= \ceil \genfrac{}{}{}{0}{m_P(i + 3)}{m_P(ab - 3)}\\
&= \ceil \genfrac{}{}{}{0}{ab[2(i+3)-a -b +4]/2}{ab[2(ab-3)-a-b+4]/2}\\
&= \ceil \genfrac{}{}{}{0}{2i-a-b+10}{2ab-a-b-2}.
\end{align*}

For (ii), we observe that $s$ is an increasing function of $i$, and
we substitute the smallest permissible value $i = ab+1$ to obtain $t
= 1 + s = 1 + \ceil \genfrac{}{}{}{0}{2ab-a-b+12}{2ab-a-b-2}$.  We
observe that the fraction $\alpha :=
\genfrac{}{}{}{0}{2ab-a-b+12}{2ab-a-b-2}$ is a decreasing function,
separately in $a$ and $b$, that is always strictly greater than 1.
Evaluating $\alpha$ for the relevant values of $(a,b)$, we have:
\begin{align*}
(a,b) = (2,3):&  \;\; \alpha = 19/5, \;\; 4 >\alpha> 3.\\
&\\
(a,b) = (2,4):&  \;\; \alpha = 22/8,\;\;  3 > \alpha> 2.\\
(a,b) = (2,5):&  \;\; \alpha =25/11, \;\;   3 > \alpha > 2.\\
(a,b) = (3,3):&  \;\; \alpha =24/10 , \;\;   3 > \alpha> 2.\\
&\\
(a,b) = (3,4):&  \;\; \alpha =29/15, \;\;   2 >\alpha > 1.\\
(a,b) = (2,6):&  \;\; \alpha =28/14, \;\;   2 = \alpha > 1.\\
\end{align*}

Part (iii) follows immediately from part (ii).
\end{proof}

\begin{thm}{} \label{g2}
In the family $G_2(a,b,i,s)$, we have

\begin{enumerate}
\item [(i)] For a fixed choice of $a$, $b$, and $i$,
the smallest possible type $t = 1 + s$ is achieved by taking
\begin{align*}
s &= \ceil\genfrac{}{}{}{0}{ab[2i-a-b+8]}{ab[ab-a-b] + 2}\; \; if
\; a \; = \; 2. \\
s &= \ceil\genfrac{}{}{}{0}{2i-a-b+8}{ab-a-b} \;  \; if\; a \;
> \; 2.
\end{align*}
\item[(ii)]
For a fixed choice of $a$ and $b$, the smallest possible type $t = 1
+ s$ is achieved by taking $i = ab/2 + 2$, in which case
\begin{align*}
s &= \ceil\genfrac{}{}{}{0}{ab[ab-a-b+12]}{ab[ab-a-b] + 2}\; \; if
\; a \; = \; 2. \\
s &= \ceil\genfrac{}{}{}{0}{ab-a-b+12}{ab-a-b} \; \; if \; a \;
> \; 2.
\end{align*}
\noindent With these choices, the type is 3 for $(a,b)$ if and only
if either $a \ge 2, b\ge14$ or $a \ge 3, b\ge8$ or $a \ge 4, b \ge
6$. The smallest values of $(a,b)$ for which the type is 4 are
(2,8), (3,6), and (4,4).
\item[(iii)] For any choice of $a$, $b$, and $i$, the type is at least 3.
\end{enumerate}
\end{thm}

\begin{proof}
For (i), we combine the various lemmas in this section with Lemma
\ref{spec_4}.  To evaluate $m_P(ab/2-2)$, we observe that if $a =
2$, $ab/2 - 2 = b-2 = a + b - 4$; but if $a > 2$, $ab/2 - 2 =
a(b-2)/2 + a -2 > (b-2) + a -2 = a + b - 4$.\\

If $a = 2$,
\begin{align*}
s &= \ceil \genfrac{}{}{}{0}{m_P(i_f)}{m_P(j-i_f)}\\
 &= \ceil \genfrac{}{}{}{0}{m_P(i + 2)}{m_P(ab/2 - 2)}\\
&= \ceil \genfrac{}{}{}{0}{ab[2(i+2)-a -b +4]/2}{ab[2(ab/2 -2)-a-b+4]/2 + 1}\\
&= \ceil \genfrac{}{}{}{0}{ab[2i-a -b +8]}{ab[ab-a-b] + 2}.
\end{align*}

If $a > 2$,
\begin{align*}
s &= \ceil \genfrac{}{}{}{0}{m_P(i_f)}{m_P(j-i_f)}\\
 &= \ceil \genfrac{}{}{}{0}{m_P(i + 2)}{m_P(ab/2 - 2)}\\
&= \ceil \genfrac{}{}{}{0}{ab[2(i+2)-a -b +4]/2}{ab[2(ab/2 -2)-a-b+4]/2}\\
&= \ceil \genfrac{}{}{}{0}{2i-a -b +8}{ab-a-b}.
\end{align*}

For (ii), we observe that $s$ is an increasing function of $i$, and
we substitute the smallest permissible value $i = ab/2 + 2$ to
obtain, for $a = 2$, $t = 1 + s = 1 +
\ceil\genfrac{}{}{}{0}{ab(ab-a-b+12)}{ab(ab-a-b) + 2}$; and for $a >
2$, $t = 1 + s = 1 + \ceil\genfrac{}{}{}{0}{ab-a-b+12}{ab-a-b}$. We
let $\alpha := \genfrac{}{}{}{0}{ab(ab-a-b+12)}{ab(ab-a-b) + 2}$ and
$\beta := \genfrac{}{}{}{0}{ab-a-b+12}{ab-a-b}$. We observe that,
for all values of $a$ and $b$, $\alpha > 1$ and $\beta > 1$.  In
addition, $\alpha$ and $\beta $ are both decreasing functions,
separately in $a$ and $b$. Evaluating $\alpha$ and $\beta$ for the
relevant values of $(a,b)$ (recalling that either $a$ or $b$ must be
even), we have:
\begin{align*}
&(a,b) = (2,6):  \;\; \alpha = 192/50, \;\; 4 > \alpha > 3.\\
&(a,b) = (2,7):  \;\; \alpha = 238/72,\;\;4 > \alpha > 3.\\
&(a,b) = (3,4):  \;\; \beta = 17/5, \;\;4 > \beta > 3.\\
&\\
&(a,b) = (2,8):  \;\; \alpha = 288/98,\;\;3 > \alpha > 2.\\
&(a,b) = (2,13):  \;\; \alpha = 598/288,\;\;3 > \alpha > 2.\\
&(a,b) = (3,6):  \;\; \beta = 21/9, \;\;3 > \beta  > 2.\\
&(a,b) = (4,4):  \;\; \beta =20/8 , \;\;3 > \beta > 2.\\
&(a,b) = (4,5):  \;\; \beta = 23/11, \;\;3 > \beta  > 2.\\
&\\
&(a,b) = (2,14):  \;\; \alpha = 672/338,\;\;2 > \alpha >1.\\
&(a,b) = (3,8):  \;\; \beta = 25/13, \;\;2 > \beta  >1.\\
&(a,b) = (4,6):  \;\; \beta = 26/14, \;\;2 > \beta >1.
\end{align*}

Part (iii) follows immediately from part (ii).
\end{proof}

\begin{thm}{} \label{h1}
In the family $H_1(a,b,c,i,s)$
\begin{enumerate}
\item [(i)] For a fixed choice of $a$, $b$, $c$, and $i$,
the smallest possible type $t = 1 + s$ is achieved by taking $s =
\ceil \genfrac{}{}{}{0}{2i-a-b-c+11}{2abc-a-b-c-1}.$
\item[(ii)] For a fixed choice of $a$, $b$, and $c$,  the smallest possible
type $t = 1 + s$ is achieved by taking $i = abc$, $s = \ceil
\genfrac{}{}{}{0}{2abc-a-b-c+11}{2abc-a-b-c-1}$.  With these
choices, the type is 3 unless $(a,b,c)=(2,2,2)$, in which case the
type is 4.
\item[(iii)] For any choice of $a$, $b$, $i$, and $s$, the type is at least 3.
\end{enumerate}
\end{thm}
\begin{proof}
For (i), we combine the various lemmas in this section with Lemma
\ref{r-2_fold}.  To evaluate $m_P(abc-3)$, we observe $abc-3 \ge a+b
+c- 3 = (a-1) + (b-1) + (c-1)$.
\begin{align*}
s &= \ceil \genfrac{}{}{}{0}{m_P(i_f)}{m_P(j-i_f)}\\
 &= \ceil \genfrac{}{}{}{0}{m_P(i + 3)}{m_P(abc - 3)}\\
&= \ceil \genfrac{}{}{}{0}{ab[2(i+3)-a -b -c+5]/2}{ab[2(abc-3)-a-b-c+5]/2}\\
&= \ceil \genfrac{}{}{}{0}{2i-a-b-c + 11}{2abc-a-b-c - 1}.
\end{align*}

For (ii), we observe that $s$ is an increasing function of $i$, and
we substitute the smallest permissible value $i = abc$ to obtain\\
$t = 1 + s = 1 + \ceil \genfrac{}{}{}{0}{2abc -a-b-c +
11}{2abc-a-b-c - 1}$.  We observe that the fraction\\
$\alpha := \genfrac{}{}{}{0}{2abc -a-b-c + 11}{2abc-a-b-c - 1}$ is a
decreasing function, separately in $a$, $b$, and $c$, that is always
strictly greater than 1. Evaluating for the relevant values of
$(a,b,c)$, we have:

\begin{align*}
(a,b,c) = (2,2,2):& \;\; \alpha = 21/9,\;\; 3 > \alpha > 2.\\
(a,b,c) = (2,2,3):&  \;\; \alpha = 28/16,\;\; 2 > \alpha > 1.\\
\end{align*}

Part (iii) follows immediately from part (ii).
\end{proof}

For the family $G_3(a,b,i,s)$, we do not attempt to state a theorem with closed-form
solutions, similar to those we have proved for the other five families.  The most
serious obstacle is finding a closed-form solution of \eqref{g3_ineq}, which is
cubic in the variable $i$. Instead, we content ourselves with demonstrating the
method that led to choosing the example
$G_3(4,4,8,7)$ above.\\

We begin with some suitable choice of $a$ and $b$, in this case $a =
b = 4$, in which case $m = 5$ because
\begin{equation*}
\genfrac(){0cm}{0}{5 + 1}{2} = 15 < ab=16 < 21 =
\genfrac(){0cm}{0}{5 + 2}{2}.
\end{equation*}
\noindent Then \eqref{g3_ineq} becomes
\begin{equation*}
(16i -32) + 56 \le (i+3)(i+2)(i+1)/6.
\end{equation*}
\noindent This is false for $i=7$ since $136 > 120$, but true for
$i=8$, since $152 \le 165$, so we must choose $i \ge 8$; and the
additional requirement that $i \ge a+ b -3 = 5$ imposes no further
condition. If we choose $i = 8$, with the intention of obtaining the
smallest possible type for this choice of $a$ and $b$, the condition
on $s$ (using Proposition \ref{spec_4}) is that
\begin{equation*}
s \ge \genfrac{}{}{}{0}{m_P(i + 2)}{m_P(m-2)} =
\genfrac{}{}{}{0}{m_P(10)}{m_P(3)}
=\genfrac{}{}{}{0}{16(20-4-4+4)/2}{\genfrac(){0cm}{0}{3+3}{3} }=
\genfrac{}{}{}{0}{128}{20} = 6.4,
\end{equation*}
\noindent so we must take $s \ge 7$.

\chapter{Further Remarks}

\section{For Which Codimensions and Types are Non-Unimodal Level Algebras
Possible?}

We now return to a question raised in Chapter 1, at which time we
were not yet ready to provide justification for the answer given:
for a specified codimension $r$ and type $t$, must level algebras
necessarily be unimodal?

\begin{prop} \label{arb_t}
In codimension $3$, there exist non-unimodal level algebras for any
type $5$ or greater.  In codimensions $4$ and $5$, there exist
non-unimodal level algebras for any type $3$ or greater.
\end{prop}
\begin{proof}
For codimension 3, we let $t \ge 5$ and describe a procedure for finding a member of
$F_1(a,i,s)$ of type $t$.  We choose $a \ge 21$ such that $m_P(3a) \ge t$, that is,
$\genfrac{}{}{}{0}{a(5a+3)}{2} \ge t$.  We apply Theorem \ref{non_uni} for $F_1$,
with $i = 2a$ and $s=t-1$. That is, we let $\mathcal{E}(C) := \langle
f_1(C),...,f_s(C)\rangle $, where $f_1(C),...,f_s(C)$ are general elements of
$\mathcal{W}_{\mathcal{M}_P(3a)} \subseteq \mathcal{D}_{3a}$, and we let
$\mathcal{F}(C') := \langle g_1(C')\rangle $, where $g_1(C')$ is a general element
of $\mathcal{D}_{3a}$.  To ensure that $A_{\mathcal{E}(C)\bigoplus\mathcal{F}(C')}$
is non-unimodal of type $1 + s = 1 + (t-1) = t$ for general $C$ and $C'$, we must
check that the parameters $a, i, s$ are permissible. It is immediate that $a \ge 4$
and $i \ge 2a$, and by Lemma \ref{best_s} we must check that $s \ge
m_P(2a+3)/m_P(a-3)$.  But, having assumed that $a \ge 21$, we know from Theorem
\ref{f1} that $4 \ge m_P(2a+3)/m_P(a-3)$, and we have also assumed that $s = :t-1
\ge 4.$  Thus the values of the parameters $a, i$, and $s$ are permissable, so
$A_{\mathcal{E}(C)\bigoplus\mathcal{F}(C')}$ is indeed a member of the
family $F_1$, hence non-unimodal by Theorem \ref{non_uni}.\\

Analogous constructions are available in codimensions 4 and 5. For codimension 4, we
let $t \ge 3$ and we find a member of $G_1(a,b,i,s)$ of given type $t$. We choose $b
\ge a \ge 4$ such that $m_P(2ab+1) \ge t$, that is, $\genfrac{}{}{}{0}{ab(4ab
-a-b+6)}{2} \ge t$, and we again apply Theorem \ref{non_uni}, this time for $G_1$,
with $i=ab+1$ and $s = t-1$. That is, we let $\mathcal{E}(C) := \langle
f_1(C),...,f_s(C)\rangle $, where $f_1(C),...,f_s(C)$ are general elements of
$\mathcal{W}_{\mathcal{M}_P(2ab+1)} \subseteq \mathcal{D}_{2ab+1}$, and we let
$\mathcal{F}(C') := \langle g_1(C')\rangle $, where $g_1(C')$ is a general element
of $k[x,y,z]_{2ab+1}\subseteq \mathcal{D}_{2ab+1}$. To ensure that
$A_{\mathcal{E}(C)\bigoplus\mathcal{F}(C')}$ is non-unimodal of type $1 + s = 1 +
(t-1) = t$ for general $C$ and $C'$, we must check that the parameters $a, b, i, s$
are permissible. It is immediate that $b \ge a \ge 2$ and $i \ge ab + 1$, and by
Lemma \ref{best_s} we must check that $s \ge m_P(ab+4)/m_P(ab-3)$.  But, having
assumed that $b \ge a \ge 4$, we know from Theorem \ref{g1} that $2 \ge
m_P(ab+4)/m_P(ab-3)$, and we have also assumed that $s := t-1 \ge 2.$  Thus the
values of the parameters $a, b, i$, and $s$ are permissable, so
$A_{\mathcal{E}(C)\bigoplus\mathcal{F}(C')}$ is indeed a member of the
family $G_1$, hence non-unimodal by Theorem \ref{non_uni}.\\

For codimension 5, we let $t \ge 3$ and we find a member of $H_1(a,b,c,i,s)$ of
given type $t$. We choose $ c \ge b \ge a \ge 3$ such that $m_P(2abc) \ge t$, that
is,\\ $\genfrac{}{}{}{0}{abc(4abc -a-b-c+5)}{2} \ge t$, and we again apply Theorem
\ref{non_uni}, this time for $H_1$, with $i=abc$ and $s = t-1$. That is, we let
$\mathcal{E}(C) := \langle f_1(C),...,f_s(C)\rangle $, where $f_1(C),...,f_s(C)$ are
general elements of $\mathcal{W}_{\mathcal{M}_P(2abc)} \subseteq
\mathcal{D}_{2abc}$, and we let $\mathcal{F}(C') := \langle g_1(C')\rangle $, where
$g_1(C')$ is a general element of $k[x,y,z]_{2abc}\subseteq \mathcal{D}_{2abc}$. To
ensure that $A_{\mathcal{E}(C)\bigoplus\mathcal{F}(C')}$ is non-unimodal of type $1
+ s = 1 + (t-1) = t$ for general $C$ and $C'$, we must check that the parameters $a,
b, c, i, s$ are permissible. It is immediate that $c \ge b \ge a \ge 2$ and $i \ge
abc$, and by Lemma \ref{best_s} we must check that $s \ge m_P(abc+3)/m_P(abc-3)$.
But, having assumed that $c \ge b \ge a \ge 3$, we know from Theorem \ref{h1} that
$2 \ge m_P(abc+3)/m_P(abc-3)$, and we have also assumed that $s := t-1 \ge 2.$ Thus
the values of the parameters $a, b, c, i$, and $s$ are permissable, so
$A_{\mathcal{E}(C)\bigoplus\mathcal{F}(C')}$ is indeed a member of the
family $H_1$, hence non-unimodal by Theorem \ref{non_uni}.\\
\end{proof}

\begin{prop} \label{ext_mech}
Let $\mathcal{W} = \langle f_1,...,f_t\rangle  \subseteq k[x_1,...,x_r]_j$ be a
nonzero vector subspace of dimension $t$. Define $\mathcal{V} = \langle f_1,...,f_t,
x_{r+1}^{j}\rangle \subseteq k[x_1,...,x_r,x_{r+1}]_j$. Then $A_{\mathcal{W}}$ is a
level algebra of codimension $r$, type $t$, and socle degree $j$; and
$A_{\mathcal{V}}$ is a level algebra of codimension $r+1$, type $t+1$, and  socle
degree $j$.  For $d = 1,...,j$, $h_{\mathcal{V}}(d) = h_{\mathcal{W}}(d) + 1$.
\end{prop}
\begin{proof}
By Theorem \ref{matlis_duality}, $A_{\mathcal{W}}$ and
$A_{\mathcal{V}}$ are level algebras of the stated codimension and
socle degree, and the types are, by construction, the dimensions
respectively of $\mathcal{W}$ and $\mathcal{V}$.  We write
\begin{align*}
\mathcal{V} &= \langle f_1,...,f_t, x_{r+1}^{j}\rangle \\
 &= \langle f_1,...,f_t\rangle \bigoplus
\langle x_{r+1}^{j}\rangle \\
 &= \mathcal{W} \bigoplus \langle x_{r+1}^{j}\rangle
\end{align*}
\noindent and observe that $R_{j-d}* \langle x_{r+1}^{j}\rangle  =
\langle x_{r+1}^{d}\rangle  \subseteq \mathcal{D}_d$ is a vector
space of dimension 1 whose intersection with $R_{j-d}*\mathcal{W}$
is $\{0\}$.  By Lemmas \ref{hilb_char} and \ref{splice_lemma},
\begin{align*}
h_{\mathcal{V}}(d) &= h_{\mathcal{W}}(d) + h_{\langle x_{r+1}^{j}\rangle }(d)\\
 &=h_{\mathcal{W}}(d) + 1.
\end{align*}
\end{proof}

\begin{prop} \label{extend}
If there exists a non-unimodal level algebra $A$ of codimension $r$,
type $t$, and socle degree $j$, then there exists a non-unimodal
level algebra $A'$ of codimension $r+1$, type $t+1$, and socle
degree $j$.
\end{prop}
\begin{proof}
This follows immediately from the construction of the previous
proposition. Writing $A = A_{\mathcal{W}}$, take $A' =
A_{\mathcal{V}}$.
\end{proof}

We next cite a result of D. Bernstein from \cite{BI1}, which we
restate in our own notation.
\begin{thm}{}
Let $r = 5, R = k[X_1,...,X_5], \mathcal{D}= k[x_1,...,x_5]$.
Consider the family of vector subspaces $\mathcal{W} = \langle x_4f
+ x_5g\rangle  \subseteq \mathcal{D}_{16}$, where $f, g \in
k[x_1,x_2,x_3]_{15}$. Then for general $f$ and $g$,
$A_{\mathcal{W}}$ is a level algebra of type 1 and socle degree 16,
with Hilbert function
(1,5,12,22,35,51,70,91,90,91,70,51,35,22,12,5,1).  That is,
$A_{\mathcal{W}}$ is a non-unimodal Gorenstein algebra..
\end{thm}
$\hspace{\fill} \square$

\begin{prop} \label{big_non}
Given integers $r \ge 5$ and $t \ge 1$, there exists a non-unimodal
level algebra of codimension $r$ and type $t$.
\end{prop}
\begin{proof}
By Proposition \ref{extend}, it is enough to demonstrate a
non-unimodal level algebra (a) when $r=5$, for any $t$, and (b) when
$t=1$, for any $r$.\\

By virtue of the Bernstein example and Proposition \ref{arb_t}, to establish (a) it
only remains to consider the case $r=5$, $t=2$. For this, we modify the Bernstein
example so that $\mathcal{V} = \langle x_4f + x_5g, x_5^{16}\rangle $.  Then, for $d
\ge 2$,
\begin{align*}
h_{\mathcal{V}}(d) &= h_{\langle x_4f +x_5g\rangle }(d) + h_{\langle x_5^{16}\rangle }(d)\\
 &=h_{\mathcal{W}}(d) + 1.
\end{align*}
\noindent for exactly the same reasons as in the proof of
Proposition \ref{ext_mech}.  So $A_{\mathcal{V}}$ is non-unimodal.\\

To establish (b) for codimension $r > 5$, we modify the Bernstein
example in a different way.  This time, we let $\mathcal{V} =
\langle x_4f + x_5g + x_6^{16} + ... + x_r^{16}\rangle $, and then
for $d = 2,...,j-1$,

\begin{align*}
h_{\mathcal{V}}(d) &= \dim_kR_e*\langle x_4f+ x_5g + x_6^{16} + ... + x_r^{16}\rangle \\
&=\dim_kR_e*\langle x_4f+ x_5g \rangle  + \dim_kR_e*\langle x_6^{16} + ... + x_r^{16}\rangle \\
 &=h_{\mathcal{W}}(d) + (r-5),
\end{align*}
\noindent once again by Lemma \ref{splice_lemma} and Lemma
\ref{hilb_char}. So, again, $A_{\mathcal{V}}$ is non-unimodal.
\end{proof}

We remark that Propositions \ref{arb_t} and \ref{big_non} were used
in Chapter 1 to list the codimensions and types for which
non-unimodal level algebras are known to exist.

\section{Minimal Socle Degree}

For non-unimodal level algebras of given codimension $r$ and type
$t$, we have no actual methods for determining what the lowest
possible socle degree $j$ might be.  We make several remarks about
interesting cases.\\

For $r=3, t=5$, the only known non-unimodals come from the family
$F_1(a,i,s)$.  For these, we see from Theorem \ref{f1} that it is
possible to achieve type 5 only for $a \ge 21$.  We have $j = i + a
\ge 2a + a = 3a$, so the smallest
known socle degree is 63.\\

For $r=4, t = 3$, we must consider the families $G_1(a,b,i,s)$ and $G_2(a,b,i,s)$.
(Logically, we ought also to consider $G_3(a,b,i,s)$, but this family is not known
to yield algebras of type 3). By arguments analogous to the one in the previous
paragraph, we see from Theorems \ref{g1} and \ref{g2} that we must check
$G_1(2,6,13,2)$, $G_1(3,4,13,2)$, $G_2(2,14,16,2)$, $G_2(3,8,14,2)$, and
$G_2(4,6,14,2)$, of respective socle degrees $j = 25,25,30,26,26$. Thus $j=25$
represents the smallest known socle degree,
arising from family $G_1$.\\

For $r=5, t=3$, the lowest known socle degree does not result from
family $H_1$, which can do no better than $H_1(2,2,3,12,2)$, of
socle degree 24.  Instead, we can modify the Bernstein example
$\mathcal{W} = \langle x_4f + x_5g\rangle $, discussed earlier, of a
non-unimodal Gorenstein of codimension 5.  We let $\mathcal{V} =
\langle x_4f + x_5g, x_5^{16}, x_4x_5^{15}\rangle $.  Then, for $d
\ge 2$,
\begin{align*}
h_{\mathcal{V}}(d) &= h_{\langle x_4f +x_5g\rangle }(d) + h_{\langle x_5^{16}, x_4x_5^{15}\rangle }(d)\\
 &=h_{\mathcal{W}}(d) + 2.
\end{align*}
\noindent So we can find socle degree $j = 16$.\\

For $r=5, t=4$, there are several possible sources to consider.
$H_1(2,2,2,8,3)$ gives $j=16$. We could try modifying a type-3
member of family $H_1$ by adding a generator, but the socle degree
would then be at least 24.  We could try using Proposition
\ref{ext_mech}, applied to a type-3 non-unimodal of codimension 4,
but the socle degree would be at least 25.  Finally, we could
consider another modification of the Bernstein example: we let
$\mathcal{V} = \langle x_4f + x_5g, x_5^{16}, x_4x_5^{15},
x_4^2x_5^{14}\rangle $. Then, for $d \ge 2$,
\begin{align*}
h_{\mathcal{V}}(d) &= h_{\langle x_4f +x_5g\rangle }(d) +
 h_{\langle x_5^{16}, x_4x_5^{15}, x_4^2x_5^{14}\rangle }(d)\\
 &=h_{\mathcal{W}}(d) + 3.
\end{align*}
\noindent This is another example of socle degree $j = 16$.  With
its single drop, this is certainly different from $H_1(2,2,2,8,3)$,
which has a
double drop.\\

\end{document}